\documentclass[11pt]{amsart}
\usepackage[top=1.5in, bottom=1.5in, left=1.4in, right=1.4in]{geometry} 
\usepackage{color}
\usepackage{graphicx}
\usepackage{amssymb}
\usepackage{epstopdf}
\usepackage{amsthm}
\usepackage{hyperref}
\usepackage[table]{xcolor}
\usepackage{array}
\usepackage{mathtools}
\usepackage{enumerate}
\usepackage{enumitem}
\usepackage{lscape}
\usepackage{verbatim}
\usepackage{color}
\usepackage{bbold}
\usepackage{multicol}

\allowdisplaybreaks

\usepackage[all]{xy}
\usepackage{graphicx}  
\usepackage{tikz}
\usetikzlibrary{decorations.pathreplacing}
\usepackage{tkz-graph}
\usetikzlibrary{arrows}
\usetikzlibrary{decorations.markings}
\usepackage{xcolor}
\usepackage{ytableau}
\ytableausetup{boxsize=1.5em}

\newtheorem{thm}{Theorem}[section]
\newtheorem{lemma}[thm]{Lemma}

\newtheorem{cor}[thm]{Corollary}
\newtheorem{prop}[thm]{Proposition}
\newtheorem{conj}[thm]{Conjecture}

\newtheorem{Definition}[thm]{Definition}
\newenvironment{definition}
  {\begin{Definition}\rm}{\end{Definition}}

\newtheorem{Example}[thm]{Example}
\newenvironment{example}
  {\begin{Example}\rm}{\end{Example}}
  
\newtheorem{Remark}[thm]{Remark}
\newenvironment{remark}
  {\begin{Remark}\rm}{\end{Remark}}
  
\numberwithin{equation}{section}

\usepackage{etoolbox}
\apptocmd{\sloppy}{\hbadness 10000\relax}{}{}

\def\SetFancyGraph {
	\SetVertexMath
	\GraphInit[vstyle=Art]
	\SetUpVertex[MinSize=2pt]
	\SetVertexLabel
	\tikzset{VertexStyle/.style = {shape = circle,shading = ball,ball color = black,inner sep = 1.5pt}}
	\SetUpEdge[color=black]
	\tikzset{->-/.style={decoration={ markings, mark=at position 0.8 with {\arrow{>}}},postaction={decorate}}}
	\tikzset{->--/.style={decoration={ markings, mark=at position 0.55 with {\arrow{>}}},postaction={decorate}}}
}

\title[The CDE property for minuscule lattices]{The CDE property for minuscule lattices}
\author{Sam Hopkins}

\begin{document}

\begin{abstract}
Reiner, Tenner, and Yong recently introduced the coincidental down-degree expectations (CDE) property for finite posets and showed that many nice posets are CDE. In this paper we further explore the CDE property, resolving a number of conjectures about CDE posets put forth by Reiner-Tenner-Yong.  A consequence of our work is the completion of a case-by-case proof that any minuscule lattice is CDE. We also explain two major applications of the study of CDE posets: formulas for certain classes of set-valued tableaux; and homomesy results for rowmotion and gyration acting on sets of order ideals.
\end{abstract}

\maketitle

Reiner, Tenner, and Yong~\cite{reiner2016poset} recently introduced the coincidental down-degree expectations (CDE) property for posets and observed that many posets familiar to, and beloved by, algebraic combinatorialists are CDE. In the present paper we further explore the CDE property. We will mostly be interested in the case of distributive lattices, in particular, in intervals of Young's lattice and the shifted version of Young's lattice. Our aim is to resolve a number of conjectures about CDE posets put forth by Reiner-Tenner-Yong. A consequence of our work is the completion of a case-by-case proof that the distributive lattice $J(P)$ associated to any minuscule poset~$P$ is CDE.

The best way to motivate the study of the CDE property is to give an example of the kind of combinatorial ``coincidence'' we are interested in. Here is such an example. Treat the set of lattice paths from the southwest corner to the northeast corner of a~$2 \times 2$ grid consisting of north and east steps as a discrete probability space. Recall that a standard Young tableau (SYT) of shape $2 \times 2$ is a filling of the boxes of this grid with the numbers $1,2,3,4$ so that the entries strictly increase in rows and columns. We say that a lattice path is \emph{compatible} with an SYT if all the entries to the northwest of the lattice path are less than all the entries to the southeast. For example, the following lattice path (drawn in red) is compatible with both SYTs of shape~$2\times 2$:
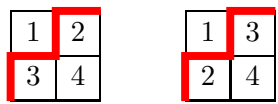
\begin{center}
\begin{tikzpicture}
\node at (0,0) {\begin{ytableau} 1 & 2 \\ 3 & 4 \end{ytableau}};
\def\x{0.6}
\draw[red,line width=3pt] (-1*\x,-1*\x) -- (-1*\x,0*\x) -- (0*\x,0*\x) -- (0*\x,1*\x) -- (1*\x,1*\x);
\end{tikzpicture} \qquad  \begin{tikzpicture}
\node at (0,0) {\begin{ytableau} 1 & 3 \\ 2 & 4 \end{ytableau}};
\def\x{0.6}
\draw[red,line width=3pt] (-1*\x,-1*\x) -- (-1*\x,0*\x) -- (0*\x,0*\x) -- (0*\x,1*\x) -- (1*\x,1*\x);
\end{tikzpicture}
\end{center}
while the following lattice path is compatible with only one of these SYTs:
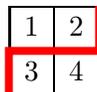
\begin{center}
\begin{tikzpicture}
\node at (0,0) {\begin{ytableau} 1 & 2 \\ 3 & 4 \end{ytableau}};
\def\x{0.6}
\draw[red,line width=3pt] (-1*\x,-1*\x) -- (-1*\x,0*\x) -- (1*\x,0*\x) -- (1*\x,1*\x);
\end{tikzpicture} 
\end{center}
Consider the probability distribution on our probability space in which each lattice path occurs with probability proportional to the number of SYTs with which it is compatible. These probabilities are recorded in the table in Figure~\ref{fig:introextable}, which also records the number of turns for each lattice path. One can easily check that the expected number of turns in a lattice path drawn with respect to this SYT-weighted distribution is $2$. Perhaps surprisingly, this $2$ is also the same as the expected number of turns if we had drawn with respect to the uniform distribution on lattice paths instead. This coincidence was first observed by Chan, L\'{o}pez Mart\'{i}n, Pflueger and Teixidor i Bigas~\cite[Remark 2.16]{chan2015genera}, who proved~\cite[Corollary 2.15]{chan2015genera} that for a general $a \times b$ grid the SYT-weighted and uniform distributions always give the same number of expected turns, namely $2\frac{ab}{a+b}$. Their motivation for considering this combinatorial problem came from algebraic geometry, more specifically, from Brill-Noether theory. Indeed, computing the expected number of turns with respect to this SYT-weighted distribution was the key combinatorial result needed to reprove a formula of Eisenbud-Harris~\cite{eisenbud1987kodaira} and Pirola~\cite{pirola1985chern} for the genus of Brill-Noether curves (i.e., Brill-Noether loci of dimension~$1$). Subsequently, Chan, Haddadan, Hopkins and Moci~\cite{chan2015expected} significantly generalized this combinatorial result by showing that the same coincidence of expected number of turns holds for a broader class of skew shapes than $a \times b$ rectangles (so-called ``balanced'' shapes), and for a broader class of probability distribution than the uniform and SYT-weighted distributions (so-called ``toggle-symmetric'' distributions).

\begin{figure}
\begin{tabular}{c | c | c | c | c | c | c }
\parbox{1in} {\begin{center} Lattice path\end{center}} & \begin{tikzpicture} \node at (0,0) {\begin{ytableau} \, & \, \\ \, & \, \end{ytableau}};
\def\x{0.6}
\draw[red,line width=3pt] (-1*\x,-1*\x) -- (-1*\x,1*\x) -- (1*\x,1*\x); \end{tikzpicture} & \begin{tikzpicture} \node at (0,0) {\begin{ytableau} \, & \, \\ \, & \, \end{ytableau}};
\def\x{0.6}
\draw[red,line width=3pt] (-1*\x,-1*\x) -- (-1*\x,0*\x) -- (0*\x,0*\x) -- (0*\x,1*\x) -- (1*\x,1*\x); \end{tikzpicture} & \begin{tikzpicture} \node at (0,0) {\begin{ytableau} \, & \, \\ \, & \, \end{ytableau}};
\def\x{0.6}
\draw[red,line width=3pt] (-1*\x,-1*\x) -- (-1*\x,0*\x) -- (1*\x,0*\x)  -- (1*\x,1*\x); \end{tikzpicture} & \begin{tikzpicture} \node at (0,0) {\begin{ytableau} \, & \, \\ \, & \, \end{ytableau}};
\def\x{0.6}
\draw[red,line width=3pt] (-1*\x,-1*\x) -- (0*\x,-1*\x) -- (0*\x,1*\x)  -- (1*\x,1*\x); \end{tikzpicture} & \begin{tikzpicture} \node at (0,0) {\begin{ytableau} \, & \, \\ \, & \, \end{ytableau}};
\def\x{0.6}
\draw[red,line width=3pt] (-1*\x,-1*\x) -- (0*\x,-1*\x) -- (0*\x,0*\x) -- (1*\x,0*\x) -- (1*\x,1*\x); \end{tikzpicture} & \begin{tikzpicture} \node at (0,0) {\begin{ytableau} \, & \, \\ \, & \, \end{ytableau}};
\def\x{0.6}
\draw[red,line width=3pt]  (-1*\x,-1*\x) -- (1*\x,-1*\x) -- (1*\x,1*\x); \end{tikzpicture}  \\ \hline
\parbox{1in} {\begin{center} SYT-weighted probability\end{center}}  & \Large $\frac{2}{10}$ & \Large $\frac{2}{10}$ & \Large $\frac{1}{10}$ & \Large $\frac{1}{10}$ & \Large $\frac{2}{10}$ & \Large $\frac{2}{10}$  \\ \hline
\parbox{1in} {\begin{center}Number of turns\end{center}} & $1$ & $3$ & $2$ & $2$ & $3$ & $1$
\end{tabular}
\caption{Table of probabilities according to SYT-weighted distribution and number of turns for lattice paths in a $2\times 2$ grid.} \label{fig:introextable}
\end{figure}

Meanwhile, Reiner-Tenner-Yong offered a poset-theoretic perspective on this expected number of turns coincidence. The set of lattice paths in a grid can be given the structure of a poset by declaring one lattice path to be greater than another if the first is weakly southwest of the second. In fact, these lattice paths correspond to order ideals in the product of two chain posets $\mathbf{a}\times\mathbf{b}$ and this poset of lattice paths is precisely the distributive lattice $J(\mathbf{a}\times\mathbf{b})$ of such order ideals.\footnote{Please forgive the unfortunate collision of terminology: the ``lattice'' in ``lattice path'' is totally unrelated to the ``lattice'' in ``distributive lattice.''} (An order ideal of a poset is a downwards-closed subset of elements; a distributive lattice $J(P)$ is the set of order ideals of some other poset $P$ partially ordered by containment.) Figure~\ref{fig:introexposet} depicts the Hasse diagram of this poset for the~$2 \times 2$ grid. Observe that the number of turns of a lattice path is its degree in this Hasse diagram. Also note that the SYT-weighted distribution on lattice paths has an alternate description in terms of this poset: it is the distribution in which each element of the poset occurs with probability proportional to the number of maximal chains passing through that element. This poset-theoretic description of the expected number of turns coincidence for lattice paths can easily be generalized to other finite posets. Reiner-Tenner-Yong introduced precisely this generalization. Actually, for technical reasons they were interested not in the expected degree but rather in the expected \emph{down-degree}, that is, the number of downwards edges adjacent to an element in the Hasse diagram. (The down-degree of a poset element is the same as the number of elements it covers.) It is not hard to see, for instance because of the global $180^{\circ}$ rotational symmetry, that in the lattice path setting the expected down-degree will always be half the expected degree. At any rate, Reiner-Tenner-Yong defined~\cite[Definition 2.1]{reiner2016poset} an arbitrary finite poset to have \emph{coincidental down-degree expectations} (CDE) if the expected down-degree is the same for the uniform distribution as for the distribution that weights each element proportional to the number of maximal chains passing through that element. They also referred to the poset as \emph{being} CDE in this case.

\begin{figure}
\begin{tikzpicture}
\node (1) at (0,0) {\scalebox{0.75}{ \begin{tikzpicture} \node at (0,0) {\begin{ytableau} \, & \, \\ \, & \, \end{ytableau}};
\def\x{0.6}
\draw[red,line width=3pt] (-1*\x,-1*\x) -- (-1*\x,1*\x) -- (1*\x,1*\x); \end{tikzpicture}}};
\node (2) at (0,2) {\scalebox{0.75}{\begin{tikzpicture} \node at (0,0) {\begin{ytableau} \, & \, \\ \, & \, \end{ytableau}};
\def\x{0.6}
\draw[red,line width=3pt] (-1*\x,-1*\x) -- (-1*\x,0*\x) -- (0*\x,0*\x) -- (0*\x,1*\x) -- (1*\x,1*\x); \end{tikzpicture}}};
\node (3) at (-2,4) {\scalebox{0.75}{\begin{tikzpicture} \node at (0,0) {\begin{ytableau} \, & \, \\ \, & \, \end{ytableau}};
\def\x{0.6}
\draw[red,line width=3pt] (-1*\x,-1*\x) -- (-1*\x,0*\x) -- (1*\x,0*\x)  -- (1*\x,1*\x); \end{tikzpicture}}};
\node (4) at (2,4) {\scalebox{0.75}{\begin{tikzpicture} \node at (0,0) {\begin{ytableau} \, & \, \\ \, & \, \end{ytableau}};
\def\x{0.6}
\draw[red,line width=3pt] (-1*\x,-1*\x) -- (0*\x,-1*\x) -- (0*\x,1*\x)  -- (1*\x,1*\x); \end{tikzpicture}}};
\node (5) at (0,6) {\scalebox{0.75}{\begin{tikzpicture} \node at (0,0) {\begin{ytableau} \, & \, \\ \, & \, \end{ytableau}};
\def\x{0.6}
\draw[red,line width=3pt] (-1*\x,-1*\x) -- (0*\x,-1*\x) -- (0*\x,0*\x) -- (1*\x,0*\x) -- (1*\x,1*\x); \end{tikzpicture}}};
\node (6) at (0,8) {\scalebox{0.75}{\begin{tikzpicture} \node at (0,0) {\begin{ytableau} \, & \, \\ \, & \, \end{ytableau}};
\def\x{0.6}
\draw[red,line width=3pt]  (-1*\x,-1*\x) -- (1*\x,-1*\x) -- (1*\x,1*\x); \end{tikzpicture}}};
\draw (1) -- (2);
\draw (2) -- (3);
\draw (2) -- (4);
\draw (3) -- (5);
\draw (4) -- (5);
\draw (5) -- (6);
\end{tikzpicture}
\caption{The poset of lattice paths in a $2 \times 2$ grid.} \label{fig:introexposet}
\end{figure}
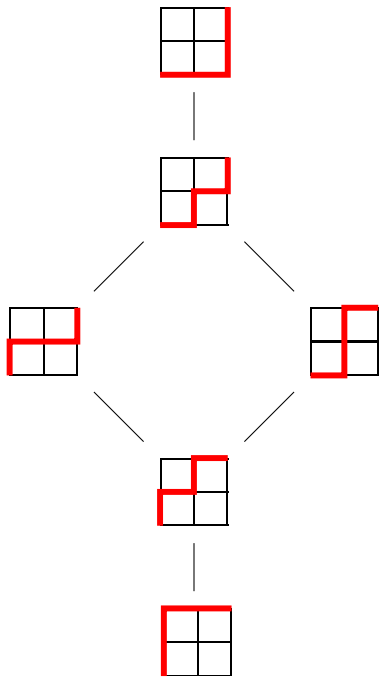

As mentioned, Reiner-Tenner-Yong observed that many nice posets (e.g., posets coming from algebra) are CDE. Beyond the distributive lattices associated to rectangles and other balanced skew shapes studied in~\cite{chan2015genera} and~\cite{chan2015expected} (these are certain intervals of Young's lattice of partitions), they showed that arbitrary products of chain posets are CDE~\cite[Corollary 2.19]{reiner2016poset}, that self-dual posets of constant Hasse diagram degree like weak Bruhat orders and Tamari lattices are CDE~\cite[Proposition 2.21]{reiner2016poset}, and that certain intervals of the weak Bruhat order on the symmetric group are CDE~\cite[Theorem~1.1]{reiner2016poset}. Moreover, they proved~\cite[Theorem 2.10]{reiner2016poset} that all connected minuscule posets are CDE. Minuscule posets are posets arising in the representation theory of semisimple Lie algebras that enjoy many remarkable combinatorial properties (see~\cite{proctor1984bruhat},~\cite{stembridge1994minuscule}, and~\cite{green2013combinatorics}). The minuscule posets have been classified and the proof that connected minuscule posets are CDE given by Reiner-Tenner-Yong, like some proofs we will present in Section~\ref{sec:minuscule} below, relies on this classification.

Reiner-Tenner-Yong also offered some tantalizing conjectures about CDE posets. Here are a few of their conjectures that will interest us in the present paper: they made two conjectures about CDE intervals of the shifted analog of Young's lattice~\cite[Conjectures~2.24 and~2.25]{reiner2016poset}, and they conjectured~\cite[Theorem 2.11]{reiner2016poset} that the distributive lattice of order ideals~$J(P)$ associated to an arbitrary minuscule poset $P$ is CDE. (Such distributive lattices are called minuscule lattices in~\cite{proctor1984bruhat} and we adopt that terminology here.) In this paper we prove one of the shifted Young's lattice conjectures and complete the case-by-case proof of the minuscule lattice conjecture.

Let us now outline the structure of the paper. In Section~\ref{sec:basics} we provide a formal definition of the CDE property. Reiner-Tenner-Yong in fact studied two closely related properties of posets, the CDE property and the mCDE property. The mCDE property requires that an infinite family of distributions on our poset, defined in terms of multichains rather than maximal chains, all yield the same expected down-degree. For graded posets, the mCDE property implies the CDE property. For us the mCDE property will be a strengthened version of the CDE property because whenever we show a poset is CDE, the poset in question will be graded and we will actually show that this poset is mCDE. In Section~\ref{subsec:defs} we formally define both the CDE and mCDE properties. In Section~\ref{subsec:remarks} we make some general remarks about these properties. In particular, answering a question of Reiner-Tenner-Yong~\cite[Question 2.20]{reiner2016poset}, we show in Section~\ref{subsec:remarks} that the mCDE property is preserved under direct products.

In Section~\ref{sec:distlat} we study the CDE and mCDE properties for distributive lattices. We show that in this case the distributions appearing in the definitions of the CDE and mCDE properties are toggle-symmetric in the sense of~\cite{chan2015expected}. (The word ``toggle'' comes from Striker and Williams~\cite{striker2012promotion}: for a given order ideal of a poset, \emph{toggling} an element of the poset is the act of adding or removing that element to the order ideal, so long as the resulting subset remains an order ideal; a distribution on the set of order ideals is \emph{toggle-symmetric} if for every element we are as likely to be able to toggle that element in as to toggle it out.) We therefore propose an even stronger property for distributive lattices which we call tCDE: a distributive lattice is tCDE if the expected down-degree is the same with respect to any toggle-symmetric distribution. In Sections~\ref{sec:younglat} and~\ref{sec:shiftyounglat} we describe several families of tCDE distributive lattices. Section~\ref{sec:younglat}, which concerns intervals of Young's lattice, is actually just a recap of the methods and results of~\cite{chan2015expected}. In Section~\ref{sec:shiftyounglat} we extend these ``toggle-theoretic'' methods to the shifted setting. In Section~\ref{subsec:shiftyounglatconj1} we apply these toggle methods to prove one of the shifted Young's lattice conjectures of Reiner-Tenner-Yong. In Section~\ref{subsec:shiftyounglatconj2} we discuss their second shifted Young's lattice conjecture: we explain how our toggle methods fail to prove this second conjecture, while we also affirmatively resolve the conjecture in a very special case.

In Section~\ref{sec:minuscule} we complete the proof that minuscule lattices are CDE. In fact, we show that such distributive lattices are tCDE. Most of the work for this case-by-case proof is actually done in Sections~\ref{sec:younglat} and~\ref{sec:shiftyounglat}. In Section~\ref{sec:minuscule} we also discuss what a uniform (that is, classification-free) proof of this result would look like. Such a uniform proof seems highly conceivable in light of work of Rush and Shi~\cite{rush2013orbits} giving a Lie-theoretic interpretation of toggles acting on a minuscule lattice. A uniform proof would certainly be desirable in so much as it would hint at the algebraic significance of the CDE property. On the other hand, our elementary, classification-based approach has also been worthwhile because it lead us to find a lot of tCDE posets beyond the minuscule lattices. At the end of Section~\ref{sec:minuscule} we point out that the connected minuscule posets, which themselves are distributive lattices, are in fact also tCDE.

In the final two sections we discuss some applications of our results. In Section~\ref{sec:tableaux} we explain how the study of the CDE property for intervals of Young's lattice and the shifted Young's lattice leads to formulas enumerating certain classes of set-valued tableaux. As Reiner-Tenner-Yong pointed out~\cite[Corollary 3.7]{reiner2016poset} (and as was implicitly noted in~\cite{chan2015genera}), computing the expected down-degree for the maximal chain distribution in an interval $[\nu,\lambda]$ of Young's lattice is the same as counting the number of a certain kind of set-valued tableaux, which they called \emph{standard barely set-valued tableaux}, of shape $\lambda/\nu$. In turn, thanks to Buch's combinatorial definition of stable Grothendieck polynomials~\cite{buch2002littlewood}, this number of tableaux is also just (negative one times) a particular coefficient of the corresponding stable Grothendieck polynomial $G_{\lambda/\nu}(x_1,x_2,\ldots)$.  In Section~\ref{sec:tableaux} we also go through the shifted analog of this story:  computing the expected down-degree for the maximal chain distribution in an interval of the shifted Young's lattice is the same as counting a certain kind of shifted set-valued tableaux, which in turn is the same as extracting a particular coefficient of the Type~B/C analogs of stable Grothendieck polynomials introduced by Ikeda and Naruse~\cite{ikeda2013ktheoretic}. 

In Section~\ref{sec:homomesy} we explain the connection between the tCDE property for distributive lattices and the homomesy paradigm of Propp and Roby~\cite{propp2015homomesy}. A triple~$(\mathcal{S},\Phi,f)$, where~$\mathcal{S}$ is a finite set of combinatorial objects, $\Phi\colon \mathcal{S} \to \mathcal{S}$ is an invertible operator acting on $\mathcal{S}$, and $f\colon \mathcal{S} \to \mathbb{R}$ is a combinatorial statistic, is said to \emph{exhibit homomesy} if the average of $f$ along every $\Phi$-orbit is the same. One of the main examples~\cite[Theorem 27]{propp2015homomesy} of homomesy discovered by Propp and Roby takes $\mathcal{S} := J(\mathbf{a}\times\mathbf{b})$ to be the set of order ideals of the product of two chain posets, $f(I) := \#\mathrm{max}(I)$ to be the \emph{antichain cardinality statistic}, and $\Phi$ to be \emph{rowmotion}~\cite{striker2012promotion}, a certain natural and well-studied action on the set of order ideals of a poset. A recent result of Striker~\cite[Lemma 6.2]{striker2015toggle} says that any distribution that is uniform on a rowmotion orbit is always toggle-symmetric for any poset. This result of Striker means that whenever we show that a distributive lattice~$J(P)$ is tCDE we can conclude that the antichain cardinality statistic is homomesic with respect to rowmotion acting on~$J(P)$. In particular we recapture a recent result (proved uniformly) of Rush and Wang~\cite[Theorem 1.4]{rush2015orbits} saying that the antichain cardinality statistic is homomesic with respect to rowmotion acting on a minuscule lattice. Striker~\cite[Theorem 6.7]{striker2015toggle} also showed that any distribution that is uniform on a gyration orbit is always toggle-symmetric for any ranked poset; here \emph{gyration}~\cite{wieland2000large} is another action on the set of order ideals of a poset. We explain how Striker's argument shows that the same is true for any ``rank-permuted'' version of rowmotion and thus conclude that the antichain cardinality statistic is homomesic for these rank-permuted rowmotions acting on~$J(P)$ whenever~$J(P)$ is tCDE. In this way we find a lot of new instances of homomesy.

\medskip

\noindent {\bf Acknowledgements}: We thank Thomas McConville and Vic Reiner for helpful conversations related to this work. The author was supported by NSF grant \#1122374.

\section{Basics on CDE posets} \label{sec:basics}

Reiner-Tenner-Yong~\cite{reiner2016poset} studied two closely related properties (CDE and mCDE) of finite posets. In this section we formally define these CDE and mCDE properties, we show that our definitions agree with those of~\cite{reiner2016poset}, and we make some general remarks about these properties, including a discussion of how they behave under various poset operations. We use the notation~$[n] := \{1,2,\ldots,n\}$ and $\binom{X}{k} := \{S \subseteq X\colon \#S = k\}$ throughout.

\subsection{CDE definitions} \label{subsec:defs}

All posets considered in this paper are finite unless stated otherwise. We use standard notation for posets (as laid out in e.g.~\cite[\S3]{stanley1996ec1}) which we nevertheless now briefly review. Let $P= (P,\leq)$ be a poset. We sometimes write~$\leq_P$ for clarity when there could be multiple partial orders in play. We write~$p < q$ to mean~$p \leq q$ but $p \neq q$; we write~$p \lessdot q$ to mean that~$q$ \emph{covers}~$p$, i.e.,~that $p < q$ and there is no $p' \in P$ with~$p < p' < q$. The reflected symbols~$\geq$, $>$, and~$\gtrdot$ have their obvious meanings. For a subset $S \subseteq P$, we use $\mathrm{max}(S)$ to denote the set of maximal elements in $S$ and $\mathrm{min}(S)$ the set of minimal elements. For $p \leq q$ the \emph{(closed) interval}~$[p,q]$ of $P$ is the induced subposet of all~$p' \in P$ with~$p \leq p' \leq q$. We represent~$P$ by its~\emph{Hasse diagram}, which is the directed graph with vertices elements of $P$ and directed edges~$(p,q)$ for $p \lessdot q$. We draw the Hasse diagram of $P$ in the plane with $q$ above~$p$ if~$(p,q)$ is an edge and thus we do not draw arrows on edges when depicting a Hasse diagram. We say $P$ is~\emph{connected} if the underlying undirected graph of its Hasse diagram is connected. We say that $P$ is \emph{ranked} if there exists a \emph{rank function}~$\mathrm{rk}\colon P \to \mathbb{N}$ satisfying the following:
\begin{itemize}
\item $0$ is in the image of $\mathrm{rk}$;
\item for $p,q \in P$ with~$p \lessdot q$ we have $\mathrm{rk}(q)=\mathrm{rk}(p) + 1 $.
\end{itemize}
For a connected poset, a rank function is unique if it exists. If $P$ is ranked with $r:=\mathrm{max}_{p\in P}\{\mathrm{rk}(p)\}$ we use $P_i$ for $i = 0,1,\ldots,r$ to denote the subset of elements of $P$ of rank $i$; these $P_i$ are called the \emph{ranks} of $P$. Similarly, we use $P_{\leq i} := \cup_{j=0}^{i} P_i$ to denote the subset of elements of $P$ of rank less than or equal to $i$.

An \emph{antichain} of $P$ is a subset of elements of $P$ which are pairwise incomparable. We denote the set of antichains of $P$ by $A(P)$. A \emph{chain of~$P$ of length~$k$} (or a~\emph{$k$-chain} of $P$) is a strictly increasing sequence $c = c_0 < c_1 < \cdots < c_k$ of elements of $P$. We write $p \in c$ to mean that~$p = c_i$ for some $i$. We say that a chain~$c$ of~$P$ is~\emph{maximal} if~$c$ is not a strict subsequence of any other chain of~$P$. We say $P$ is~\emph{graded of rank $r$} if all its maximal chains have the same length $r$. A graded poset is always ranked, but not conversely. A \emph{multichain of~$P$ of length~$m$} (or an~\emph{$m$-multichain} of~$P$) is a weakly increasing sequence~$c = c_0 \leq c_1 \leq \cdots \leq c_m$ of elements of $P$. Again we write~$p \in c$ to mean $p = c_i$ for some~$i$.  

An important class of graded posets are the distributive lattices, whose structure we now recall. An~\emph{order ideal} of a poset $P$ is a a subset $I \subseteq P$ such that if $y \in I$ and~$x \in P$ with $x \leq y$ then $x \in I$. The set of all order ideals of $P$ is denoted $J(P)$. Observe that~$J(P)$ is in bijection with $A(P)$ via the map $I \mapsto \mathrm{max}(I)$ that sends an order ideal to its antichain of maximal elements. The set of order ideals~$J(P)$ is naturally a poset with the partial order of containment. A~\emph{distributive lattice} is a poset of the form $J(P)$ for some (in fact, unique) poset $P$. (There is another algebraic definition of distributive lattice that we will not need here; this characterization of finite distributive in terms of order ideals is Birkhoff's \emph{fundamental theorem for finite distributive lattices}. We can recover $P$ from $J(P)$ as the induced subposet of ``join-irreducible elements.'' See~\cite[\S3.4]{stanley1996ec1} for more details.) Whether or not $P$ is graded, the distributive lattice $J(P)$ is always graded of rank $\#P$ with rank function~$\mathrm{rk}(I) := \#I$. Distributive lattices are also always connected. We use $\mathbf{a}$ to denote the \emph{chain poset} (i.e. the totally ordered poset) on $a$ elements. Observe that $J(\mathbf{a})=\mathbf{b}$ where $b=a+1$. Thus for all $a\geq 1$ the chain poset $\mathbf{a}$ is a distributive lattice.

We will consider various probability distributions on posets. Let $\mu$ be a probability distribution on a poset $P$. When it is not clear from context we will write $\mu_P$ to stress that $\mu$ is a distribution on $P$. We use some slightly nonstandard probability notation. If~$\mathbf{P}(X)$ is a proposition about a variable $X$ we use $\mathbb{P}(X\sim\mu; \mathbf{P}(X))$ to denote the probability that $\mathbf{P}(X)$ is true given that $X \sim \mu$, i.e., that~$X$ is distributed like $\mu$. For~$p\in P$ we use~$\mathbb{P}(\mu;p)$ as shorthand for $\mathbb{P}(X\sim\mu;X=p)$. We often identify the distribution $\mu$ on~$P$ with the formal sum $\sum_{p\in P} \mathbb{P}(\mu;p) [p]$ so that we may formally scale and add distributions. For any statistic (that is, arbitrary function)~$f\colon P \to \mathbb{R}$ we use~$\mathbb{E}(\mu;f) := \langle \sum_{p \in P} f(p)[p], \mu \rangle = \sum_{p \in P} f(p)\cdot\mathbb{P}(\mu;p)$ to denote the expectation of~$f(X)$ given that $X \sim \mu$. We use $\mathbb{1}\colon P \to \mathbb{R}$ for the statistic that is constantly $1$ so that $\mathbb{E}(\mu;\mathbb{1}) = 1$ for any distribution~$\mu$. We are now ready to define the CDE and mCDE properties. 

\begin{definition}  \label{def:cde} ({See Reiner-Tenner-Yong~\cite[Definition 2.1]{reiner2016poset}}).
Let $P$ be a poset. The \emph{down-degree} $\mathrm{ddeg}(p)$ of~$p \in P$ is the indegree of~$p$ in the Hasse diagram of $P$; that is, $\mathrm{ddeg}(p) := \#\{q \in P\colon q \lessdot p\}$ is the number of elements that $p$ covers. Let $\mathrm{uni}$ denote the uniform distribution on $P$. Let $\mathrm{maxchain}$ denote the distribution on $P$ in which each $p \in P$ occurs with probability proportional to the number of maximal chains that pass through $p$; that is, $\mathrm{maxchain}$ is defined by
\[\mathbb{P}(\mathrm{maxchain};p) := \frac{\#\{c\colon\textrm{$c$ is a maximal chain of $P$ and $p \in c$}\}}{\#\{(c,q)\colon\textrm{$c$ is a maximal chain of $P$ and $q \in c$}\}}\]
for all $p \in P$. We say that $P$ has \emph{coincidental down-degree expectations} (\emph{CDE}) if~$\mathbb{E}(\mathrm{maxchain};\mathrm{ddeg})=\mathbb{E}(\mathrm{uni};\mathrm{ddeg})$. In this case we also refer to~$P$ as \emph{being} CDE.
\end{definition}

Note that the quantity $\mathbb{E}(\mathrm{uni}_P;\mathrm{ddeg})$ is equal to the number of edges of the Hasse diagram of $P$ divided by its number of vertices. Therefore we also refer to $\mathbb{E}(\mathrm{uni}_P;\mathrm{ddeg})$ as the \emph{edge density} of the poset $P$.

\begin{definition} \label{def:mcde} ({See Reiner-Tenner-Yong~\cite[Definition 2.3]{reiner2016poset}}).
Let $P$ be a poset. Suppose that the longest chains of $P$ have length~$r$. Then for $k=0,1,\ldots,r$ define the \emph{$k$-chain distribution} $\mathrm{chain}(k)$ on $P$ by
\[\mathbb{P}(\mathrm{chain}(k); p) := \frac{\#\{c\colon c = c_0 < c_1 < \cdots < c_k\textrm{ is a $k$-chain of $P$ and $p \in c$}\}}{(k+1)\cdot\#\{c\colon \textrm{$c$ is a $k$-chain of $P$}\}}\]
for all $p \in P$. In other words, in $\mathrm{chain}(k)$ each $p \in P$ occurs with probability proportional to the number of $k$-chains passing through $p$. Observe that $\mathrm{chain}(0) = \mathrm{uni}$. Also observe that if $P$ is graded then $\mathrm{chain}(r) = \mathrm{maxchain}$. We say that $P$ is \emph{mCDE} if $\mathbb{E}(\mathrm{chain}(k);\mathrm{ddeg})=\mathbb{E}(\mathrm{uni};\mathrm{ddeg})$ for~$k = 0,1,\ldots,r$.
\end{definition}

The reader may be wondering why there is an~``m'' in the term ``mCDE.'' It stands for ``multichain.'' Actually, the definition~\cite[Definition 2.3]{reiner2016poset} of mCDE given by Reiner-Tenner-Yong in terms of multichains is not obviously the same as our Definition~\ref{def:mcde}. We now show that these definitions are indeed equivalent. To do that we need to talk about multichain-weighted distributions.

\begin{definition} 
Let $P$ be a poset. For $m \geq 0$, define the \emph{$m$-multichain distribution $\mathrm{mchain}(m)$} and \emph{modified $m$-multichain distribution $\widehat{\mathrm{mchain}}(m)$} on $P$ by
\begin{align*}
\mathbb{P}(\mathrm{mchain}(m); p) &:= \frac{\#\{c\colon c = c_0 \leq \cdots \leq c_m \textrm{ is an $m$-multichain of $P$ and $p\in c$}\}}{\#\{(c,q)\colon \textrm{$c$ is an $m$-multichain of $P$ and $q\in c$\}}}; \\
\mathbb{P}(\widehat{\mathrm{mchain}}(m); p) &:= \frac{\#\{(c,i)\colon c = c_0 \leq \cdots \leq c_m \textrm{ is an $m$-multichain of $P$ and $p=c_i$}\}}{(m+1)\cdot \#\{c\colon \textrm{$c$ is an $m$-multichain of $P$\}}}
\end{align*}
for all $p \in P$. In other words, in $\mathrm{mchain}(m)$ each $p \in P$ occurs with probability proportional to the number of $m$-multichains passing through $p$, while in $\widehat{\mathrm{mchain}}(m)$ each $p \in P$ occurs with probability proportional to the number of times $p$ appears in an $m$-multichain. Note that $\mathrm{mchain}(0) = \widehat{\mathrm{mchain}}(0) = \mathrm{uni}$.
\end{definition}

Reiner-Tenner-Yong~\cite[Definition 2.3]{reiner2016poset} defined a poset $P$ to be \emph{multichain-CDE} (\emph{mCDE}) if~$\mathbb{E}(\mathrm{mchain}(m);\mathrm{ddeg})=\mathbb{E}(\mathrm{uni};\mathrm{ddeg})$ for all $m \geq 0$. On the other hand, in the case where $P = J(Q)$ is a distributive lattice, the distribution~$\widehat{\mathrm{mchain}}(m)$ is the same as the ``weak distribution'' $\mu_{m+1,\leq}$ on $J(Q)$ defined and studied in~\cite{chan2015expected}. In fact, Chan, Haddadan, Hopkins and Moci~\cite[Definition 2.4]{chan2015expected} defined $\mu_{m+1,\leq}$ in terms of ``weak reverse $Q$-partitions'' but there is a well-known correspondence between these objects and multichains of~$J(Q)$ (see~\cite[Proposition 3.5.1]{stanley1996ec1}). At any rate, the next proposition says that we need not worry about the technical differences between these distributions when it comes to expected down-degree: the three distributions $\mathrm{chain}(k)$, $\mathrm{mchain}(m)$ and $\widehat{\mathrm{mchain}}(m)$ all lead to the same notion of strong~CDE.\footnote{The definition of mCDE given in~\cite{reiner2016poset} resulted from a miscommunication between the author and V.~Reiner. I meant to suggest the use of the distribution $\widehat{\mathrm{mchain}}(m)$ rather than $\mathrm{mchain}(m)$. It is a kind of happy accident that the definition given in~\cite{reiner2016poset} ended up being equivalent to my suggestion.}

\begin{prop} \label{prop:mcde}
Let $P$ be a poset. Suppose the longest chains of $P$ have length~$r$. Then the following are equivalent:
\begin{enumerate}
\item $\mathbb{E}(\mathrm{chain}(k);\mathrm{ddeg})=\mathbb{E}(\mathrm{uni};\mathrm{ddeg})$ for all~$k=0,1,\ldots,r$ (in other words, $P$ is mCDE according to Definition~\ref{def:mcde});
\item $\mathbb{E}(\mathrm{mchain}(m);\mathrm{ddeg})=\mathbb{E}(\mathrm{uni};\mathrm{ddeg})$ for all $m \geq 0$ (in other words, $P$ is mCDE according to~\cite[Definition 2.3]{reiner2016poset});
\item $\mathbb{E}(\mathrm{mchain}(m);\mathrm{ddeg})=\mathbb{E}(\mathrm{uni};\mathrm{ddeg})$ for all $m =0,1,\ldots,r$;
\item $\mathbb{E}(\widehat{\mathrm{mchain}}(m);\mathrm{ddeg})=\mathbb{E}(\mathrm{uni};\mathrm{ddeg})$ for all $m \geq 0$;
\item $\mathbb{E}(\widehat{\mathrm{mchain}}(m);\mathrm{ddeg})=\mathbb{E}(\mathrm{uni};\mathrm{ddeg})$ for all $m =0,1,\ldots,r$.
\end{enumerate}
\end{prop}

\begin{proof} 
Let $P$ be as in the statement of the proposition. Also, let $\mathrm{Comp}(n,k)$ denote the set of compositions of $n$ into~$k$ nonempty parts; that is,
\[ \mathrm{Comp}(n,k) := \{\alpha=(\alpha_1,\alpha_2,\ldots,\alpha_k)\colon \textrm{$\alpha_i \geq 1$ are integers and $\sum_{i=1}^{k}\alpha_i = n$}\}.\]
Recall that there is a well-known bijection
\begin{align*}
 \alpha\colon \binom{[n]}{k} &\xrightarrow{\sim} \mathrm{Comp}(n+1,k+1), \\
 S=\{s_1<s_2<\cdots<s_k\} &\mapsto (s_1-0,s_2-s_1,s_3-s_2,\ldots,s_{k}-s_{k+1},n+1-s_{k}).
 \end{align*} 
We proceed to prove all the necessary implications.

$(1) \Rightarrow (2)$: We claim that for all $m\geq0$,
\begin{equation} \label{eqn:chaintomchain}
\mathrm{mchain}(m) = \frac{\sum_{k=0}^{\mathrm{min}(m,r)} (k+1)\cdot\#\{c\colon \textrm{$c$ is a $k$-chain of $P$}\} \cdot \binom{m}{k} \cdot \mathrm{chain}(k)}{\sum_{k=0}^{\mathrm{min}(m,r)} (k+1)\cdot\#\{c\colon \textrm{$c$ is a $k$-chain of $P$}\} \cdot \binom{m}{k}}.
\end{equation}
To prove~(\ref{eqn:chaintomchain}), first let us show that the left- and right-hand sides of this equality of formal sums are \emph{proportional}, i.e., that $\mathrm{LHS} = \kappa \cdot \mathrm{RHS}$ for some scalar~$\kappa$. This amounts to showing that for any~$p,p' \in P$ we have
\[ \frac{\#\{c\colon \textrm{$c$ is an $m$-multichain, $p\in c$}\}}{\#\{c\colon \textrm{$c$ is an $m$-multichain, $p'\in c$}\}} = \frac{\sum_{k=0}^{m}\binom{m}{k}\cdot\#\{c\colon \textrm{$c$ is a $k$-chain, $p \in c$}\}}{\sum_{k=0}^{m}\binom{m}{k}\cdot\#\{c\colon \textrm{$c$ is a $k$-chain, $p' \in c$}\}}.\]
In fact we claim, what is stronger, that for any $p \in P$ we have
\[\#\{c\colon \textrm{$c$ is an $m$-multichain, $p\in c$}\} = \sum_{k=0}^{m}\binom{m}{k}\cdot\#\{c\colon \textrm{$c$ is a $k$-chain, $p \in c$}\}.\]
This last claim follows from the existence of the bijection
\[\varphi\colon\left\{(c,S)\colon \textrm{$c$ is a $k$-chain, $S \in \binom{[m]}{k}$ } \right\} \xrightarrow{\sim} \{c\colon \textrm{$c$ is an $m$-multichain}\}  \]
which sends the pair $(c,S)$ with $c=c_1<c_2<\cdots<c_k$ and $S \in \binom{[m]}{k}$ to the following $m$-multichain:
\[\scriptstyle \overbrace{c_0 = \cdots = c_0}^{\alpha(S)_1} < \overbrace{c_1= \cdots = c_1}^{\alpha(S)_2} < \overbrace{c_2 = \cdots = c_2}^{\alpha(S)_3} < \cdots < \overbrace{c_{k-1} = \cdots = c_{k-1}}^{\alpha(S)_{k}} < \overbrace{c_k = \cdots = c_k}^{\alpha(S)_{k+1}}. \]
In particular, note that $p \in \varphi(c,S) \Leftrightarrow p \in c$. Now having established that the left- and right-hand sides of~(\ref{eqn:chaintomchain}) are proportional, it follows by taking the expectation of $\mathbb{1}$ that they must in fact be equal. Taking the expectation  of $\mathrm{ddeg}$ on both sides of~(\ref{eqn:chaintomchain}) then yields~$\mathbb{E}(\mathrm{mchain}(m);\mathrm{ddeg}) = \mathbb{E}(\mathrm{uni};\mathrm{ddeg})$.

$(2) \Rightarrow (3)$: This is trivial.

$(3) \Rightarrow (1)$: The matrix representing the system of equations~(\ref{eqn:chaintomchain}) for $m=0,1,\ldots,r$ is upper-triangular with nonzero entries along the diagonal, hence invertible. This means that for~$k=0,1,\ldots,r$ we can write
\begin{equation} \label{eqn:mchaintochain}
\mathrm{chain}(k) = \sum_{m=0}^{k} \kappa_{k,m} \; \mathrm{mchain}(m)
\end{equation}
for some coefficients $\kappa_{k,m} \in \mathbb{R}$. By taking the expectation of $\mathbb{1}$ on both sides of~(\ref{eqn:mchaintochain}) we see that $\sum_{m=0}^{k}\kappa_{k,m} = 1$. Then taking the expectation of $\mathrm{ddeg}$ on both sides of~(\ref{eqn:mchaintochain}) gives $\mathbb{E}(\mathrm{chain}(k);\mathrm{ddeg}) = \mathbb{E}(\mathrm{uni};\mathrm{ddeg})$.

$(1) \Rightarrow (4)$: We will make use of the bijection $\varphi$ from $(1) \Rightarrow (2)$. We will also make use of the map $\zeta \colon \binom{[n]}{k} \to \binom{[n]}{k}$ defined by
\[\zeta(S) := \{m+1-s_k < m+1-s_k+s_1 < m+1-s_k+s_2 < \cdots < m+1-s_k + s_{k-1}\}\]
for $S = \{s_1 < s_2 <\cdots <s_k\}$. Observe that applying $\zeta$ to $S$ corresponds to rotating~$\alpha(S)$, i.e.~$\alpha(\zeta(S)) = (\alpha(S)_{k+1},\alpha(S)_{1},\alpha(S)_{2},\cdots,\alpha(S)_{k})$. Thus $\zeta^{(k+1)}(S) = S$. Let~$f(n,k)$ denote the number of $\zeta$-orbits of $\binom{[n]}{k}$. We claim that for all $m\geq0$,
\begin{equation} \label{eqn:chaintommchain}
\widehat{\mathrm{mchain}}(m) = \frac{\sum_{k=0}^{\mathrm{min}(m,r)} (k+1)\cdot\#\{c\colon \textrm{$c$ is a $k$-chain of $P$}\} \cdot f(m,k) \cdot \mathrm{chain}(k)}{\sum_{k=0}^{\mathrm{min}(m,r)} (k+1)\cdot \#\{c\colon \textrm{$c$ is a $k$-chain of $P$}\} \cdot f(m,k)\cdot \binom{m}{k}}.
\end{equation}
As in the proof of~(\ref{eqn:chaintomchain}), to prove~(\ref{eqn:chaintommchain}) it suffices to show that the left- and right-hand sides are proportional. This amounts to showing that for any $p,p' \in P$ we have 
\[ \frac{\#\{(c,i)\colon \textrm{$c$ is an $m$-multichain, $p = c_i$}\}}{\#\{(c,i)\colon \textrm{$c$ is an $m$-multichain, $p' = c_i$}\}} = \frac{\sum_{k=0}^{m}f(m,k) \cdot\#\{c\colon \textrm{$c$ is a $k$-chain, $p \in c$}\}}{\sum_{k=0}^{m}f(m,k) \cdot\#\{c\colon \textrm{$c$ is a $k$-chain, $p' \in c$}\}}.\]
This last equality holds because for any $p \in P$ we have
\begin{align*}
\#\left\{(c,i)\colon \parbox{1.4in}{\begin{center}$c$ is an $m$-multichain,\\ $p = c_i$\end{center}}\right\} &= \sum_{\substack{\textrm{$m$-multichain $c$}, \\ p \in c}} \#\{i\colon c_i = p\} \\
&= \sum_{k=0}^{m} \; \sum_{\substack{\textrm{$k$-chain $c$}, \\ p \in c}}  \sum_{S \in \binom{[m]}{k}} \#\{i\colon \varphi(c,S)_i = p\} \\
&= \sum_{k=0}^{m} \; \sum_{\substack{\textrm{$k$-chain $c$}, \\ p \in c}} \sum_{\substack{\textrm{$\mathcal{O}$ a $\zeta$-orbit}\\ \textrm{of $\binom{[m]}{k}$} }} \sum_{S \in \mathcal{O}} \hspace{-0.1cm} \#\{i\colon \varphi(c,S)_i = p\} \\
&=  \sum_{k=0}^{m} \; \sum_{\substack{\textrm{$k$-chain $c$}, \\ p \in c}} \sum_{\substack{\textrm{$\mathcal{O}$ a $\zeta$-orbit}\\ \textrm{of $\binom{[m]}{k}$} }} (m+1) \\
&= (m+1) \sum_{k=0}^{m} f(m,k)\cdot \#\{c\colon \textrm{$c$ is a $k$-chain, $p \in c$}\}
\end{align*}
In the above computation we used the fact that for any fixed $k$-chain $c$ with $p = c_t$ and any fixed $\zeta$-orbit $\mathcal{O}$ of $\binom{[m]}{k}$ we have
\begin{align*} 
 \sum_{S \in \mathcal{O}} \#\{i\colon \varphi(c,S)_i = p\}  &= \sum_{S \in \mathcal{O}} \alpha(S)_{t+1} \\
&=\alpha(S)_1 + \alpha(S)_2 + \cdots + \alpha(S)_{k+1} \textrm{ for a fixed $S \in \mathcal{O}$} \\
&= m+1.
\end{align*}
That~$\mathbb{E}(\widehat{\mathrm{mchain}}(m);\mathrm{ddeg}) = \mathbb{E}(\mathrm{uni};\mathrm{ddeg})$ then follows by taking the expectation  of $\mathrm{ddeg}$ on both sides of~(\ref{eqn:chaintommchain}).

$(4) \Rightarrow (5)$: This is trivial.

$(5) \Rightarrow (1)$: This is directly analogous to $(3) \Rightarrow (1)$. The matrix representing the system of equations~(\ref{eqn:chaintommchain}) for $m=0,1,\ldots,r$ is upper-triangular with nonzero entries along the diagonal, hence invertible. So for~$k=0,1,\ldots,r$ we can write
\begin{equation} \label{eqn:mmchaintochain}
\mathrm{chain}(k) = \sum_{m=0}^{k} \widehat{\kappa}_{k,m} \; \widehat{\mathrm{mchain}}(m)
\end{equation}
for some coefficients $\widehat{\kappa}_{k,m} \in \mathbb{R}$ with $\sum_{m=0}^{k}\widehat{\kappa}_{k,m}=1$. Taking the expectation of $\mathrm{ddeg}$ on both sides of~(\ref{eqn:mmchaintochain}) gives $\mathbb{E}(\mathrm{chain}(k);\mathrm{ddeg}) = \mathbb{E}(\mathrm{uni};\mathrm{ddeg})$.
\end{proof}

\subsection{General remarks and poset operations} \label{subsec:remarks}

Having established that we are working with the same properties studied in~\cite{reiner2016poset}, let us make some general remarks about these properties. First let us address the relationship between the CDE and mCDE properties. In general, there is no direct relationship between $P$ being CDE and $P$ being mCDE, as we can see from the following example.
\begin{example} \label{ex:cdevsmcde}
Let $P$ be the following poset with $5$ elements:
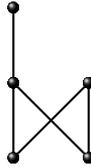
\begin{center}
 \begin{tikzpicture}
	\SetFancyGraph
	\Vertex[NoLabel,x=0,y=0]{v_1}
	\Vertex[NoLabel,x=1,y=0]{v_2}
	\Vertex[NoLabel,x=0,y=1]{v_3}
	\Vertex[NoLabel,x=1,y=1]{v_4}
	\Vertex[NoLabel,x=0,y=2]{v_5}
	\Edges[style={thick}](v_1,v_3)
	\Edges[style={thick}](v_1,v_4)
	\Edges[style={thick}](v_2,v_3)
	\Edges[style={thick}](v_2,v_4)
	\Edges[style={thick}](v_3,v_5)
\end{tikzpicture}
\end{center}
Then $P$ is CDE but not mCDE: we have $\mathbb{E}(\mathrm{uni}_P;\mathrm{ddeg})=\mathbb{E}(\mathrm{maxchain}_P;\mathrm{ddeg})=1$ but~$\mathbb{E}(\mathrm{chain}(1)_P;\mathrm{ddeg})=\frac{13}{14}$. Now let $P$ be the following poset with $8$ elements:
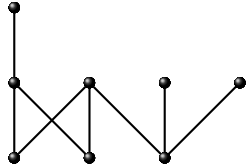
\begin{center}
 \begin{tikzpicture}
	\SetFancyGraph
	\Vertex[NoLabel,x=0,y=0]{v_1}
	\Vertex[NoLabel,x=1,y=0]{v_2}
	\Vertex[NoLabel,x=0,y=1]{v_3}
	\Vertex[NoLabel,x=1,y=1]{v_4}
	\Vertex[NoLabel,x=0,y=2]{v_5}
	\Vertex[NoLabel,x=2,y=0]{v_6}
	\Vertex[NoLabel,x=2,y=1]{v_7}
	\Vertex[NoLabel,x=3,y=1]{v_8}
	\Edges[style={thick}](v_1,v_3)
	\Edges[style={thick}](v_1,v_4)
	\Edges[style={thick}](v_2,v_3)
	\Edges[style={thick}](v_2,v_4)
	\Edges[style={thick}](v_3,v_5)
	\Edges[style={thick}](v_6,v_4)
	\Edges[style={thick}](v_6,v_7)
	\Edges[style={thick}](v_6,v_8)
\end{tikzpicture}
\end{center}
Then $P$ is mCDE but not CDE: we have $\mathbb{E}(\mathrm{uni}_P;\mathrm{ddeg})=\mathbb{E}(\mathrm{chain}(1)_P;\mathrm{ddeg})=\mathbb{E}(\mathrm{chain}(2)_P;\mathrm{ddeg})=1$ but~$\mathbb{E}(\mathrm{maxchain}_P;\mathrm{ddeg})=\frac{17}{16}$.
\end{example}
However, if $P$ is graded then it is immediate from Definitions~\ref{def:cde} and~\ref{def:mcde} that~$P$ being mCDE implies that $P$ is CDE. The converse implication need not hold even for graded posets: Reiner-Tenner-Yong~\cite[Remark 2.7]{reiner2016poset} gave an example of an interval~$[e,w]$ in the weak Bruhat order on the symmetric group which is CDE but not mCDE; the following example shows that distributive lattices can be CDE without being mCDE.
\begin{example}
Let $P$ be the following poset with $6$ elements:
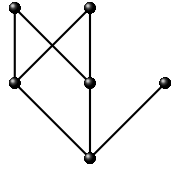
\begin{center}
 \begin{tikzpicture}
	\SetFancyGraph
	\Vertex[NoLabel,x=0,y=2]{v_3}
	\Vertex[NoLabel,x=0,y=1]{v_2}
	\Vertex[NoLabel,x=0,y=0]{v_1}
	\Vertex[NoLabel,x=-1,y=2]{v_6}
	\Vertex[NoLabel,x=-1,y=1]{v_5}
	\Vertex[NoLabel,x=1,y=1]{v_4}
	\Edges[style={thick}](v_1,v_2)
	\Edges[style={thick}](v_2,v_3)
	\Edges[style={thick}](v_1,v_5)
	\Edges[style={thick}](v_5,v_6)
	\Edges[style={thick}](v_1,v_4)
	\Edges[style={thick}](v_2,v_6)
	\Edges[style={thick}](v_5,v_3)
\end{tikzpicture}
\end{center}
Then the distributive lattice $J(P)$ of rank $6$ is CDE but is not mCDE: we have $\mathbb{E}(\mathrm{chain}(0)_{J(P)};\mathrm{ddeg}) = \mathbb{E}(\mathrm{chain}(6)_{J(P)};\mathrm{ddeg})=\frac{8}{5}$, but~$\mathbb{E}(\mathrm{chain}(1)_{J(P)};\mathrm{ddeg}) = \frac{83}{52}$.
\end{example}
Nevertheless we will see in Sections~\ref{sec:younglat} and~\ref{sec:shiftyounglat} below that for some nice classes of distributive lattices (initial intervals of Young's lattice and the shifted Young's lattice) the CDE and mCDE properties apparently coincide. See Remarks~\ref{rem:younglatcde} and~\ref{rem:shiftyounglatcde}.

Now let us address how these properties behave under poset operations. The most important unary poset operation is duality. The \emph{dual poset}~$P^*$ to a poset $P$ is the poset on the same elements as $P$ with $p \leq_{P^*} p'$ iff $p \geq_{P} p'$. As Reiner-Tenner-Yong noted~\cite[Example 2.15]{reiner2016poset}, the CDE property is not preserved under duality. The following example shows that, even for graded posets, neither the CDE nor mCDE property is preserved under duality.

\begin{example}
Let $P$ be the following graded poset of rank $3$ with $8$ elements:
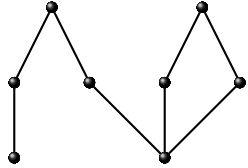
\begin{center}
 \begin{tikzpicture}
	\SetFancyGraph
	\Vertex[NoLabel,x=0,y=0]{v_1}
	\Vertex[NoLabel,x=2,y=0]{v_2}
	\Vertex[NoLabel,x=0,y=1]{v_3}
	\Vertex[NoLabel,x=1,y=1]{v_4}
	\Vertex[NoLabel,x=2,y=1]{v_5}
	\Vertex[NoLabel,x=3,y=1]{v_6}
	\Vertex[NoLabel,x=0.5,y=2]{v_7}
	\Vertex[NoLabel,x=2.5,y=2]{v_8}
	\Edges[style={thick}](v_1,v_3)
	\Edges[style={thick}](v_2,v_4)
	\Edges[style={thick}](v_2,v_5)
	\Edges[style={thick}](v_2,v_6)
	\Edges[style={thick}](v_3,v_7)
	\Edges[style={thick}](v_4,v_7)
	\Edges[style={thick}](v_5,v_8)
	\Edges[style={thick}](v_6,v_8)
\end{tikzpicture}
\end{center}
Then $\mathbb{E}(\mathrm{uni}_P;\mathrm{ddeg}) = \mathbb{E}(\mathrm{chain}(1)_P;\mathrm{ddeg}) = \mathbb{E}(\mathrm{chain}(2)_P;\mathrm{ddeg})=1$ and hence $P$ is mCDE and CDE. But $\mathbb{E}(\mathrm{uni}_{P^*};\mathrm{ddeg}) = 1 \neq \frac{7}{6} = \mathbb{E}(\mathrm{chain}(2)_{P^*};\mathrm{ddeg})$ so $P^*$ is neither mCDE nor CDE.
\end{example}

By way of contrast, Proposition~\ref{prop:tcdedual} in the next section will show that the CDE and mCDE properties are preserved under duality of distributive lattices.

The two most important binary poset operations are disjoint union and direct product. The~\emph{disjoint union} $P+Q$ of two posets $P$ and $Q$ is the poset whose underly set is the disjoint union~$P \sqcup Q$ with $x\leq_{P+Q} y$ iff $x,y \in P$ and $x \leq_P y$ or~$x,y \in Q$ and~$x \leq_Q y$. Every poset is the disjoint union of some number of connected posets which are called its \emph{connected components}. The~\emph{direct product} $P\times Q$ of two posets $P$ and $Q$ is the poset whose underlying set is the Cartesian product $P\times Q$ with~$(p,q) \leq_{P\times Q} (p',q')$ iff~$p \leq_P p'$ and~$q \leq_Q q'$. Note that if~$P$ is graded of rank~$r$ and $Q$ is graded of rank $s$ then $P\times Q$ is graded of rank~$r+s$.

Reiner-Tenner-Yong~\cite[Example 2.8]{reiner2016poset} observed that a disjoint union of any number of chain posets~$\mathbf{a_1} + \mathbf{a_2} + \cdots + \mathbf{a_k}$ is always CDE but in general is not mCDE unless $a_1 = a_2 = \cdots = a_k$. Moreover, they showed~\cite[Example 2.15]{reiner2016poset} that neither the CDE nor the mCDE property is preserved under disjoint unions, even for graded posets of the same rank. Nevertheless, we do at least have the following simple observation which says that the disjoint union of two~(m)CDE posets \emph{with the same edge densities} remains~(m)CDE.

\begin{prop}
Let $P$ and $Q$ be posets with $\mathbb{E}(\mathrm{uni}_P;\mathrm{ddeg}) = \mathbb{E}(\mathrm{uni}_Q;\mathrm{ddeg})$. If~$P$ and~$Q$ are both CDE then $P+Q$ is also CDE. Similarly, if~$P$ and $Q$ are both mCDE then $P+Q$ is also mCDE.
\end{prop}
\begin{proof}
Let $P$ and $Q$ be as in the statement of the proposition. Suppose the longest chains of $P$ have length~$r$ and the longest chains of $Q$ have length $s$. Since a (maximal) chain of $P+Q$ is nothing but either a (maximal) chain of~$P$ or a (maximal) chain of~$Q$, we have
\begin{align} \label{eqn:disjdist}
\mathrm{chain}_{P+Q}(k) &= \frac{x_k}{x_k+y_k} \mathrm{chain}_{P}(k) +  \frac{y_k}{x_k+y_k}\mathrm{chain}_Q(k), \; \; k=0,1,\ldots,\mathrm{max}(r,s); \\ \nonumber
\mathrm{maxchain}_{P+Q} &= \frac{x_{\mathrm{max}}}{x_{\mathrm{max}}+y_{\mathrm{max}}} \mathrm{maxchain}_{P} +  \frac{y_{\mathrm{max}}}{x_{\mathrm{max}}+y_{\mathrm{max}}}\mathrm{maxchain}_Q,
\end{align}
where
\begin{align*}
x_k &:= \#\{c\colon \textrm{$c$ is a $k$-chain of $P$}\}, \\
y_k &:= \#\{c\colon \textrm{$c$ is a $k$-chain of $Q$}\}, \\
x_{\mathrm{max}} &:= \#\{(c,p)\colon \textrm{$c$ is a maximal chain of $P$, $p \in c$}\}, \\
y_{\mathrm{max}} &:= \#\{(c,q)\colon \textrm{$c$ is a maximal chain of $Q$, $q \in c$}\}.
\end{align*}
The proposition follows by taking the expectation of $\mathrm{ddeg}$ on both sides of~(\ref{eqn:disjdist}).
\end{proof}

As for direct products, Reiner-Tenner-Yong proved~\cite[Corollary 2.19]{reiner2016poset} that a direct product of any number of chain posets~$\mathbf{a_1} \times \mathbf{a_2} \times \cdots \times \mathbf{a_k}$ is both CDE and mCDE. Moreover, they established the following.

\begin{prop}[{Reiner-Tenner-Yong~\cite[Proposition 2.13]{reiner2016poset}}] \label{prop:productcde}
Let $P$ and $Q$ be graded posets. If $P$ and $Q$ are both CDE then $P\times Q$ is also CDE.
\end{prop}

They showed~\cite[Example 2.14]{reiner2016poset} that the gradedness assumption in Proposition~\ref{prop:productcde} is essential. And they asked~\cite[Question 2.20]{reiner2016poset} if the mCDE property is also preserved under direct products, without the gradedness assumption. This is indeed the case.

\begin{prop} \label{prop:productmcde}
Let $P$ and $Q$ be posets. If $P$ and $Q$ are both mCDE then $P\times Q$ is also mCDE.
\end{prop}
\begin{proof} 
The proof is rather similar to the proof of Proposition~\ref{prop:mcde}. Accordingly, we will use some notation from the proof of Proposition~\ref{prop:mcde}: the bijection $\alpha$ between $\binom{[n]}{k}$ and $\mathrm{Comp}(n+1,k+1)$, the map $\zeta$ that acts on $\binom{[n]}{k}$ which corresponds to rotation of $\alpha(S)$, and the number $f(n,k)$ of $\zeta$-orbits.

Let $P$ and $Q$ be mCDE posets. Suppose the longest chains of $P$ have length~$r$ and the longest chains of $Q$ have length $s$. It is clear that the longest chains of~$P\times Q$ have length $r+s$. In fact, as Reiner-Tenner-Yong observed~\cite[Proposition 2.13]{reiner2016poset}, the chains of $P\times Q$ are merely ``shuffles'' of chains of $P$ and chains of $Q$. That is to say, we have a bijection
\[\phi\colon \left\{(c_P,c_Q,S)\colon \parbox{2.3in}{\begin{center}$c_P$ is an $i$-chain of $P$, \\ $c_Q$ is a $(k-i)$-chain of $Q$, $S \in \binom{[k]}{i}$\end{center}} \right\} \xrightarrow{\sim} \{c\colon \textrm{$c$ is a $k$-chain of $P\times Q$}\}\]
which sends the triple $(c_P,c_Q,S)$ with $c_P = p_0 < p_1 < \cdots < p_i$, $c_Q = q_0 < \cdots < q_{k-i}$ and $S \in \binom{[k]}{i}$ to the following $k$-chain of $P\times Q$:
\begin{gather*} 
\overbrace{(p_0,q_0) < (p_0,q_1) < \cdots < (p_0,q_{s_1-1})}^{\alpha(S)_1} < \overbrace{(p_1,q_{s_1-1}) < (p_1,s_{s_1}) <  \cdots < (p_1,q_{s_2-2})}^{\alpha(S)_2} \\
< \overbrace{(p_2,q_{s_2-2}) < \cdots < (p_2,q_{s_3-3})}^{\alpha(S)_3} < \cdots < \overbrace{(p_{i-1},q_{s_{i-1}-(i-1)}) < \cdots < (p_{i-1},q_{s_i-i})}^{\alpha(S)_i} \\
<  \overbrace{(p_i,q_{s_i-i}) < \cdots < (p_i,q_{k-i})}^{\alpha(S)_{i+1}}.
\end{gather*}

Let $\pi_P\colon P \times Q \to P$ and $\pi_Q \colon P \times Q \to P$ denote the projection maps. Note that the projection $\pi_P(\mu)$ of a probability distribution $\mu$ on $P\times Q$ is naturally a probability distribution on $P$. Also note that if $c$ is a chain of $P\times Q$, then the projection $\pi_P(c)$ is naturally a multichain of $P$.  For any $p \in P$, we have $p \in \pi_P(\phi(c_P,c_Q,S)) \Leftrightarrow p \in c_P$.

Observe that for any $(p,q) \in P\times Q$, $\mathrm{ddeg}(p,q) = \mathrm{ddeg}(\pi_P(p,q)) + \mathrm{ddeg}(\pi_Q(p,q))$. Thus by linearity of expectation and because of the symmetry between $P$ and $Q$ it suffices to show that 
\begin{equation} \label{eqn:projddeg}
\mathbb{E}(\pi_P(\mathrm{chain}(k)_{P\times Q});\mathrm{ddeg})=\mathbb{E}(\mathrm{uni}_P;\mathrm{ddeg})
\end{equation}
for all~$k =0,1,\ldots,r+s$. We claim that in fact for all~$k=0,1\ldots,r+s$,
\begin{equation} \label{eqn:projdist}
\pi_P(\mathrm{chain}(k)_{P\times Q}) = \frac{\sum_{i=\mathrm{max}(0,k-s)}^{\mathrm{min}(k,r)}x_{k,i}(i+1) \hspace{-0.05cm} \cdot \hspace{-0.05cm} \#\{c_P\colon\textrm{$c_Q$ is an $i$-chain of $P$}\}\cdot\mathrm{chain}(i)_P }{\sum_{i=\mathrm{max}(0,k-s)}^{\mathrm{min}(k,r)}x_{k,i}(i+1)\cdot\#\{c_P\colon\textrm{$c_P$ is an $i$-chain of $P$}\}}
\end{equation}
where
\[ x_{k,i} := f(k,i)\cdot\#\{c_Q\colon\textrm{$c_Q$ is a $(k-i)$-chain of $Q$}\}.\]
As in the proof of~(\ref{eqn:chaintomchain}), to prove~(\ref{eqn:projdist}) it suffices to show that the left- and right-hand sides are proportional. This amounts to showing that for $p,p' \in P$ we have 
\[ \frac{\#\left\{(c,j)\colon \parbox{1.3in}{\begin{center}\text{\small$c$ a $k$-chain of $P\times Q$}, \\ $\pi_P(c)_j = p$\end{center}}\right\}}{\#\left\{(c,j)\colon \parbox{1.3in}{\begin{center}\text{\small$c$ a $k$-chain of $P\times Q$}, \\ $\pi_P(c)_j = p'$\end{center}}\right\}} = \frac{\sum_{i=\mathrm{max}(0,k-s)}^{\mathrm{min}(k,r)}x_{k,i}\cdot\#\{c\colon \textrm{$c$ an $i$-chain of $P$, $p \in c$}\}}{\sum_{i=\mathrm{max}(0,k-s)}^{\mathrm{min}(k,r)}x_{k,i} \cdot\#\{c\colon \textrm{$c$ an $i$-chain of $P$, $p' \in c$}\}}\]
This last equality holds becase for any $p \in P$ we have
\begin{align*}
\#\left\{(c,j)\colon \parbox{0.71in}{\begin{center} $k$-chain $c$ \\ of $P\times Q$, \\ $\pi_P(c)_j = p$\end{center}}\right\} &= \sum_{\substack{\textrm{$k$-chain $c$ of $P\times Q$},\\ p\in \pi_P(c)}} \#\{j\colon \pi_P(c)_j = p\} \\
&= \sum_{i=0}^{k} \sum_{\substack{\textrm{$(k-i)$-chain $c_Q$}, \\ \textrm{$i$-chain $c_P$},\\p\in c_P }} \sum_{S \in \binom{[k]}{i}} \#\{j\colon \pi_P(\phi(c_P,c_Q,S))_j = p\} \\
&= \sum_{i=0}^{k} \sum_{\substack{\textrm{$(k-i)$-chain $c_Q$}, \\ \textrm{$i$-chain $c_P$},\\p\in c_P }}\sum_{\substack{\textrm{$\mathcal{O}$ a $\zeta$-orbit}\\\textrm{of $\binom{[k]}{i}$}}} \sum_{S \in \mathcal{O}} \#\{j\colon \pi_P(\phi(c_P,c_Q,S))_j = p\} \\
&=  \sum_{i=0}^{k} \sum_{\substack{\textrm{$(k-i)$-chain $c_Q$}, \\ \textrm{$i$-chain $c_P$},\\p\in c_P }}\sum_{\substack{\textrm{$\mathcal{O}$ a $\zeta$-orbit}\\\textrm{of $\binom{[k]}{i}$}}} (k+1) \\
&= (k+1) \sum_{i=\mathrm{max}(0,k-s)}^{\mathrm{min}(k,r)}x_{k,i}\cdot\#\{c\colon \textrm{$c$ an $i$-chain of $P$, $p \in c$}\}.
\end{align*}
In the above computation we used the fact that for any fixed $(k-i)$-chain $c_Q$ of $Q$, $i$-chain $c_P$ of $P$ with $p = c_t$, and $\zeta$-orbit $\mathcal{O}$ of $\binom{[k]}{i}$ we have
\begin{align*} 
 \sum_{S \in \mathcal{O}} \#\{j\colon \pi_P(\varphi(c_P,c_Q,S))_j = p\}  &= \sum_{S \in \mathcal{O}} \alpha(S)_{t+1} \\
&=\alpha(S)_1 + \alpha(S)_2 + \cdots + \alpha(S)_{i+1} \textrm{ for a fixed $S \in \mathcal{O}$} \\
&= k+1.
\end{align*}
Finally, taking the expectation of $\mathrm{ddeg}$ on both sides of~(\ref{eqn:projdist}) yields~(\ref{eqn:projddeg}).

Let us briefly remark that the above reasoning also leads to a proof Proposition~\ref{prop:productcde}. Suppose $P$ and $Q$ are graded. Since~$\mathrm{ddeg}(p,q) = \mathrm{ddeg}(\pi_P(p,q)) + \mathrm{ddeg}(\pi_Q(p,q))$, to establish the conclusion of Proposition~\ref{prop:productcde} it suffices by linearity of expectation to show
\begin{align} \label{eqn:projddeggraded}
\mathbb{E}(\pi_P(\mathrm{uni}_{P\times Q});\mathrm{ddeg})&=\mathbb{E}(\mathrm{uni}_P;\mathrm{ddeg}); \\  \nonumber
\mathbb{E}(\pi_Q(\mathrm{uni}_{P\times Q});\mathrm{ddeg}) &=\mathbb{E}(\mathrm{uni}_Q;\mathrm{ddeg});\\ \nonumber
\mathbb{E}(\pi_P(\mathrm{maxchain}_{P\times Q});\mathrm{ddeg}) &=\mathbb{E}(\mathrm{maxchain}_P;\mathrm{ddeg});\\ \nonumber
\mathbb{E}(\pi_Q(\mathrm{maxchain}_{P \times Q});\mathrm{ddeg}) &=\mathbb{E}(\mathrm{maxchain}_Q;\mathrm{ddeg}).
\end{align}
Taking $k=0$ in~(\ref{eqn:projdist}) yields $\pi_P(\mathrm{uni}_{P\times Q})  = \mathrm{uni}_P$, which anyways is obvious. By symmetry we also have $\pi_Q(\mathrm{uni}_{P\times Q})  = \mathrm{uni}_Q$. Meanwhile, taking $k=r+s$ in~(\ref{eqn:projdist}) yields~$\pi_P(\mathrm{chain}(r+s)_{P\times Q})  = \mathrm{chain}(r)_P$ and~$\pi_Q(\mathrm{chain}(r+s)_{P\times Q})  = \mathrm{chain}(s)_Q$ by symmetry. But since $P$ and $Q$ are graded we in fact have $\pi_P(\mathrm{maxchain}_{P\times Q}) = \mathrm{maxchain}_P$ and $\pi_Q(\mathrm{maxchain}_{P\times Q}) = \mathrm{maxchain}_Q$. The equations~(\ref{eqn:projddeggraded}) follow from the preceding equalities of distributions by taking the expectation of $\mathrm{ddeg}$. We should note, however, that the argument we gave above to establish~(\ref{eqn:projdist}) is basically the same the proof of Proposition~\ref{prop:productcde} given in~\cite[Proposition 2.13]{reiner2016poset}.
\end{proof}

\section{The CDE property for distributive lattices and toggle-symmetry} \label{sec:distlat}

Classifying CDE posets in general is surely an intractable problem. As mentioned in the introduction, we are mostly interested in understanding the CDE property for distributive lattices. Actually we will investigate the stronger mCDE property. (The mCDE property is stronger than the CDE property for distributive lattices because all distributive lattices are graded.) So throughout this section $P$ is some fixed poset and our aim is to understand the distributions $\mathrm{chain}(k)_{J(P)}$ on~$J(P)$ for $k=0,1,\ldots,\#P$. It turns out that these distributions belong to a family of distributions studied by Chan, Haddadan, Hopkins and Moci (hereafter referred to as CHHM) in~\cite{chan2015expected} called the ``toggle-symmetric'' distributions. We now review what these distributions are.

Let $I \in J(P)$ be an order ideal and $p \in P$ some element. We say that $p$ can be \emph{toggled into} $I$ if $p \notin I$ and $p$ is a minimal element of $P \setminus I$. We say that $p$ can be \emph{toggled out of} $I$ if $p \in I$ and $p$ is a maximal element of $I$. This terminology comes from the toggle group. The ``toggle group'' is the name given by Striker and Williams~\cite{striker2012promotion} to a group introduced by Cameron and Fon-der-Flaass~\cite{cameron1995orbits} in their study of a certain action on $J(P)$ now referred to as rowmotion. We will revisit rowmotion in Section~\ref{sec:homomesy}. For now let us just discuss toggles. We define the \emph{toggle at~$p$} $\tau_p\colon J(P) \to J(P)$ by
\[ \tau_p(I) := \begin{cases} I \cup \{p\} &\textrm{if $p$ can be toggled into $I$}, \\ I \setminus \{p\} &\textrm{if $p$ can be toggled out of $I$}, \\ $I$ &\textrm{otherwise}. \end{cases}\]  
The \emph{toggle group} is the subgroup of the permutation group $\mathfrak{S}_{J(P)}$ generated by the toggles $\tau_p$ for all $p\in P$. Granting that toggles and the toggle group are worth studying, it makes sense to consider the following toggleability statistics $\mathcal{T}^{+}_p, \mathcal{T}^{-}_p\colon J(P) \to J(P)$ for~$p \in P$:
\begin{align*}
\mathcal{T}^{+}_p &:= \begin{cases} 1 &\textrm{if $p$ can be toggled into $I$}, \\ 0 &\textrm{otherwise}; \end{cases} \\
\mathcal{T}^{-}_p &:= \begin{cases} 1 &\textrm{if $p$ can be toggled out of $I$}, \\ 0 &\textrm{otherwise}. \end{cases} \\
\end{align*}
The toggle-symmetric property is defined in terms of~$\mathcal{T}^{+}_p$ and~$\mathcal{T}^{-}_p$.

\begin{definition} \label{def:togsym} (CHHM~\cite[Definition 2.2]{chan2015expected}).
A distribution $\mu$ on $J(P)$ is called \emph{toggle-symmetric} if for each $p \in P$ the probability that $p$ can be toggled into $I$ is equal to the probability that it can be toggled out of $I$ when $I \sim \mu$. In other words, $\mu$ is toggle-symmetric if $\mathbb{E}(\mu;\mathcal{T}^{+}_p ) = \mathbb{E}(\mu;\mathcal{T}^{-}_p)$ for all $p \in P$.
\end{definition}

One reason to think toggle-symmetry could be related to the CDE property is because of the following simple observation: for~$I \in J(P)$ we have~$\mathrm{ddeg}(I) = \sum_{p \in P} \mathcal{T}^{-}_p(I)$. (This quantity is also equal to $\#\mathrm{max}(I)$.) CHHM~\cite[Definition 2.1]{chan2015expected} defined the \emph{jaggedness} $\mathrm{jag}(I)$ of an order ideal $I \in J(P)$ to be  $\mathrm{jag}(I) := \sum_{p \in P} \mathcal{T}^{+}_p(I) + \mathcal{T}^{-}_p(I)$. The jaggedness of $I \in J(P)$ is the same as its degree in the Hasse diagram of $J(P)$. CHHM were interested in computing the expected jaggedness for toggle-symmetric distributions on certain distributive lattices~$J(P)$. But knowing the expected jaggedness is the same as knowing the expected down-degree:  it follows immediately from the definition of toggle-symmetry and linearity of expectation that~$\mathbb{E}(\mu;\mathrm{ddeg}) = \frac{1}{2}\mathbb{E}(\mu;\mathrm{jag})$ for any toggle-symmetric distribution~$\mu$ on~$J(P)$.

CHHM~\cite[Corollary 3.8]{chan2015expected} showed that for some special distributive lattices~$J(P)$ the expected jaggedness $\mathbb{E}(\mu;\mathrm{jag})$ is the same for all toggle-symmetric distributions~$\mu$ on~$J(P)$. Such a result is very close to showing that these~$J(P)$ are mCDE. What still needs to be addressed is whether the relevant distributions (i.e.,~the distributions appearing in Proposition~\ref{prop:mcde}) are toggle-symmetric. In fact they are. CHHM showed that the modified $m$-multichain distributions on~$J(P)$ are toggle-symmetirc.

\begin{lemma}[{CHHM~\cite[Lemma 2.8]{chan2015expected}}] \label{lem:mmchaintogsym}
For any $m\geq 0$ the distribution $\widehat{\mathrm{mchain}}(m)$ on~$J(P)$ is toggle-symmetric.
\end{lemma}

Because the proof of Lemma~\ref{lem:mmchaintogsym} is so short, and because in~\cite[Lemma 2.8]{chan2015expected} the proof was presented in the slightly different language of $P$-partitions rather than multichains of~$J(P)$, we include a proof of this lemma here.

\begin{proof}[Proof of Lemma~\ref{lem:mmchaintogsym}]
Fix $m \geq 0$. Let $p \in P$. Set $\mathcal{T}_p := \mathcal{T}^{+}_p-\mathcal{T}^{-}_p$. We want to show
\[\mathbb{E}(\widehat{\mathrm{mchain}}(m)_{J(P)};\mathcal{T}^{+}_p) - \mathbb{E}(\widehat{\mathrm{mchain}}(m)_{J(P)};\mathcal{T}^{-}_p) = \mathbb{E}(\widehat{\mathrm{mchain}}(m)_{J(P)};\mathcal{T}_p) = 0.\] 
For an $m$-multichain $c = I_0 \subseteq I_1 \subseteq \cdots \subseteq I_m$ of $J(P)$, set $\mathcal{T}_p(c) := \sum_{i=0}^{m} \mathcal{T}_p(I_i)$. What we want to show is equivalent to $\sum_{c} \mathcal{T}_p(c) = 0$ where the sum is over all $m$-multichains of $J(P)$. To show that this sum is zero we will define an involution
\[\tau_p\colon \{c \colon \textrm{$c$ an $m$-multichain of $J(P)$}\} \to  \{c \colon \textrm{$c$ an $m$-multichain of $J(P)$}\}\]
satisfying $\mathcal{T}_p(c) = -\mathcal{T}_p(\tau_p(c))$. So let $c = I_0 \subseteq I_1 \subseteq \cdots \subseteq I_m$ be an $m$-multichain of~$J(P)$. If $\mathcal{T}_p(c) = 0$ then we set $\tau_p(c) := c$. Thus asssume $\mathcal{T}_p(c) \neq 0$. As we move up this multichain, the statistic $\mathcal{T}_p$ changes in the way depicted in the following table:
\begin{center}
\begin{tabular}{c | c | c | c | c | c | c | c | c | c | c | c | c | c | c }
$I$ & $I_0$ & $I_1$ & $\cdots$ & $I_{i-1}$ & $I_i$ & $I_{i+1}$ & $\cdots$ & $I_{j}$ & $I_{j+1}$ & $\cdots$ & $I_k$ & $I_{k+1}$ & $\cdots$ & $I_m$  \\ \hline
$\mathcal{T}_p(I)$ & $0$ & $0$ & $\cdots$ & $0$ & $1$ & $1$ & $\cdots$ & $1$ & $-1$ & $\cdots$ & $-1$ & $0$ & $\cdots$ & $0$
\end{tabular}
\end{center}
Let $i$, $j$, $k$ be as in that table. (So $i := \mathrm{min}\{t\colon \mathcal{T}_p(I_t) \neq 0\}$, $k := \mathrm{max}\{t\colon \mathcal{T}_p(I_t) \neq 0\}$, and $j := \mathrm{max}\{t\colon \mathcal{T}_p(I_t) = 1\}$ or $j := i-1$ if there is no $t$ for which $\mathcal{T}_p(I_t) = 1$.) Observe that $\mathcal{T}_p(c) = (j+1-i) - (k-j)$. Define a new $m$-multichain $c' := I'_0 \subseteq I'_1 \subseteq \cdots \subseteq I'_m$ of $J(P)$ by
\[ I'_t := \begin{cases}\tau_p(I_t) & t \in  \{j-\mathcal{T}_p(c)+1,j-\mathcal{T}_p(c)+2,\ldots,j\}, \mathcal{T}_p(c) > 0\\
\tau_p(I_t) & t \in  \{j+1,j+2,\ldots,j-\mathcal{T}_p(c)\}, \mathcal{T}_p(c) < 0\\
I_t &\textrm{otherwise}. \end{cases}\]
Set $\tau_p(c) := c'$. It is easy to see that $\mathcal{T}_p(c) = -\mathcal{T}_p(\tau_p(c))$ and that this $\tau_p$ indeed defines an involution on the set of $m$-multichains of $J(P)$.
\end{proof}

Meanwhile, Reiner-Tenner-Yong showed the $m$-multichain distributions on $J(P)$ are toggle-symmetric.

\begin{lemma}[{Reiner-Tenner-Yong~\cite[Proposition 2.5]{reiner2016poset}}] \label{lem:mchaintogsym}
For any $m\geq 0$ the distribution $\mathrm{mchain}(m)$ on $J(P)$ is toggle-symmetric.
\end{lemma}

And, as promised, the $k$-chain distributions on $J(P)$ are toggle-symmetric as well.

\begin{lemma} \label{lem:chaintogsym}
For any $k=0,1,\ldots,\#P$ the distribution $\mathrm{chain}(k)$ on $J(P)$ is toggle-symmetric.
\end{lemma}
\begin{proof}
This lemma is equivalent to Lemma~\ref{lem:mmchaintogsym} thanks to the equations~(\ref{eqn:chaintommchain}) and~(\ref{eqn:mmchaintochain}): any convex combination of toggle-symmetric distributions remains toggle-symmetric. It is similarly equivalent to Lemma~\ref{lem:mchaintogsym} thanks to the equations~(\ref{eqn:chaintomchain}) and~(\ref{eqn:mchaintochain}).
\end{proof}

Lemma~\ref{lem:chaintogsym} motivates the following strengthened version of the CDE property for distributive lattices.

\begin{definition}
We say that a distributive lattice $J(P)$ is \emph{toggle-CDE} (\emph{tCDE}) if $\mathbb{E}(\mu;\mathrm{ddeg})=\mathbb{E}(\mathrm{uni}_{J(P)};\mathrm{ddeg})$ for every toggle-symmetric distribution $\mu$ on~$J(P)$.
\end{definition}

Thanks to Lemma~\ref{lem:chaintogsym}, and because distributive lattices are graded, we have the following implications: $\textrm{$J(P)$ is tCDE $\Rightarrow$ $J(P)$ is mCDE $\Rightarrow$ $J(P)$ is CDE}$.  

\begin{example} \label{ex:chaintcde}
For any $a \geq 1$, the chain poset $\mathbf{a}$ is tCDE because there is only one toggle-symmetric distribution on $\mathbf{a}$, namely, the uniform distribution.
\end{example}

Let us now discuss how these three CDE properties behave under poset operations within the class of distributive lattices. The class of distributive lattices is closed under duality: we have~$J(P)^* = J(P^*)$ for any distributive lattice $J(P)$. Unlike in the case of arbitrary posets, duality preserves the three CDE properties for distributive lattices.

\begin{prop} \label{prop:tcdedual}
Let $J(P)$ be a distributive lattice.
\begin{itemize}
\item If $J(P)$ is CDE then $J(P)^*$ is CDE.
\item If $J(P)$ is mCDE then $J(P)^*$ is mCDE.
\item If $J(P)$ is tCDE then $J(P)^*$ is tCDE.
\end{itemize}
\end{prop}
\begin{proof}
The down-degree of an order ideal $I \in J(P)$ when considered as an element of the dual poset $J(P)^*$ is~$\sum_{p \in P} \mathcal{T}^{+}_p(I)$. But for any toggle-symmetric distribution~$\mu$ on $J(P)$ we clearly have~$\mathbb{E}(\mu;\sum_{p \in P} \mathcal{T}^{+}_p) = \mathbb{E}(\mu;\sum_{p \in P} \mathcal{T}^{-}_p) = \mathbb{E}(\mu;\mathrm{ddeg})$. The third bullet point now follows because $\mu$ is toggle-symmetric as a distribution on $J(P)$ iff~$\mu$ is toggle-symmetric as a distribution on $J(P)^*$. For the first two bullet points, observe that $\mathrm{chain}(i)_{J(P)} = \mathrm{chain}(i)_{J(P)^*}$ for any $i=0,1,\ldots,\#P$ and then use that these distributions are indeed toggle-symmetric thanks to Lemma~\ref{lem:chaintogsym}.
\end{proof}

The class of distributive lattices is also closed under direct products: for any two distributive lattices $J(P)$ and $J(Q)$ we have~$J(P) \times J(Q) = J(P + Q)$. The three CDE properties are preserved under direct products of distributive lattices, as we show in the following proposition. In light of this proposition we will mostly be interested in studying~$J(P)$ for $P$ connected.

\begin{prop} \label{prop:tcdeproduct}
Let $J(P)$ and $J(Q)$ be distributive lattices. 
\begin{itemize}
\item  If $J(P)$ and $J(Q)$ are both CDE then $J(P) \times J(Q)$ is also CDE.
\item If $J(P)$ and $J(Q)$ are both mCDE then $J(P) \times J(Q)$ is also mCDE.
\item If $J(P)$ and $J(Q)$ are both tCDE then $J(P) \times J(Q)$ is also tCDE.
\end{itemize}
\end{prop}
\begin{proof}
The first bullet point is a consequence of Proposition~\ref{prop:productcde} since $J(P)$ is necessarily graded. The second bullet is a consequence of Proposition~\ref{prop:productmcde}. So let us address the third bullet point. We use $\pi_{J(P)}$ and~$\pi_{J(Q)}$ for the projection maps from $J(P) \times J(Q)$ as in the proof of Proposition~\ref{prop:productmcde}. As in that proof, it suffices to show that
\[ \mathbb{E}(\pi_{J(P)}(\mu);\mathrm{ddeg}) = \mathbb{E}(\mathrm{uni}_{J(P)};\mathrm{ddeg}) \]
for any toggle-symmetric distribution $\mu$ on $J(P) \times J(Q)$. But this indeed holds because if $\mu$ is a toggle-symmetric distribution on $J(P) \times J(Q) = J(P + Q)$ then $\pi_{J(P)}(\mu)$ is a toggle-symmetric distribution on $J(P)$: $p \in P$ can be toggled into (out of) $I \in  J(P + Q)$ iff $p$ can be toggled into (out of) $\pi_{J(P)}(I)$.
\end{proof}

Let us conclude this section with a discussion of how the study of the tCDE property leads to nice formulas for edge densities. As we will see in a moment in Section~\ref{sec:younglat}, a consequence of the main result of~\cite{chan2015expected} is the assertion that certain special distributive lattices~$J(P)$ (intervals~$[\nu,\lambda]$ of Young's lattice where $\lambda/\nu$ is ``balanced'') are tCDE. Not only did CHHM show that such~$J(P)$ are tCDE, they also offered a simple expression for the edge densities of these~$J(P)$. In fact, under an additional assumption (namely, that~$P$ is graded) we can conclude that there has to be a simple expression for the edge density of~$J(P)$ if~$J(P)$ is to be~tCDE.

\begin{prop} \label{prop:togsymedgedensity}
Let $P$ be a graded poset of rank $r$ for which $J(P)$ is tCDE. Then the edge density of $J(P)$ is
\[ \mathbb{E}(\mathrm{uni}_{J(P)};\mathrm{ddeg}) = \frac{\#P}{r+2}.\]
\end{prop}
\begin{proof}
Let $P$ be as in the statement of the proposition. Recall the notation for the ranks of $P$: $P_i := \{p \in P\colon \mathrm{rk}(p) = i\}$ and $P_{\leq i} := \cup_{j=0}^{i}P_j$. Set $P_{-1} := P_{\leq -1} := \varnothing$ by convention. Define a distribution $\mathrm{rank}$ on $J(P)$ by
\[ \mathbb{P}(\mathrm{rank};I) := \begin{cases} \frac{1}{r+2} &\textrm{if $I = P_{\leq i}$ for some $i=-1,0,1,\ldots,r$}, \\ 0 &\textrm{otherwise}.\end{cases} \]
Observe that $\mathrm{rank}$ is toggle-symmetric: $p \in P$ can be toggled into $P_{\leq i}$ iff $p$ can be toggled out of $P_{\leq i+1}$. But also observe that
\begin{align*}
\mathbb{E}(\mathrm{rank};\mathrm{ddeg}) &= \frac{1}{r+2} \sum_{i=-1}^{r} \mathrm{ddeg}(P_{\leq i})\\
&= \frac{1}{r+2}\sum_{i=-1}^{r} \#P_{i} \\
&= \frac{\#P}{r+2}. 
\end{align*}
Since $J(P)$ is tCDE we conclude $\mathbb{E}(\mathrm{uni}_{J(P)};\mathrm{ddeg}) = \mathbb{E}(\mathrm{rank};\mathrm{ddeg}) =  \frac{\#P}{r+2}$.
\end{proof}

Proposition~\ref{prop:togsymedgedensity} says that if $P$ is a graded poset for which $J(P)$ is tCDE, then the average size of an antichain of $P$ is equal to the average size of a rank of $P$ (once we include the extra ``empty rank'' $P_{-1}$). In the next several sections we will provide many examples of tCDE distributive lattices, culminating with Section~\ref{sec:minuscule} where we show that the distributive lattice $J(P)$ associated to a minuscule poset $P$ is always tCDE. Note that if $P$ is a connected minuscule poset then $P$ is graded (see e.g.~\cite{proctor1984bruhat}) and thus Proposition~\ref{prop:togsymedgedensity} applies and tells us that a~priori we should expect a nice formula for the edge density of $J(P)$. Heuristically, even when Proposition~\ref{prop:togsymedgedensity} does not apply we still get a nice formula for the edge density of a tCDE distributive lattice.

\section{Intervals of Young's lattice} \label{sec:younglat}

Having dealt with generalities for so long, in this section we finally get to some nontrivial examples of CDE (in fact tCDE) distributive lattices. This section, which concerns Young's lattice and ordinary Young diagrams, is actually just a summary of the methods and results of CHHM~\cite{chan2015expected}. We feel compelled to review material in~\cite{chan2015expected} because in the next section will extend these methods and results to the shifted setting.

We use standard notation for partitions (as in~e.g.~\cite[\S7]{stanley1999ec2}) which we nevertheless review. A \emph{partition} is a an infinite sequence~$\lambda = (\lambda_1, \lambda_2, \lambda_3, \cdots)$ of nonnegative integers that is weakly decreasing ($\lambda_i \geq \lambda_{i+1}$ for all $i \geq 1$) and eventually zero ($\lambda_N = 0$ for all~$N \gg 0$ sufficiently large). The $\lambda_i$ are called the \emph{parts} of $\lambda$. The \emph{length} of $\lambda$ is the smallest nonnegative integer $\ell(\lambda)$ such that $\lambda_i = 0$ for all $i > \ell(\lambda)$; i.e.,~$\ell(\lambda)$ is the number of nonzero parts of~$\lambda$. We also write $\lambda = (\lambda_1,\lambda_2,\ldots,\lambda_{\ell(\lambda)})$. The \emph{size} of~$\lambda$ is~$|\lambda| := \sum_{i=1}^{\infty} \lambda_i$. If $|\lambda| = n$ then we say $\lambda$ is a \emph{partition of $n$} and write~$\lambda \vdash n$. Associated to any $\lambda \vdash n$ is its \emph{Young diagram}, which is a collection of $n$ boxes in $\ell(\lambda)$ left-justified rows, with $\lambda_i$ boxes in row $i$. (Thus we use so-called ``English notation.'')  For example, the Young diagram of $\lambda = (4,3,3,3)$ is
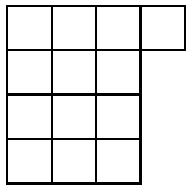
\[\ydiagram{4,3,3,3} \]
Two families of partitions that will be important to us are the following:
\begin{itemize}
\item \emph{$a\times b$ rectangles}: $b^a := (\overbrace{b,b,\ldots,b}^{a})$ for $a,b \geq 0$;
\item \emph{length $d$ staircases}: $\delta_d := (d,d-1,d-2,\ldots,1)$ for $d \geq 0$.
\end{itemize}
For partitions $\nu,\lambda$ we write $\nu \subseteq \lambda$ if the Young diagram of $\nu$ is contained in $\lambda$; in other words, $\nu \subseteq \lambda$ iff $\nu_i \leq \lambda_i$ for all $i \geq 1$. \emph{Young's lattice} is the infinite poset of all partitions with the partial order $\subseteq$. There is a unique minimal element of Young's lattice called the \emph{empty partition} and denoted $\varnothing$. Every interval $[\nu,\lambda]$ in Young's lattice is a (finite) distributive lattice. Let us explain how to write $[\nu,\lambda] = J(P)$ for some poset $P$.

For any partitions~$\nu \subseteq \lambda$ we define the \emph{skew shape} $\lambda/\nu$ to be the set-theoretical differences of the Young diagrams of $\lambda$ and $\nu$. For example, taking~$\lambda = (4,3,3,3)$ and~$\nu = (2,2)$, the skew shape $\lambda/\nu$ is
\[\ydiagram{2+2,2+1,3,3}\]
We use the standard compass directions to refer to the relative positions of boxes of a skew shape. We turn the skew shape $\lambda/\nu$ into a poset $P_{\lambda/\nu}$ by defining the following partial order on the boxes of $\lambda/\nu$: for boxes $u,v$ of $\lambda/\nu$ we have $u \leq v$ iff $v$ is weakly southeast of~$u$. Note that $P_{\lambda/\nu}$ is always ranked but may or may not be graded. Also observe that~$[\nu,\lambda] = J(P_{\lambda/\nu})$: the order ideals of $P_{\lambda/\nu}$ are exactly the skew shapes~$\rho/\nu$ for partitions $\rho$ satisfying $\nu \subseteq \rho \subseteq \lambda$. We say $\lambda/\nu$ is \emph{connected} if $P_{\lambda/\nu}$ is connected. In the special case~$\nu = \varnothing$, the skew shape $\lambda/\nu$ is called a \emph{straight shape} and is also denoted simply by~$\lambda$ with a corresponding poset denoted~$P_{\lambda}$. A straight shape~$\lambda$ is always connected.

For the next two paragraphs let $\lambda/\nu$ be a fixed connected skew shape. We define the \emph{height}~$a$ and \emph{width}~$b$ to be the least nonnegative integers~$a$ and~$b$ such that $\lambda/\nu$ can be put (via translation) inside an $a \times b$ rectangle Young diagram. A straight shape $\lambda$ has height equal to its length~$\ell(\lambda)$ and width equal to its first part~$\lambda_1$. We define a coordinate system (``matrix coordinates'') on the lattice points of this $a \times b$ rectangle whereby the northwestern-most lattice point of the rectangle is $(0,0)$ and the southeastern-most is~$(a,b)$. We extend this to a coordinate system on the boxes of this $a \times b$ rectangle by letting~$[i,j]$ denote the box whose southeast corner is $(i,j)$. Thus the northwestern-most box is~$[1,1]$ and the southeastern-most is $[a,b]$. We write~$[i,j] \in \lambda/\nu$ to mean that $[i,j]$ is a box of $\lambda/\nu$. Necessarily~$[a,1] \in \lambda/\nu$ and $[1,b] \in \lambda/\nu$.

\begin{figure}
\begin{center}
\begin{tikzpicture}
\node at (0,0) {\begin{ytableau}\none & \none & *(yellow) \bullet & \\ \none & \none &   \\ *(yellow) \bullet &  &  \\ *(yellow) \bullet & & \end{ytableau}};
\def\x{0.6}
\draw[red,line width=3pt] (-2*\x,-2*\x) -- (-1*\x,-2*\x) -- (-1*\x,0*\x) -- (0*\x,0*\x) -- (0*\x,1*\x) -- (1*\x,1*\x) -- (1*\x,2*\x) -- (2*\x,2*\x);
\end{tikzpicture}
\end{center}
\caption{With $\lambda/\nu = (4,3,3,3) / (2,2)$, we have shaded the boxes of~$(3,2,1,1) / (2,2) \in J(P_{\lambda/\nu})$ in yellow (and marked them with black dots) and drawn its associated lattice path in red.} \label{fig:latticepath}
\end{figure}

As mentioned, the order ideals of $\lambda/\nu$ correspond to $\rho/\nu$ for partitions $\nu \subseteq \rho \subseteq \lambda$. These order ideals are also in bijective correspondence with lattice paths from~$(a,0)$ to~$(0,b)$ with steps of the form $(-1,0)$ (\emph{north steps}) and $(0,1)$ (\emph{east step}) that stay inside the boxes of~$\lambda/\nu$: we associate a skew shape~$\rho/\nu$ to the lattice path from~$(a,0)$ to $(0,b)$ that traces the southeast border of $\rho$. See Figure~\ref{fig:latticepath} for an example of this correspondence. The \emph{northwest border of $\lambda/\nu$} is the lattice path associated to the order ideal $\nu/\nu$. The \emph{southeast border of~$\lambda/\nu$} is the lattice path associated to the order ideal~$\lambda/\nu$. A \emph{northwest outward corner} of $\lambda/\nu$ is an east step followed by a north step on the northwest border of $\lambda/\nu$. A \emph{southeast outward corner} is a north step followed by an east step on the southeast border. The northwest and southeast outward corners of $\lambda/\nu$ are collectively referred to as simply the \emph{outward corners} of $\lambda/\nu$. The set of outward corners is denoted $C(\lambda/\nu)$. We say that an outward corner $c \in C(\lambda/\nu)$ \emph{occurs} at $(i,j)$ if that is the lattice point in common between the two steps of~$c$. The \emph{main anti-diagonal} of $\lambda/\nu$ is the line connecting $(a,0)$ to $(0,b)$. We say that $\lambda/\nu$ is \emph{balanced} if $c$ occurs at a lattice point on the main anti-diagonal for every $c \in C(\lambda/\nu)$. (By convention saying that~$\lambda/\nu$ is balanced implies it is connected.)

\begin{example} \label{ex:balanced}
Let $\lambda/\nu = (4,3,3,3)/(2,2)$ as depicted below:
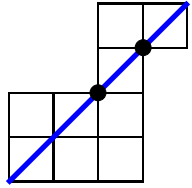
\begin{center}
\begin{tikzpicture}
\node at (0,0) {\ydiagram{2+2,2+1,3,3}};
\def\x{0.6}
\draw[blue,line width=2pt] (-2*\x,-2*\x) -- (2*\x,2*\x);
\filldraw (1*\x,1*\x) circle (3pt);
\filldraw (0*\x,0*\x) circle (3pt);
\end{tikzpicture}
\end{center}
Observe that $\lambda/\nu$ has one northwest outward corner and one southeast outward corner. For each outward corner $c \in C(\lambda/\nu)$, the point $(i,j)$ at which $c$ occurs is marked with a black circle. The main anti-diagonal of $\lambda/\nu$ is drawn in blue. We see that all outward corners occur on the main anti-diagonal and thus $\lambda/\nu$ is balanced.
\end{example}

The straight shapes $b^a$ and $\delta_{d}$ are balanced for all $a,b,d \geq 0$. For a skew shape~$\lambda/\nu$ let us use~$(\lambda/\nu) \circ b^a$ to denote the skew shape obtained from~$\lambda/\nu$ by replacing each box of~$\lambda/\nu$ with an $a \times b$ rectangle. If $\lambda/\nu$ is balanced then $(\lambda/\nu) \circ b^a$ is also balanced. In particular, the ``stretched staircase''~$\delta_d \circ b^a$ is always balanced. For any~$a,b \geq 1$ there are up to translation in the plane~$3^{\mathrm{gcd}(a,b)-1}$ balanced skew shapes of height~$a$ and width~$b$ and~$2^{\mathrm{gcd}(a,b)-1}$ such balanced straight shapes. The point of considering balanced skew shapes is the following result of CHHM~\cite{chan2015expected}.

\begin{thm}[{CHHM~\cite[Corollary 3.8]{chan2015expected}}] \label{thm:younglatcde}
Let $\lambda/\nu$ be a balanced skew shape of height $a$ and width $b$ with~$a,b \geq 1$. Then the distributive lattice $[\nu,\lambda]$ is tCDE with edge density
\[\mathbb{E}(\mathrm{uni}_{[\nu,\lambda]};\mathrm{ddeg}) = \frac{ab}{a+b}.\]
\end{thm}

Let us take a moment to review the history of Theorem~\ref{thm:younglatcde}. First of all, CHHM stated their result in terms of expected jaggedness rather than expected down-degree but as mentioned in the previous section these are easily seen to be equivalent. They were preceded and motivated by Chan, L\'{o}pez Mart\'{i}n, Pflueger and Teixidor i Bigas~\cite[Corollary 2.15]{chan2015genera} who established (in different language) that $[\varnothing,b^a]$ is CDE with edge-density $\frac{ab}{a+b}$. Indeed, the computation of~$\mathbb{E}(\mathrm{maxchain}_{[\varnothing,b^a]};\mathrm{ddeg})$ in~\cite{chan2015genera} was the key combinatorial result needed to reprove an algebro-geometric formula of Eisenbud-Harris~\cite{eisenbud1987kodaira} and Pirola~\cite{pirola1985chern} for the genus of Brill-Noether curves. Theorem~\ref{thm:younglatcde} is also directly related to the work of Reiner-Tenner-Yong: one part of their main result~\cite[Theorem 1.1(c)]{reiner2016poset} is the assertion that $[\varnothing,\delta_d \circ b^a]$ is CDE. Since~$\delta_d\circ b^a$ is balanced this assertion is a consequence of Theorem~\ref{thm:younglatcde}. However, we should note that the approach of Reiner-Tenner-Yong to finding CDE intervals of Young's lattice was rather different than that of CHHM: Reiner-Tenner-Yong used symmetric function and tableaux theory (which they were able to apply also to intervals of the weak Bruhat order); as we have seen, CHHM instead adopted the perspective of ``toggle theory.'' We will return to tableaux in Section~\ref{sec:tableaux}. For now let us sketch the proof of Theorem~\ref{thm:younglatcde} in~\cite[\S3]{chan2015expected}.

Fix $\lambda/\nu$, a connected skew shape of height~$a$ and width~$b$ with coordinates as described earlier. The proof of Theorem~\ref{thm:younglatcde} follows from consideration of certain ``rook'' statistics on $J(P_{\lambda/\nu})$ that are linear combinations of the toggleability statistics: for~$[i,j] \in \lambda/\nu$ we define the rook $R_{ij}\colon J(P_{\lambda/\nu})\to \mathbb{R}$ by
\[R_{ij} := \sum_{\substack{i'\leq i, \; \, j' \leq j \\ [i',j']\in\lambda/\nu }}\!\mathcal{T}_{[i',j']}^+ + \sum_{\substack{i'\geq i, \; \, j' \geq j \\ [i',j']\in\lambda/\nu}}\!\mathcal{T}_{[i',j']}^- - \sum_{\substack{i'< i, \; \, j'< j \\ [i',j']\in\lambda/\nu}}\!\mathcal{T}_{[i',j']}^- - \sum_{\substack{i'> i, \; \, j'> j \\ [i',j']\in\lambda/\nu}}\!\mathcal{T}_{[i',j']}^+.\]
The reason this statistic is called a rook is because for a toggle-symmetric distribution~$\mu$ on $J(P_{\lambda/\nu})$ the expectation $\mathbb{E}(\mu;R_{ij})$ depends only the toggleability statistics for boxes in either the same column or row as~$[i,j]$. In other words, for a toggle-symmetric distribution the rook $R_{ij}$ ``attacks'' in expectation the boxes in its row and column:

\begin{lemma}[{CHHM~\cite[Lemma 3.5]{chan2015expected}}] \label{lem:rooktogsym}
For any toggle-symmetric distribution $\mu$ on $J(P_{\lambda/\nu})$ and any $[i,j] \in \lambda/\nu$,
\[\mathbb{E}(\mu;R_{ij}) = \sum_{[i',j] \in \lambda/\nu} \mathbb{E}(\mu;\mathcal{T}^{-}_{[i',j]}) + \sum_{[i,j'] \in \lambda/\nu} \mathbb{E}(\mu;\mathcal{T}^{-}_{[i,j']}).\] 
\end{lemma}

Lemma~\ref{lem:rooktogsym} follows immediately from the definition of toggle-symmetry. The rooks end up being useful because there is another, more subtle expression for $\mathbb{E}(\mu;R_{ij})$ for any probability distribution (toggle-symmetric or otherwise) $\mu$ on $J(P_{\lambda/\nu})$. To state this expression we need a little bit more notation. For an order ideal $\rho/\nu \in J(P_{\lambda/\nu})$ we say the $\rho/\nu$ \emph{contains} an outward corner $c \in C(\lambda/\nu)$, written $c \in \rho/\nu$, if the lattice path associated to $\rho/\nu$ contains the two steps that make up $c$. For instance, in Figure~\ref{fig:latticepath}, the depicted order ideal $\rho/\nu$ contains the northwest corner that occurs at $(2,2)$ but does not contain the southeast corner that occurs at $(1,3)$. For a box $[i,j] \in \lambda/\nu$ let $C_{ij}(\lambda/\nu)$ denote the set of outward corners $c \in C(\lambda/\nu)$ that occur strictly northeast or strictly southwest of the center of the box $[i,j]$. For example, with $\lambda/\nu$ as in Example~\ref{ex:balanced}, $C_{3,3}(\lambda/\nu)$ is just the northwest corner that occurs at $(2,2)$, while $C_{1,3}(\lambda/\nu)$ is just the southeast corner that occurs at $(1,3)$. With this notation, we have,

\begin{lemma}[{CHHM~\cite[Lemma 3.6]{chan2015expected}}] \label{lem:rookexpectation}
For any probability distribution $\mu$ on $J(P_{\lambda/\nu})$ and any $[i,j] \in \lambda/\nu$,
\[\mathbb{E}(\mu;R_{ij}) = 1+\sum_{c \in C_{ij}(\lambda/\nu)} \mathbb{P}(\rho/\nu\sim \mu; c \in \rho/\nu).\] 
\end{lemma}

For a proof of Lemma~\ref{lem:rookexpectation} see~\cite[Lemma 3.6]{chan2015expected}. We do not go into details here because we will have to provide a very similar proof of an analogous lemma in the shifted setting in Section~\ref{sec:shiftyounglat}. To combine Lemmas~\ref{lem:rooktogsym} and~\ref{lem:rookexpectation} to prove Theorem~\ref{thm:younglatcde} all we have to do is find a way to place (positive and negative) rooks on the boxes of $\lambda/\nu$ so that every box is attacked the same number of times and so that in aggregate the error terms from Lemma~\ref{lem:rookexpectation} cancel out for every outward corner. We can do this when $\lambda/\nu$ is balanced. That is,

\begin{lemma}[{CHHM~\cite[Lemma 3.7]{chan2015expected}}] \label{lem:rookplacement}
For any connected skew shape $\lambda/\nu$ with height~$a$ and width~$b$ there exist integral coefficients $r_{ij} \in \mathbb{Z}$ for $[i,j] \in \lambda/\nu$ such that
\begin{itemize}
\item for all $1 \leq i \leq a$, $\sum_{[i,j'] \in \lambda/\nu}r_{i,j'} = b$;
\item for all $1 \leq j \leq b$, $\sum_{[i',j] \in \lambda/\nu}r_{i',j} = a$.
\end{itemize}
When $\lambda/\nu$ is balanced these coefficients additionally satisfy $\sum_{[i,j] \in \lambda/\nu, \; c \in C_{ij}(\lambda/\nu)} r_{ij} = 0$ for all outward corners $c \in C(\lambda/\mu)$.
\end{lemma}

An example of coefficients which satisfy Lemma~\ref{lem:rookplacement} for the skew shape $\lambda/\nu$ from Example~\ref{ex:balanced} is
\begin{center}
\begin{ytableau}\none & \none & 0 & 4 \\ \none & \none & 4  \\ 0 & 0 & 4 \\ 4 & 4 & -4 \end{ytableau}
\end{center} 
where $r_{ij}$ is written inside box $[i,j]$. Theorem~\ref{thm:younglatcde} follows easily from Lemmas~\ref{lem:rooktogsym},~\ref{lem:rookexpectation} and~\ref{lem:rookplacement}: let $r_{ij}$ be as in Lemma~\ref{lem:rookplacement} and consider the statistic $\sum_{[i,j] \in \lambda/\nu}r_{ij}R_{ij}$; on the one hand, for a toggle-symmetric distribution $\mu$ on $J(P_{\lambda/\nu})$ we have
\[ \mathbb{E}(\mu; \hspace{-0.2cm} \sum_{[i,j] \in \lambda/\nu} \hspace{-0.2cm} r_{ij}R_{ij}) =  \hspace{-0.2cm} \sum_{[i,j] \in \lambda/\nu} (a+b) \cdot \mathbb{E}(\mu;\mathcal{T}^{-}_{[i,j]}) = (a+b) \cdot \mathbb{E}(\mu;\mathrm{ddeg}) \]
thanks to Lemma~\ref{lem:rooktogsym}; on the other hand,
\[ \mathbb{E}(\mu; \hspace{-0.2cm} \sum_{[i,j] \in \lambda/\nu} \hspace{-0.2cm} r_{ij}R_{ij}) = ab +  \hspace{-0.2cm} \sum_{c \in C(\lambda/\mu)} \mathbb{P}(\rho/\nu\sim \mu; c \in \rho/\nu)\cdot \hspace{-1.1cm} \sum_{[i,j] \in \lambda/\nu, \; c \in C_{ij}(\lambda/\nu)}  \hspace{-0.8cm} r_{ij} = ab\]
thanks to Lemma~\ref{lem:rookexpectation}; comparing these two expressions for $\mathbb{E}(\mu; \sum_{[i,j] \in \lambda/\nu} r_{ij}R_{ij})$ yields the claimed formula for $\mathbb{E}(\mu;\mathrm{ddeg})$.

\begin{remark} \label{rem:younglatcde}
It is natural to ask if Theorem~\ref{thm:younglatcde} has a converse. In other words, it is natural to try to understand \emph{exactly} when $[\nu,\lambda]$ is tCDE, or better yet, CDE. Computation with SAGE mathematical software tells us that for all partitions $\lambda$ with~$|\lambda| \leq 30$ if the interval~$[\varnothing,\lambda]$ is CDE then~$\lambda$ is balanced. It thus seems reasonable to conjecture that the only initial intervals (i.e., intervals of the form $[\varnothing,\lambda]$) of Young's lattice that are CDE correspond to balanced straight shapes. If this conjecture were true then the CDE, mCDE, and tCDE properties would all coincide for initial intervals of Young's lattice. The main result of CHHM~\cite[Theorem 3.4]{chan2015expected} is a formula for $\mathbb{E}(\mu;\mathrm{ddeg})$ for any toggle-symmetric distribution~$\mu$ on~$J(P_{\lambda/\nu})$ where~$\lambda/\nu$ is an arbitrary connected skew shape. In principle one could analyze this formula to see when it is possible to have~$\mathbb{E}(\mathrm{uni}_{[\varnothing,\lambda]};\mathrm{ddeg}) = \mathbb{E}(\mathrm{maxchain}_{[\varnothing,\lambda]};\mathrm{ddeg})$. However, such analysis does not appear to be totally straightforward. For instance, one might hope that $\mathbb{E}(\mathrm{uni}_{[\varnothing,\lambda]};\mathrm{ddeg})$ is always less than or equal to $\mathbb{E}(\mathrm{maxchain}_{[\varnothing,\lambda]};\mathrm{ddeg})$ or vice-versa. This is not the case; for example, with $\lambda = (3,1)$ we have
\[\mathbb{E}(\mathrm{uni}_{[\varnothing,(3,1)]};\mathrm{ddeg}) = \frac{8}{7} > \frac{17}{15} = \mathbb{E}(\mathrm{maxchain}_{[\varnothing,(3,1)]};\mathrm{ddeg})\]
but with $\lambda = (3,2)$ we have
\[\mathbb{E}(\mathrm{uni}_{[\varnothing,(3,2)]};\mathrm{ddeg}) = \frac{11}{9} < \frac{37}{30} = \mathbb{E}(\mathrm{maxchain}_{[\varnothing,(3,2)]};\mathrm{ddeg}).\]
These examples also show that $[\nu,\lambda]$ can be CDE without each connected component of the skew shape~$\lambda/\nu$ being balanced. This is because, as we have seen in the proof of Proposition~\ref{prop:productmcde}, we have~$\mathbb{E}(\mathrm{uni}_{P\times Q};\mathrm{ddeg}) =\mathbb{E}(\mathrm{uni}_{P};\mathrm{ddeg}) + \mathbb{E}(\mathrm{uni}_{Q};\mathrm{ddeg})$ and~$\mathbb{E}(\mathrm{maxchain}_{P\times Q};\mathrm{ddeg}) =\mathbb{E}(\mathrm{maxchain}_{P};\mathrm{ddeg}) + \mathbb{E}(\mathrm{maxchain}_{Q};\mathrm{ddeg})$; and because for any two intervals $[\nu^1,\lambda^1]$ and $[\nu^2,\lambda^2]$ of Young's lattice there is an interval $[\nu,\lambda]$ with~$[\nu,\lambda] = [\nu^1,\lambda^1] \times [\nu^2,\lambda^2]$: we form the skew shape~$\lambda/\nu$ by putting~$\lambda^2/\nu^2$ completely to the southwest of $\lambda^1/\nu^1$. So by bringing together many copies of the straight shapes~$(3,1)$ and~$(3,2)$ in this way we can get a CDE interval of Young's lattice corresponding to a skew shape whose connected components are all $(3,1)$'s or $(3,2)$'s. Classifying arbitrary intervals of Young's lattice that are CDE therefore appears intractable. Perhaps something could be said about intervals~$[\nu,\lambda]$ with~$\lambda/\nu$ connected.
\end{remark}

\section{Intervals of the shifted Young's lattice} \label{sec:shiftyounglat}

We now extend the results of Section~\ref{sec:younglat} to the shifted setting. A \emph{strict partition} is a partition $\lambda$ with $\lambda_i > \lambda_{i+1}$ whenever $\lambda_i \neq 0$. So $\lambda = (4,3,1)$ is strict but~$\lambda = (4,3,3)$ is not. The \emph{shifted Young's lattice} is the restriction of Young's lattice to the subset of all strict partitions. If $\lambda$ and $\nu$ are strict partitions with $\nu \subseteq \lambda$ we use the notation~$[\nu,\lambda]_{\mathrm{shift}}$ to denote the interval between $\nu$ and $\lambda$ in the shifted Young's lattice to avoid confusion with intervals in the ordinary Young's lattice. It turns out that every interval $[\nu,\lambda]_{\mathrm{shift}}$ of the shifted Young's lattice is also a distributive lattice. In this section we will in fact restrict our attention to initial intervals~$[\varnothing,\lambda]_{\mathrm{shift}}$. So let us show how to write~$[\varnothing,\lambda]_{\mathrm{shift}} = J(P)$ for some poset $P$. Associated to any strict partition $\lambda$ is its \emph{shifted Young diagram}, which is the same as the Young diagram of $\lambda$ except that each row after the first is indented by one box from the row above it. For example, the shifted Young diagram of $\lambda = (8,6,5,2,1)$ is
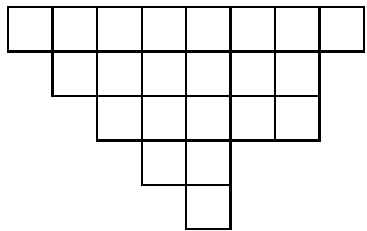
\[\ydiagram{8,1+6,2+5,3+2,4+1}\]
We again use compass directions to describe the relative positions of boxes of shifted Young diagrams. We turn the shifted Young diagram of $\lambda$ into a poset $P^{\mathrm{shift}}_{\lambda}$ via the same partial order as in the ordinary Young diagram case: for boxes $u,v$ of the shifted Young diagram of $\lambda$ we have $u \leq v$ iff $v$ is weakly southeast of~$u$. Note that once again~$P^{\mathrm{shift}}_{\lambda}$ is always ranked but may or may not be graded. Also observe that~$[\varnothing,\lambda]_{\mathrm{shift}} = J(P^{\mathrm{shift}}_{\lambda})$: the order ideals of $P^{\mathrm{shift}}_{\lambda}$ are exactly the shifted Young diagrams of strict partitions~$\nu$ with $\nu \subseteq \lambda$.

For two partitions $\nu$ and $\lambda$ we use $\nu+\lambda$ to denote the partition whose parts are the sums of the corresponding parts of $\nu$ and $\lambda$; that is, $\nu+\lambda := (\nu_1+\lambda_1,\nu_2+\lambda_2,\nu_3+\lambda_3,\ldots)$. Observe that every strict partition $\lambda$ is $\lambda = \delta_n + \nu$ for a unique $n$ and (not necessarily strict) partition $\nu$ with $\ell(\nu) \leq n$. For example, in the shifted Young diagram below for~$\lambda = \delta_3 + (2,2)$ we have shaded the boxes of the~$\delta_3$ ``section'' in yellow:
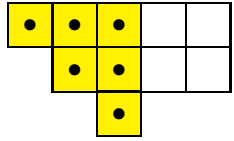
\begin{center}
\begin{ytableau}  *(yellow) \bullet &  *(yellow) \bullet  &  *(yellow) \bullet &  & \\  \none &  *(yellow) \bullet  &  *(yellow) \bullet &  &   \\ \none & \none &  *(yellow) \bullet \end{ytableau}
\end{center} 

\subsection{The first shifted Young's lattice conjecture} \label{subsec:shiftyounglatconj1}

Reiner-Tenner-Yong made two conjectures about CDE initial intervals of the shifted Young's lattice~\cite[\S2.3.1]{reiner2016poset}. Here we prove one of these conjectures.

\begin{conj}[{Reiner-Tenner-Yong~\cite[Conjecture 2.25]{reiner2016poset}}] \label{conj:stretchstair}
Let $\lambda := \delta_n + \delta_m\circ a^a$ where $n > am$ with $a,m,n \geq 1$. Then $[\varnothing,\lambda]_{\mathrm{shift}}$ is CDE with edge density
\[\mathbb{E}(\mathrm{uni}_{[\varnothing,\lambda]_{\mathrm{shift}}};\mathrm{ddeg}) = \frac{n+1+am}{4}\]
\end{conj}

We prove a stronger version of Conjecture~\ref{conj:stretchstair}. Namely,

\begin{thm} \label{thm:shiftyounglatcde}
Let $0 \leq k < n$. Let $\nu$ be a (not necessarily strict) partition which as a straight shape is balanced with height and width both equal to $k$.
\begin{enumerate}
\item Setting $\lambda:= \delta_{n} + \nu$, the distributive lattice $[\varnothing,\lambda]_{\mathrm{shift}}$ is tCDE with edge density
\[\mathbb{E}(\mathrm{uni}_{[\varnothing,\lambda]_{\mathrm{shift}}};\mathrm{ddeg}) = \frac{n+1+k}{4}.\]
\item Setting $\lambda:= \delta_{n} + \nu + (n-1-k)^{n}$, the distributive lattice $[\varnothing,\lambda]_{\mathrm{shift}}$ is tCDE with edge density
\[\mathbb{E}(\mathrm{uni}_{[\varnothing,\lambda]_{\mathrm{shift}}};\mathrm{ddeg}) = \frac{n}{2}.\]
\end{enumerate}
\end{thm}

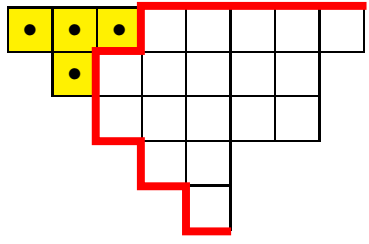
\begin{figure}
\begin{center}
\begin{tikzpicture}
\node at (0,0) {\begin{ytableau}  *(yellow) \bullet & *(yellow) \bullet & *(yellow) \bullet &  & & & & \\ \none &  *(yellow) \bullet &  & & & &   \\  \none &  \none & & & & &  \\ \none & \none & \none & & \\ \none & \none & \none & \none & \end{ytableau}};
\def\x{0.6}
\draw[red,line width=3pt] (1*\x,-2.5*\x) -- (0*\x,-2.5*\x) -- (0*\x,-1.5*\x) -- (-1*\x,-1.5*\x) -- (-1*\x,-0.5*\x) -- (-2*\x,-0.5*\x) -- (-2*\x,1.5*\x) -- (-1*\x,1.5*\x) -- (-1*\x,2.5*\x) -- (4*\x,2.5*\x);
\end{tikzpicture}
\end{center}
\caption{For $\lambda = (8,6,5,2,1)$, the boxes of $(3,1) \in J(P_{\lambda}^{\mathrm{shift}})$ are shaded in yellow and its associated lattice path is drawn in red.} \label{fig:shiftlatticepath}
\end{figure}

Let us call a strict partition \emph{shifted-balanced of Type~(1)} if it is of the form appearing in part~(1) of the statement of Theorem~\ref{thm:shiftyounglatcde}, and define \emph{shifted-balanced of Type~(2)} analogously. The strict partition $\delta_n + \delta_m\circ a^a$ from Conjecture~\ref{conj:stretchstair} is shifted-balanced of Type~(1) and thus Theorem~\ref{thm:shiftyounglatcde} indeed implies Conjecture~\ref{conj:stretchstair}.

The key to proving Theorem~\ref{thm:shiftyounglatcde} is coming up with an appropriate shifted version of the rooks from Section~\ref{sec:younglat}. To define these rooks we need a little more notation. From now on in this subsection let $\lambda$ be a fixed strict partition. We use the same ``matrix coordinates'' for boxes of the shifted Young diagram of $\lambda$ as we did for ordinary Young diagrams in Section~\ref{sec:younglat}. So the northwestern-most box of $\lambda$ is $[1,1]$, the northeastern-most box is $[1,\lambda_1]$, and the western-most box in the last row is $[\ell(\lambda),\ell(\lambda)]$. Observe that $i \leq j$ for all boxes $[i,j] \in P_{\lambda}^{\mathrm{shift}}$. Let us call a box~$[i,j] \in P_{\lambda}^{\mathrm{shift}}$ with $i=j$ a \emph{main diagonal box}; the main diagonal boxes behave somewhat differently than the other boxes. As with order ideals for ordinary skew shapes, the order ideals of $P_{\lambda}^{\mathrm{shift}}$ correspond to lattice paths: in this case we get lattice paths from $(\ell(\lambda),\ell(\lambda))$ to $(\lambda_1,0)$ inside of the shifted Young diagram of $\lambda$ that start off with some (possibly empty) sequence of steps of the form west, north, west, north, et cetera along the southwest main diagonal border and then consists only of north and east steps. See Figure~\ref{fig:shiftlatticepath} for an example of this correspondence of order ideals with lattice paths.

\begin{figure}
\ytableausetup{boxsize=2em}
\begin{tikzpicture} 
\node at (0,0) {\begin{ytableau}
\scalebox{1.0}[1.1]{$^{1}\,\,\,\,\,\,$} & \scalebox{1.0}[1.1]{$^{1}\,\,_{\text{-}1}$} & \scalebox{1.0}[1.1]{$^{1}\,\,_{\text{-}1}$} & \scalebox{1.0}[1.1]{$^{1}\,\,\,\,\,\,$} & & & & \\
\none &  \scalebox{1.0}[1.1]{$^{1}\,\,\,\,\,\,$} &  \scalebox{1.0}[1.1]{$^{1}\,\,\,\,\,\,$} &\scalebox{1.0}[1.1]{$^{1}\,\,\,_{1}$} & \scalebox{1.0}[1.1]{$\,\,\,\,\,\,_{1}$} & \scalebox{1.0}[1.1]{$\,\,\,\,_{1}$} & \scalebox{1.0}[1.1]{$\,\,\,\,_{1}$} \\
\none & \none & & \scalebox{1.0}[1.1]{$\,\,\,\,\,\,_{1}$} & \scalebox{1.0}[1.1]{$^{\text{-}1}\,\,_{1}$} & \scalebox{1.0}[1.1]{$^{\text{-}1}\,\,_{1}$} & \scalebox{1.0}[1.1]{$^{\text{-}1}\,\,_{1}$} \\
 \none & \none & \none & \scalebox{1.0}[1.1]{$\,\,\,\,\,\,_{1}$} & \scalebox{1.0}[1.1]{$^{\text{-}1}\,\,_{1}$}  \\
  \none & \none & \none & \none & \scalebox{1.0}[1.1]{$\,\,\,\,\,\,_{1}$}
\end{ytableau} }; 
\end{tikzpicture} \; \; \begin{tikzpicture}
\node at (0,0) {\ydiagram{8,1+6,2+5,3+2,4+1}};
\def\x{0.8}
\draw[blue,line width=2pt] (-3.5*\x,2*\x) -- (-2.5*\x,1*\x) -- (2.5*\x,1*\x);
\draw[blue,line width=2pt] (-0.5*\x,2*\x) -- (-0.5*\x,-1*\x) -- (0.5*\x,-2*\x);
\filldraw (-0.5*\x,1*\x) circle (3pt);
\end{tikzpicture}
\ytableausetup{boxsize=1.5em}
\caption{On the left we depict the rook $R_{2,4}^{\mathrm{shift}}$; on the right we show the boxes that $R_{2,4}^{\mathrm{shift}}$ ``attacks'' in expectation for a toggle-symmetric distribution (note that $[2,4]$ itself is attacked ``twice'').} \label{fig:shiftrook}
\end{figure}

Here are the shifted rooks: for $[i,j] \in P_{\lambda}^{\mathrm{shift}}$ we define $R_{ij}^{\mathrm{shift}}\colon J(P_{\lambda}^{\mathrm{shift}}) \to \mathbb{R}$ by
\[R^{\mathrm{shift}}_{ij} := \sum_{\substack{i'\leq i, \; \, j' \leq j \\ [i',j']\in P_{\lambda}^{\mathrm{shift}} }}\!\mathcal{T}_{[i',j']}^+ + \sum_{\substack{i'\geq i, \; \, j' \geq j \\ [i',j']\in P_{\lambda}^{\mathrm{shift}} }}\!\mathcal{T}_{[i',j']}^- - \sum_{\substack{i'< i, \; \, j'< j, \; \, i' < j' \\ [i',j']\in P_{\lambda}^{\mathrm{shift}} }}\!\mathcal{T}_{[i',j']}^- - \sum_{\substack{i'> i, \; \, j'> j, \; \, i' < j' \\ [i',j']\in P_{\lambda}^{\mathrm{shift}} }}\!\mathcal{T}_{[i',j']}^+.\]
This formula is complicated. The left of Figure~\ref{fig:shiftrook} presents~$R_{2,4}^{\mathrm{shift}}$ for~$\lambda = (8,6,5,2,1)$ in a way that is easier to comprehend than the above formula: in this diagram we draw the coefficient of $\mathcal{T}_{ij}^{+}$ in the northwest corner of each box  $[i,j]$ and the coefficient of~$\mathcal{T}_{ij}^{-}$ in the southeast corner; if the coefficient is zero we leave the corresponding corner empty. The right of Figure~\ref{fig:shiftrook} depicts the boxes ``attacked'' in expectation by the rook~$R_{2,4}^{\mathrm{shift}}$ for a toggle-symmetric distribution. In contrast to the ordinary rooks, the shifted rook~$R_{ij}^{\mathrm{shift}}$ attacks not just the boxes in the same row and column as $[i,j]$, but also any main diagonal boxes strictly north or strictly east of $[i,j]$. That is,

\begin{lemma} \label{lem:shiftrooktogsym}
For any toggle-symmetric distribution $\mu$ on $J(P^{\mathrm{shift}}_{\lambda})$ and $[i,j] \in P^{\mathrm{shift}}_{\lambda}$,
\[\mathbb{E}(\mu;R^{\mathrm{shift}}_{ij}) = \hspace{-0.4cm} \sum_{[i',j] \in P^{\mathrm{shift}}_{\lambda}} \hspace{-0.4cm} \mathbb{E}(\mu;\mathcal{T}^{-}_{[i',j]}) + \hspace{-0.5cm} \sum_{[i,j'] \in P^{\mathrm{shift}}_{\lambda}} \hspace{-0.4cm} \mathbb{E}(\mu;\mathcal{T}^{-}_{[i,j']}) + \hspace{-0.5cm} \sum_{\substack{i' < i, \\ [i',i'] \in P^{\mathrm{shift}}_{\lambda}}} \hspace{-0.4cm} \mathbb{E}(\mu;\mathcal{T}^{-}_{[i,j']}) + \hspace{-0.5cm} \sum_{\substack{j' > j, \\ [j',j'] \in P^{\mathrm{shift}}_{\lambda}}} \hspace{-0.4cm} \mathbb{E}(\mu;\mathcal{T}^{-}_{[i,j']})  .\] 
\end{lemma}
\begin{proof}
This follows immediately from the above definition of $R^{\mathrm{shift}}_{ij}$ together with the definition (Definition~\ref{def:togsym}) of toggle-symmetry.
\end{proof}

Now we need to find the other, more subtle expression which says that for any probability distribution the expectation of $R_{ij}^{\mathrm{shift}}$ is one up to an error term involving a sum over outward corners. For this we need a bit more notation concerning outward corners of shifted Young diagrams. As we are only considering ``straight'' shifted shapes, we can only have southeast outward corners. The \emph{southeast border} of the shifted Young diagram of~$\lambda$ is the lattice path associated to the order ideal $\lambda \in J(P_{\lambda}^{\mathrm{shift}})$. An \emph{outward corner} of the shifted Young diagram of $\lambda$ is a north step followed by an east step on the southeast border. We denote the set of these outward corners by~$C^{\mathrm{shift}}(\lambda)$.  As in the ordinary Young diagram case, we say that $c \in C^{\mathrm{shift}}(\lambda)$ \emph{occurs} at~$(i,j)$ if this is the lattice point in common between the two steps of $c$. And again for an order ideal~$\nu \in J(P_{\lambda}^{\mathrm{shift}})$ we say that~$\nu$ \emph{contains} $c \in C^{\mathrm{shift}}(\lambda)$, written $c \in \nu$, if the lattice path associated to $\nu$ contains the two steps that make up~$c$. For $[i,j] \in P_{\lambda}^{\mathrm{shift}}$ we use~$C_{ij}^{\mathrm{shift}}(\lambda)$ to denote the subset of $c \in C^{\mathrm{shift}}(\lambda)$ that occur strictly southeast of the center of the box $[i,j]$. We can now state the subtler expectation expression for $R_{ij}^{\mathrm{shift}}$.

\begin{figure}
\ytableausetup{boxsize=2em}
\begin{tikzpicture} 
\node at (0,0) {\begin{ytableau}
\scalebox{1.0}[1.1]{$^{1}\,\,\,\,\,\,$} & \scalebox{1.0}[1.1]{$^{1}\,\,_{\text{-}1}$} & \scalebox{1.0}[1.1]{$^{1}\,\,_{\text{-}1}$} & \scalebox{1.0}[1.1]{$^{1}\,\,\,\,\,\,$} & & & & \\
\none &  \scalebox{1.0}[1.1]{$^{1}\,\,\,\,\,\,$} &  \scalebox{1.0}[1.1]{$^{1}\,\,\,\,\,\,$} &\scalebox{1.0}[1.1]{$^{1}\,\,\,_{1}$} & \scalebox{1.0}[1.1]{$\,\,\,\,\,\,_{1}$} & \scalebox{1.0}[1.1]{$\,\,\,\,_{1}$} & \scalebox{1.0}[1.1]{$\,\,\,\,_{1}$} \\
\none & \none & & \scalebox{1.0}[1.1]{$\,\,\,\,\,\,_{1}$} & \scalebox{1.0}[1.1]{$^{\text{-}1}\,\,_{1}$} & \scalebox{1.0}[1.1]{$^{\text{-}1}\,\,_{1}$} & \scalebox{1.0}[1.1]{$^{\text{-}1}\,\,_{1}$} \\
 \none & \none & \none & \scalebox{1.0}[1.1]{$\,\,\,\,\,\,_{1}$} & \scalebox{1.0}[1.1]{$^{\text{-}1}\,\,_{1}$}  \\
  \none & \none & \none & \none & \scalebox{1.0}[1.1]{$\,\,\,\,\,\,_{1}$}
\end{ytableau} }; 
\def\x{0.8}
\draw[blue,line width=3pt] (1*\x,-2.5*\x-0.05*\x) -- (0*\x-0.05*\x,-2.5*\x-0.05*\x) -- (0*\x-0.05*\x,0.5*\x) -- (1*\x,0.5*\x) -- (1*\x,2.5*\x) -- (4*\x,2.5*\x);
\draw[red,line width=3pt] (1*\x,-2.5*\x+0.07*\x) -- (0*\x+0.07*\x,-2.5*\x+0.07*\x) -- (0*\x+0.07*\x,-1.5*\x) -- (1*\x,-1.5*\x) -- (1*\x,-0.5*\x)  -- (2*\x,-0.5*\x) -- (2*\x,0.5*\x) -- (3*\x,0.5*\x) -- (3*\x,1.5*\x) -- (4*\x-0.05*\x,1.5*\x) -- (4*\x-0.05*\x,2.5*\x+0.07*\x);
\node at (1.5*\x,-0.8*\x) {{\footnotesize ``-1''}};
\end{tikzpicture}
\ytableausetup{boxsize=1.5em}
\caption{Examples of the key equality from the proof of Lemma~\ref{lem:shiftrookexpectation}: we depict the lattices paths associated to two different $\nu \in J(P^{\mathrm{shift}}_{\lambda})$ in blue and red; in either case, $R_{2,4}^{\mathrm{shift}}(\nu) -  \#C^{\mathrm{shift}}_{2,4}(\lambda;\nu) =1$.} \label{fig:shiftrookexpectation}
\end{figure}

\begin{lemma} \label{lem:shiftrookexpectation}
For any probability distribution $\mu$ on $J(P_{\lambda}^{\mathrm{shift}})$ and any $[i,j] \in P_{\lambda}^{\mathrm{shift}}$,
\[\mathbb{E}(\mu;R_{ij}^{\mathrm{shift}}) = 1 + \sum_{c \in C^{\mathrm{shift}}_{ij}(\lambda)} \mathbb{P}(\nu \sim \mu; c \in \nu). \]
\end{lemma}
\begin{proof}
Let $\mu$ and $[i,j]$ be as in the statement of the lemma. Let $\nu \in J(P_{\lambda}^{\mathrm{shift}})$. Let us use $C^{\mathrm{shift}}_{ij}(\lambda;\nu)$ to denote the subset of outward corners $c \in C^{\mathrm{shift}}_{ij}(\lambda)$ with $c \in \nu$. The key observation, as in the proof of Lemma~\ref{lem:rookexpectation} given in~\cite[Lemma 3.6]{chan2015expected}, is that the following equality always holds:
\begin{equation} \label{eqn:shiftrookpathequality}
R_{ij}^{\mathrm{shift}}(\nu) - \#C^{\mathrm{shift}}_{ij}(\lambda;\nu) = 1.
\end{equation}
Figure~\ref{fig:shiftrookexpectation} depicts two order ideals $\nu \in J(P_{\lambda}^{\mathrm{shift}})$ where we can check that~(\ref{eqn:shiftrookpathequality}) holds. In these examples~$\lambda = (8,6,5,2,1)$ and~$[i,j] = [2,4]$. For the order ideal $\nu \in J(P_{\lambda}^{\mathrm{shift}})$ whose associated lattice path is drawn in blue in Figure~\ref{fig:shiftrookexpectation} we see that
\[R_{2,4}^{\mathrm{shift}}(\nu)  = \mathcal{T}^{-}_{[4,4]}(\nu) - \mathcal{T}^{+}_{[3,5]}(\nu) + \mathcal{T}^{-}_{[2,5]}(\nu) = 1-1+1 = 1.\]
Observe that $T^{\pm}_{[i,j]}(\nu) = 0$ for most $T^{\pm}_{ij}$ appearing in $R_{2,4}^{\mathrm{shift}}$. Because this~$\nu$ contains no outward corners, the equality~(\ref{eqn:shiftrookpathequality}) indeed holds in this case. A more generic example is the order ideal~$\nu \in J(P_{\lambda}^{\mathrm{shift}})$ whose associated lattice path is drawn in red in Figure~\ref{fig:shiftrookexpectation}. This~$\nu$ contains the outward corner $c \in C^{\mathrm{shift}}_{2,4}(\lambda)$ that occurs at $(3,5)$. But for this $\nu$ we still have
\begin{align*}
R_{2,4}^{\mathrm{shift}}(\nu) - \#C^{\mathrm{shift}}_{2,4}(\lambda;\nu)  &= \mathcal{T}^{-}_{[4,5]}(\nu) + \mathcal{T}^{-}_{[3,6]}(\nu) - \mathcal{T}^{+}_{[3,7]}(\nu) +  \mathcal{T}^{-}_{[2,7]}(\nu)  - \#C^{\mathrm{shift}}_{2,4}(\lambda;\nu) \\
&= 1+1-1+1-1=1.
\end{align*}
In Figure~\ref{fig:shiftrookexpectation} we wrote ``$-1$'' next to the outward corner that occurs at $(3,5)$ to emphasize that there would be a component of $\mathcal{T}^{+}_{[4,6]}$ with coefficient $-1$ in  $R_{2,4}^{\mathrm{shift}}$ if $[4,6]$ were a box of the shifted Young diagram of $\lambda$. These ``missing'' $-1$'s explain why the outward corners contained by~$\nu$ contribute an error term to~(\ref{eqn:shiftrookpathequality}). A formal proof of~(\ref{eqn:shiftrookpathequality}) for general~$\nu$ would start with the observation that~(\ref{eqn:shiftrookpathequality}) holds for $\nu = \varnothing$ and proceed to show, via a tedious checking of cases, that if~(\ref{eqn:shiftrookpathequality}) holds for $\nu$ then~(\ref{eqn:shiftrookpathequality}) also holds for~$\tau_{[i,j]}(\nu)$ where~$[i,j]$ is any box that can be toggled into $\nu$. Content with the examples from Figure~\ref{fig:shiftrookexpectation}, we omit this tedious verification. That~(\ref{eqn:shiftrookpathequality}) holds for all $\nu \in J(P_{\lambda}^{\mathrm{shift}})$ is equivalent the following equality of functions $J(P_{\lambda}^{\mathrm{shift}}) \to \mathbb{R}$:
\[ R_{ij}^{\mathrm{shift}} = \mathbb{1} +  \#C^{\mathrm{shift}}_{ij}(\lambda;\cdot). \]
The lemma follows by taking the expectation with respect to $\mu$ of both sides of this equality of functions, where we use that
\[\mathbb{E}(\mu;\#C^{\mathrm{shift}}_{ij}(\lambda;\cdot)) =  \sum_{c \in C^{\mathrm{shift}}_{ij}(\lambda)} \mathbb{P}(\nu \sim \mu; c \in \nu) \]
by linearity of expectation.
\end{proof}

To complete the proof of Theorem~\ref{thm:shiftyounglatcde} we need to show that it is possible to place (positive and negative) shifted rooks on the boxes of $P_{\lambda}^{\mathrm{shift}}$ so that all boxes are attacked the same number of times and so that in aggregate the error terms from Lemma~\ref{lem:shiftrookexpectation} cancel out for every outward corner. This is possible for the shifted-balanced strict partitions $\lambda$ appearing in the statement of Theorem~\ref{thm:shiftyounglatcde}.

\begin{figure}
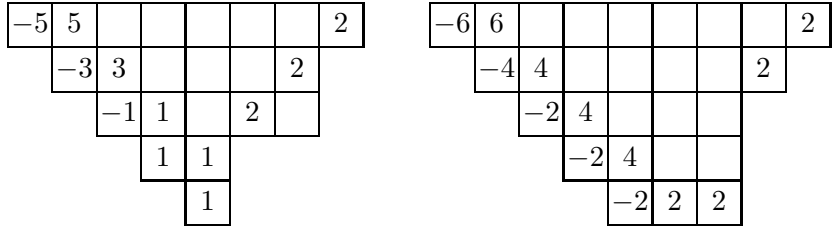

\begin{ytableau}
-5 & 5 & & & & & & 2 \\ \none & -3 & 3 & & &  & 2 \\ \none & \none & -1 & 1 & & 2 & \\ \none & \none & \none & 1 & 1 \\ \none & \none & \none & \none & 1
\end{ytableau} \qquad \begin{ytableau}
-6 & 6 & & & & & & & 2 \\ \none & -4 & 4 & & & &  & 2 \\ \none & \none & -2 & 4 & & &  \\ \none & \none & \none & -2 & 4  & &\\ \none & \none & \none & \none & -2 & 2 & 2
\end{ytableau}
\caption{Rook placements satisfying Lemma~\ref{lem:shiftrookplacement}: on the left we have a~$\lambda$ of Type~(1), on the right one of Type~(2).} \label{fig:shiftrookplacement}
\end{figure}

\begin{lemma} \label{lem:shiftrookplacement}
Suppose the strict partition $\lambda$ is shifted-balanced of Type (1) or Type (2). Then there exist integral coefficients $r_{ij} \in \mathbb{Z}$ for $[i,j] \in P_{\lambda}^{\mathrm{shift}}$ such that: 
\begin{enumerate}[label=(\alph*)]
\item for any box~$[i,j] \in P_{\lambda}^{\mathrm{shift}}$ with $i \neq j$ we have 
\begin{itemize}
\item $\sum_{[i',j] \in P_{\lambda}^{\mathrm{shift}}} r_{i',j} = 2$,
\item $\sum_{[i,j'] \in P_{\lambda}^{\mathrm{shift}}} r_{i,j'} = 2$;
\end{itemize}
\item for any box $[i,i] \in P_{\lambda}^{\mathrm{shift}}$ we have $\sum_{\substack{i'\leq i, \; \, j' \leq j \\ [i',j'] \in P_{\lambda}^{\mathrm{shift}} }} r_{i',j'} + \sum_{\substack{i'\geq i, \; \, j' \geq j \\ [i',j'] \in P_{\lambda}^{\mathrm{shift}} }} r_{i',j'} = 4$;
\item for any outward corner $c \in C^{\mathrm{shift}}(\lambda)$ we have $\sum_{\substack{ c \in C^{\mathrm{shift}}_{ij}(\lambda) \\ [i,j] \in P_{\lambda}^{\mathrm{shift}}}} r_{ij} = 0$.
\end{enumerate}
\end{lemma}
\begin{proof}
Examples of rook placements satisfying the conditions of this lemma, for $\lambda$ both of Type~(1) and of Type~(2), are given in Figure~\ref{fig:shiftrookplacement}. In this figure we draw $r_{ij}$ inside of box~$[i,j]$ unless $r_{ij} = 0$ in which case we leave the box empty. Let us explain how we arrived at these placements. Let $\lambda$ be shifted-balanced of Type~(1) or Type~(2) and let $n$, $k$ and~$\nu$ be as in the statement of Theorem~\ref{thm:shiftyounglatcde}.  The shifted Young diagram of $\lambda$ consists of a~$\delta_n$ ``section'' and a $\nu$ ``section,'' as well as a $(n-1-k)^n$ ``section'' in the case that~$\lambda$ is of Type~(2). This sectional decomposition of $\lambda$ looks as follows:
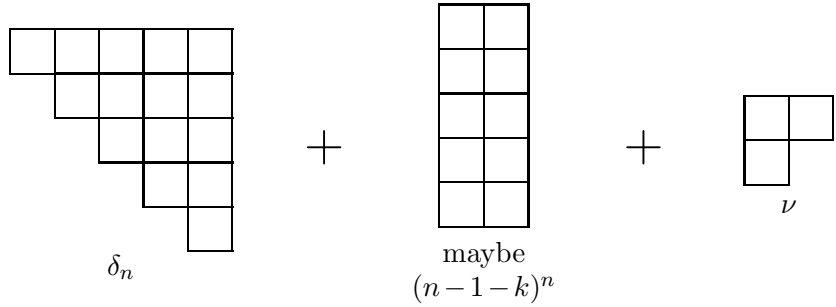
\begin{center}
\parbox{8em}{\begin{center}\ydiagram{5,1+4,2+3,3+2,4+1} \\ $\delta_n$ \end{center}}  \qquad {\huge +} \qquad \parbox{5em}{\begin{center} \ydiagram{2,2,2,2,2} \\ \medskip maybe \\ $(n-1-k)^n$\end{center}} \qquad {\huge +} \qquad \parbox{4em}{\begin{center} \ydiagram{2,1} \\ \medskip $\nu$  \end{center}}
\end{center}
We address these three sections in turn. First we describe how to place rooks in the~$\delta_n$ section. We can easily work out what $r_{1,1}$ has to be: if condition~(a) is to be satisfied then we need~$\sum_{i = 2}^{\lambda_1} \sum_{[i,j] \in P_{\lambda}^{\mathrm{shift}}} r_{ij} = \sum_{i = 2}^{\lambda_1} 2 = 2(\lambda_1-1)$; thus  to satisfy condition~(b) we have to have~$r_{1,1} = \frac{1}{2} \left(4- \sum_{i = 2}^{\lambda_1} \sum_{[i,j] \in P_{\lambda}^{\mathrm{shift}}} r_{ij} \right)=  3-\lambda_1$. We want each of the first~$k$ rows to be attacked a net zero times by rooks placed in $\delta_n$ (because these rows will be attacked two times by boxes in the $\nu$ section) whereas we want the next $n-k-1$ rows to be attacked a net of two times by rooks placed in $\delta_n$. We also want the second through $n$th columns to be attacked a net two times by rooks in~$\delta_n$. We can achieve all of this by setting~$r_{i,i} := 1+2i-\lambda_1$ and~$r_{i,i+1} := \lambda_1-1-2i$ for~$1 \leq i \leq k$, setting~$r_{i,i} := 3+2k-\lambda_1$ and~$r_{i,i+1} := \lambda_1-1-2k$ for $k+1 \leq i \leq n-1$, and setting~$r_{n,n} := 3+2k-\lambda_1$. The resulting placement of rooks looks as follows:
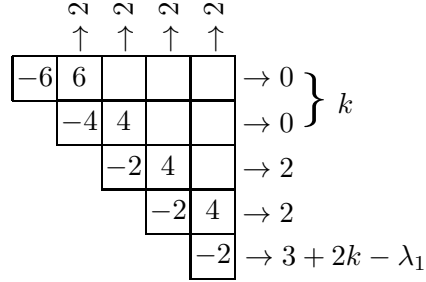
\begin{center}
\begin{tikzpicture}
\node at (0,0) {\begin{ytableau}
-6 & 6 & & & \\ \none & -4 & 4 & & \\ \none & \none & -2 & 4 & \\ \none & \none & \none & -2 & 4 \\ \none & \none & \none & \none & -2
\end{ytableau}};
\def \x{0.6}
\node at (2*\x,3.1*\x) {\rotatebox{90}{$\rightarrow 2$}};
\node at (1*\x,3.1*\x) {\rotatebox{90}{$\rightarrow 2$}};
\node at (0*\x,3.1*\x) {\rotatebox{90}{$\rightarrow 2$}};
\node at (-1*\x,3.1*\x) {\rotatebox{90}{$\rightarrow 2$}};
\node at (3.2*\x,2*\x) {$\rightarrow 0$};
\node at (3.2*\x,1*\x) {$\rightarrow 0$};
\node at (3.2*\x,0*\x) {$\rightarrow 2$};
\node at (3.2*\x,-1*\x) {$\rightarrow 2$};
\node at (4.7*\x,-2*\x) {$\rightarrow 3+2k-\lambda_1$};
\node at (4.5*\x,1.5*\x) {{\huge $\}$} $k$};
\end{tikzpicture}
\end{center}
In this picture we have drawn numbers next to arrows to indicate the net number of times the represented rooks are attacking the corresponding row or column. The next section, which applies only in the case that~$\lambda$ of Type~(2), is the~$(n-1-k)^n$ section. Here we simply put two rooks on each box in the last row; that is, we set~$r_{n,n+i} := 2$ for~$1 \leq i \leq n-1-k$. The resulting placement of rooks looks as follows:
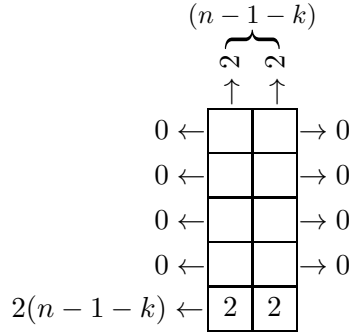
\begin{center}
\begin{tikzpicture}
\node at (0,0) {\begin{ytableau}
\; &  \\ &  \\ &  \\ &  \\ 2 & 2
\end{ytableau}};
\def \x{0.6}
\node at (0*\x,4.5*\x) {\small $(n-1-k)$};
\node at (0*\x,4.0*\x) {\rotatebox{90}{\huge $\}$}};
\node at (0.5*\x,3.1*\x) {\rotatebox{90}{$\rightarrow 2$}};
\node at (-0.5*\x,3.1*\x) {\rotatebox{90}{$\rightarrow 2$}};
\node at (1.6*\x,2*\x) {$\rightarrow 0$};
\node at (1.6*\x,1*\x) {$\rightarrow 0$};
\node at (1.6*\x,0*\x) {$\rightarrow 0$};
\node at (1.6*\x,-1*\x) {$\rightarrow 0$};
\node at (-1.6*\x,2*\x) {$0 \leftarrow$};
\node at (-1.6*\x,1*\x) {$0 \leftarrow$};
\node at (-1.6*\x,0*\x) {$0 \leftarrow$};
\node at (-1.6*\x,-1*\x) {$0 \leftarrow$};
\node at (-3.2*\x,-2*\x) {$2(n-1-k) \leftarrow$};
\end{tikzpicture}
\end{center}
Finally, we consider the~$\nu$ section. It is a simple fact that because~$\nu$ is balanced as a straight shape with height and width both equal to $k$, the $k$ boxes whose centers lie on the anti-main diagonal of the $k \times k$ rectangle containing $\nu$ must actually belong to $\nu$. In other words, the boxes~$[1,\lambda_1], [2,\lambda_1-1], \cdots, [k,\lambda_1-(k-1)]$ all belong to~$P_{\lambda}^{\mathrm{shift}}$. Let us call these the \emph{anti-main diagonal boxes}. Each box in the~$\nu$ section is in the same row as a unique anti-main diagonal box, and likewise for columns. So we place two rooks on each of these anti-main diagonal boxes; that is, we set~$r_{i,\lambda_1+1-i} := 2$ for~$1 \leq i \leq k$ . The resulting placement of rooks looks as follows:
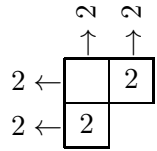
\begin{center}
\begin{tikzpicture}
\node at (0,0) {\begin{ytableau}
\; & 2  \\  2
\end{ytableau}};
\def \x{0.6}
\node at (-1.6*\x,0.5*\x) {$2 \leftarrow$};
\node at (-1.6*\x,-0.5*\x) {$2 \leftarrow$};
\node at (0.5*\x,1.6*\x) {\rotatebox{90}{$\rightarrow 2$}};
\node at (-0.5*\x,1.6*\x) {\rotatebox{90}{$\rightarrow 2$}};
\end{tikzpicture}
\end{center}
For all remaining boxes $[i,j]$ for which $r_{ij}$ has not yet been defined we set $r_{ij} := 0$.

It is easy to see from the above pictures that condition~(a) is satisfied: the only thing to check is that in the case where $\lambda$ is of Type~(2) we have $\lambda_1 = n+(n-1) = 2n-1$ so that in this case the last row is attacked a net of $2(n-1-k) + 3+2k-\lambda_1 = 2$ times. As for condition~(c), let $c \in C^{\mathrm{shift}}(\lambda)$ be an outward corner. Because $\nu$ is balanced as a straight shape, for any anti-main diagonal box $[i,j]$ we have~$c \notin C^{\mathrm{shift}}_{ij}(\lambda)$. Thus, the only boxes~$[i,j]$ with~$r_{ij} \neq 0$ that could have $c \in C^{\mathrm{shift}}_{ij}(\lambda)$ are those of the form~$[i,i]$ or~$[i,i+1]$ for some~$i \leq k$. But these boxes come in pairs $[i,i]$ and $[i,i+1]$ that exactly cancel each other out: $r_{i,i} = -r_{i,i+1}$. Finally, let us check condition~(b). For notational convenience let us set set $\rho(i) := \sum_{i'\leq i, \; \, j' \leq j} r_{i',j'} + \sum_{i'\geq i, \; \, j' \geq j} r_{i',j'}$. We want to verify that $\rho(i) = 4$ for all~$1 \leq i \leq n$. Since (a) holds, we actually already verified above that~$\rho(1) = 4$. But then note that for any $1 \leq i \leq n$ we have
\[ \rho(i) - \rho(1) = r_{i,i} - r_{1,1} - \sum_{\substack{i' \leq i -1, \\ j' \geq i+1}} r_{i',j'}= r_{i,i} - r_{1,1} - \begin{cases} 2(i-1) &\textrm{if $i \leq k$}, \\ 2k &\textrm{otherwise};\end{cases}\]
and $r_{i,i} - r_{1,1}$ precisely equals $2(i-1)$ if $i \leq k$ and $2k$ otherwise.
\end{proof}

\begin{proof}[Proof of Theorem~\ref{thm:shiftyounglatcde} from Lemmas~\ref{lem:shiftrooktogsym},~\ref{lem:shiftrookexpectation},~\ref{lem:shiftrookplacement}]

Let $\lambda$ be a shifted-balanced strict partition of Type~(1) or Type~(2) as in the statement of the theorem and let $\mu$ be a toggle-symmetric distribution on~$J(P_{\lambda}^{\mathrm{shift}})$. Let $r_{ij}$ be the coefficients guaranteed by Lemma~\ref{lem:shiftrookplacement}. The theorem follows from consideration of the statistic $\sum_{[i,j] \in P_{\lambda}^{\mathrm{shift}}} r_{ij}R^{\mathrm{shift}}_{ij}$. If $[i,j] \in P_{\lambda}^{\mathrm{shift}}$ is not a main diagonal box (i.e.,~$i \neq j$) then the ``coefficient'' of~$\mathbb{E}(\mu;\mathcal{T}^{-}_{[i,j]})$ in~$\mathbb{E}(\mu;\sum_{[i,j] \in P_{\lambda}^{\mathrm{shift}}} r_{ij}R^{\mathrm{shift}}_{ij})$ is
\[\sum_{[i',j] \in P_{\lambda}^{\mathrm{shift}}} r_{i',j}  +\sum_{[i,j'] \in P_{\lambda}^{\mathrm{shift}}} r_{i,j'} = 2+2 = 4\]
by Lemma~\ref{lem:shiftrookexpectation} and condition~(a) of Lemma~\ref{lem:shiftrookplacement}. Meanwhile if $[i,i] \in P_{\lambda}^{\mathrm{shift}}$ is a main diagonal box then the ``coefficient'' of $\mathbb{E}(\mu;\mathcal{T}^{-}_{[i,i]})$ in $\mathbb{E}(\mu;\sum_{[i,j] \in P_{\lambda}^{\mathrm{shift}}} r_{ij}R^{\mathrm{shift}}_{ij})$ is
\[\sum_{\substack{i'\leq i, \; \, j' \leq j \\ [i',j'] \in P_{\lambda}^{\mathrm{shift}} }} r_{i',j'} + \sum_{\substack{i'\geq i, \; \, j' \geq j \\ [i',j'] \in P_{\lambda}^{\mathrm{shift}} }} r_{i',j'} = 4\]
by Lemma~\ref{lem:shiftrookexpectation} and condition~(b) of Lemma~\ref{lem:shiftrookplacement}. Hence we conclude
\begin{equation} \label{eqn:shiftrooktogsym}
\mathbb{E}\left(\mu;\sum_{[i,j] \in P_{\lambda}^{\mathrm{shift}}} r_{ij}R^{\mathrm{shift}}_{ij}\right) = \sum_{[i,j] \in P_{\lambda}^{\mathrm{shift}}}4 \cdot \mathbb{E}(\mu;\mathcal{T}^{-}_{[i,j]}) = 4 \cdot \mathbb{E}(\mu;\mathrm{ddeg}).
\end{equation}
As observed in the proof of Lemma~\ref{lem:shiftrookplacement}, we have $\sum_{i = 2}^{\lambda_1} \sum_{[i,j] \in P_{\lambda}^{\mathrm{shift}}} r_{ij} = 2(\lambda_1-1)$ and~$r_{1,1} = 3-\lambda_1$. Thus
\begin{equation} \label{eqn:shiftrooktotal}
\sum_{[i,j] \in P_{\lambda}^{\mathrm{shift}}} r_{ij} =   r_{1,1} + \sum_{i = 2}^{\lambda_1} \sum_{[i,j] \in P_{\lambda}^{\mathrm{shift}}} r_{ij} =3-\lambda_1 + 2(\lambda_1-1) = \lambda_1 + 1.
\end{equation}
Lemma~\ref{lem:shiftrookexpectation} and together with~(\ref{eqn:shiftrooktotal}) and condition~(c) of Lemma~\ref{lem:shiftrookplacement} yields
\begin{align} \label{eqn:shiftrookexpectation}
\mathbb{E}\left(\mu; \hspace{-0.3cm} \sum_{[i,j] \in P_{\lambda}^{\mathrm{shift}}} \hspace{-0.4cm} r_{ij}R^{\mathrm{shift}}_{ij}\right)  &= \sum_{[i,j] \in P_{\lambda}^{\mathrm{shift}}} r_{ij} + \hspace{-0.4cm} \sum_{c \in C^{\mathrm{shift}}(\lambda)} \left(\sum_{\substack{c \in C^{\mathrm{shift}}_{ij}(\lambda), \\ [i,j] \in P_{\lambda}^{\mathrm{shift}}}}r_{ij}\right)\cdot \mathbb{P}(\nu \sim \mu;c \in \nu) \\ \nonumber
&=  \lambda_1 + 1 + \sum_{c \in C^{\mathrm{shift}}(\lambda)} 0\cdot \mathbb{P}(\nu \sim \mu;c \in \nu) \\ \nonumber
&= \lambda_1 + 1.
\end{align}
If $\lambda$ is of Type~(1) then $ \lambda_1 + 1 = n+1+k$, and if $\lambda$ is of Type~(2) then $\lambda_1 + 1 = 2n$. So the theorem follows by comparing~(\ref{eqn:shiftrooktogsym}) and~(\ref{eqn:shiftrookexpectation}).
\end{proof}

\subsection{The second shifted Young's lattice conjecture} \label{subsec:shiftyounglatconj2}

The second conjecture of Reiner-Tenner-Yong about CDE initial intervals of the shifted Young's lattice concerns strict partitions with ``trapezoidal'' shifted Young diagrams.

\begin{conj}[{Reiner-Tenner-Yong~\cite[Conjecture 2.24]{reiner2016poset}}] \label{conj:trap}
For $n \geq 1$, $0 \leq k < n/2$, let $\lambda := (n,n-2,n-4,\ldots,n-2k)$. Then $[\varnothing,\lambda]_{\mathrm{shift}}$ is CDE with edge density
\[\mathbb{E}(\mathrm{uni}_{[\varnothing,\lambda]_{\mathrm{shift}}};\mathrm{ddeg}) = \frac{|\lambda|}{n+1}\]
\end{conj}

All computational evidence suggests that ``CDE'' may be replaced by ``mCDE'' in Conjecture~\ref{conj:trap}. Note also that $P_{\lambda}^{\mathrm{shift}}$ is graded for all $\lambda$ appearing in Conjecture~\ref{conj:trap} and that~$\frac{|\lambda|}{n+1}$ agrees with the edge density that Proposition~\ref{prop:togsymedgedensity} predicts for~$[\varnothing,\lambda]_{\mathrm{shift}}$ if this distributive lattice were to be tCDE. However, we unfortunately are unable to resolve Conjecture~\ref{conj:trap} using the technology we developed above because the distributive lattices~$[\varnothing,\lambda]_{\mathrm{shift}}$ in question are \emph{not} tCDE  in general. This failure to be tCDE is somewhat surprising in light of Theorems~\ref{thm:younglatcde} and~\ref{thm:shiftyounglatcde}. We can see this failure to be tCDE already in the smallest case of Conjecture~\ref{conj:trap} not covered by Theorem~\ref{thm:younglatcde} or~\ref{thm:shiftyounglatcde}, as the next example demonstrates.

\begin{example} \label{ex:trapcounterex}
Let $\lambda := (4,2)$. So the shifted Young diagram of $\lambda$ is:
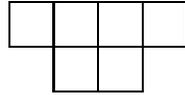
\[\ydiagram{4,1+2}\]
Define a probability distribution $\mu$ on $[\varnothing,\lambda]_{\mathrm{shift}}$ via the following table:
\begin{center} 
\renewcommand\arraystretch{1.4} 
\ytableausetup{boxsize=0.5em}
\begin{multicols}{2}
\begin{tabular}{c | c} 
$\rho \in [\varnothing,\lambda]_{\mathrm{shift}}$ & $\mathbb{P}(\mu;\rho)$ \\ \hline
$\varnothing$ & $1/11$ \\ \hline
\parbox{0.6em}{$\ydiagram{1}$} & $1/11$ \\ \hline
\parbox{1.1em}{$\ydiagram{2}$} & $1/11$ \\ \hline
\parbox{1.6em}{$\ydiagram{3}$} & $0$ \\ \hline
\parbox{1.1em}{$\ydiagram{2,1+1}$}  & $0$ \\ \hline
\end{tabular} \\ \columnbreak
\begin{tabular}{c | c} 
$\rho \in [\varnothing,\lambda]_{\mathrm{shift}}$ & $\mathbb{P}(\mu;\rho)$ \\ \hline
\parbox{2.2em}{$\ydiagram{4}$}  & $2/11$ \\ \hline
\parbox{1.6em}{$\ydiagram{3,1+1}$}  & $1/11$ \\ \hline
\parbox{2.2em}{$\ydiagram{4,1+1}$}  & $2/11$ \\ \hline
\parbox{1.6em}{$\ydiagram{3,1+2}$}  & $3/11$ \\ \hline
\parbox{2.2em}{$\ydiagram{4,1+2}$}  & $0$ \\ \hline
\end{tabular} 
\end{multicols} 
\ytableausetup{boxsize=1.5em}
\end{center}
We can verify that $\mu$ is toggle-symmetric. However,
\[\mathbb{E}(\mu;\mathrm{ddeg}) = \frac{13}{11} \neq \frac{6}{5} = \mathbb{E}(\mathrm{uni}_{[\varnothing,\lambda]_{\mathrm{shift}}};\mathrm{ddeg}) \]
and thus $[\varnothing,\lambda]_{\mathrm{shift}}$ is not tCDE. It is also easy to verify that $[\varnothing,\lambda]_{\mathrm{shift}}$ is mCDE (which anyways will follow from Proposition~\ref{prop:trap} below).
\end{example}

Example~\ref{ex:trapcounterex} shows that already in the case $k=1$ of Conjecture~\ref{conj:trap} a straightforward application of the rook approach cannot work. But we can actually prove Conjecture~\ref{conj:trap} for $k=1$ (and with ``mCDE'' instead of ``CDE'') by supplementing the rook approach with a certain trick.

\begin{prop} \label{prop:trap}
Let $\lambda := (n,n-2)$ for $n \geq 3$. Then $[\varnothing,\lambda]_{\mathrm{shift}}$ is mCDE with edge density
\[\mathbb{E}(\mathrm{uni}_{[\varnothing,\lambda]_{\mathrm{shift}}};\mathrm{ddeg}) = \frac{2(n-1)}{n+1}\]
\end{prop}

\begin{proof}
Let $\lambda := (n,n-2)$ for $n \geq 3$. We will prove that $[\varnothing,\lambda]_{\mathrm{shift}}$ has a property even stronger than mCDE, a property in fact approaching tCDE in strength. Let $\mu$ be a probability distribution on~$[\varnothing,\lambda]_{\mathrm{shift}}$ such that
\begin{itemize}
\item $\mu$ is toggle-symmetric;
\item $\mathbb{P}(\mu;\varnothing) = \mathbb{P}(\mu;\lambda)$.
\end{itemize}
We claim that in this case~$\mathbb{E}(\mu;\mathrm{ddeg}) = \frac{2(n-1)}{n+1}$. (It is easy to see that for any poset $P$ we have $\mathbb{P}(\mathrm{chain}(i)_{J(P)};\varnothing) = \mathbb{P}(\mathrm{chain}(i)_{J(P)};P)$ for all $i=0,1,\ldots,\#P$ so this claimed property is indeed stronger than mCDE.) To prove the claim we consider the following placement of rooks on the boxes of $P_{\lambda}^{\mathrm{shift}}$:
\begin{center}
\ytableausetup{boxsize=2em}
\begin{ytableau}
0 & 0 & 0 & \cdots & 0 & \frac{n-3}{n+1} & \frac{2}{n+1} \\
\none & \frac{2}{n+1} & \frac{2}{n+1} & \cdots & \frac{2}{n+1} & \frac{5-n}{n+1}
\end{ytableau}
\ytableausetup{boxsize=1.5em}
\end{center}
Let $r_{ij}$ be the coefficient written in box $[i,j]$ above. Because $\mu$ is toggle-symmetric, we can check how many times each box is attacked and deduce from Lemma~\ref{lem:shiftrooktogsym} that
\begin{align} \label{eqn:traprooktogsym}
\mathbb{E}\left(\mu;\sum_{[i,j]\in P_{\lambda}^{\mathrm{shift}}}r_{ij}R^{\mathrm{shift}}_{ij}\right) &=\sum_{[i,j]\in P_{\lambda}^{\mathrm{shift}}} \mathbb{E}(\mu;\mathcal{T}^{-}_{[i,j]}) + \frac{n-3}{n+1} \cdot \mathbb{E}(\mu;\mathcal{T}^{-}_{[1,1]}) \\ \nonumber
&= \mathbb{E}(\mu;\mathrm{ddeg}) +  \frac{n-3}{n+1}\cdot \mathbb{E}(\mu;\mathcal{T}^{+}_{[1,1]}).
\end{align}
On the other hand, denoting the unique outward corner of $P_{\lambda}^{\mathrm{shift}}$ by $c_0$, we deduce from Lemma~\ref{lem:shiftrookexpectation} that
\begin{align} \label{eqn:traprookexpectation}
\mathbb{E}\left(\mu;\sum_{[i,j]\in P_{\lambda}^{\mathrm{shift}}}r_{ij}R^{\mathrm{shift}}_{ij}\right) &= \sum_{[i,j] \in P_{\lambda}^{\mathrm{shift}}}r_{ij} + \sum_{[i,j]\in P_{\lambda}^{\mathrm{shift}}}r_{ij}\cdot\mathbb{P}(\nu \sim \mu;c_0 \in \nu) \\ \nonumber
&= \frac{2(n-1)}{n+1} + \frac{n-3}{n+1}\cdot\mathbb{P}(\nu \sim \mu;c_0 \in \nu).
\end{align}
The trick is this: we have
\begin{align*}
\mathcal{T}^{+}_{[1,1]}(\nu) = 1 &\Leftrightarrow \nu = \varnothing;\\
c_0 \in \nu &\Leftrightarrow \nu = \lambda.
\end{align*}
(The first bi-implication is true of any strict partition~$\lambda$; the second is a very particular to this $\lambda$.) Thus,
\begin{align} \label{eqn:trapemptyfullprobs}
\mathbb{E}(\mu;\mathcal{T}^{+}_{[1,1]}) &=  \mathbb{P}(\mu;\varnothing); \\ \nonumber
\mathbb{P}(\nu \sim \mu;c_0 \in \nu) &= \mathbb{P}(\mu;\lambda).
\end{align}
Combining~(\ref{eqn:traprooktogsym}),~(\ref{eqn:traprookexpectation}), and~(\ref{eqn:trapemptyfullprobs}) we obtain
\[ \mathbb{E}(\mu;\mathrm{ddeg}) +  \frac{n-3}{n+1}\cdot   \mathbb{P}(\mu;\varnothing) =  \frac{2(n-1)}{n+1} + \frac{n-3}{n+1}\cdot\mathbb{P}(\mu;\lambda).\]
Finally, we use the assumption $\mathbb{P}(\mu;\varnothing) = \mathbb{P}(\mu;\lambda)$ to conclude $\mathbb{E}(\mu;\mathrm{ddeg}) = \frac{2(n-1)}{n+1}$.
\end{proof}

The trick from Proposition~\ref{prop:trap} unfortunately does not apply to Conjecture~\ref{conj:trap} when~$k>1$ (or at least I do not see how to apply it!). I would personally be very interested in any progress on Conjecture~\ref{conj:trap}.

\begin{remark} \label{rem:shiftyounglatcde}
As with Remark~\ref{rem:younglatcde}, it is natural to ask for a converse to the results and conjectures about CDE initial intervals of the shifted Young's lattice. Computation with SAGE mathematical software tells us that for all strict partitions $\lambda$ with $|\lambda| \leq 18$, the initial interval $[\varnothing,\lambda]_{\mathrm{shift}}$ is CDE only for those $\lambda$ appearing in Theorem~\ref{thm:shiftyounglatcde} and Conjecture~\ref{conj:trap}. It is therefore reasonable to conjecture that the only strict partitions~$\lambda$ with $[\varnothing,\lambda]_{\mathrm{shift}}$ CDE are the $\lambda$ appearing in Theorem~\ref{thm:shiftyounglatcde} and Conjecture~\ref{conj:trap}. As mentioned, it is also reasonable to replace ``CDE'' with ``mCDE'' in Conjecture~\ref{conj:trap}. If all of this speculation were correct, it would mean that the mCDE property coincides with the CDE property for initial intervals of the shifted Young's lattice (although as we have seen in Example~\ref{ex:trapcounterex}, there are initial intervals that are CDE but not tCDE).
\end{remark}

\begin{remark}
Another obvious problem is to understand when arbitrary intervals of the shifted Young's lattice are CDE. Note however that if $\nu,\lambda$ are partitions (strict or otherwise) then $[\nu,\lambda] = [\delta_{\ell(\lambda)}+\nu,\delta_{\ell(\lambda)}+\lambda]_{\mathrm{shift}}$. Thus any complete description of the CDE intervals of the shifted Young's lattice has to include not just Theorem~\ref{thm:shiftyounglatcde} and Conjecture~\ref{conj:trap}, but also Theorem~\ref{thm:younglatcde}. Such an account could be quite involved.
\end{remark}

\begin{remark}
Proposition~\ref{prop:tcdedual}, which says the tCDE property is preserved under duality, adds nothing to Theorem~\ref{thm:younglatcde} because the class of intervals of Young's lattice corresponding to balanced skew shapes is closed under duality (which geometrically corresponds to $180^\circ$ rotation). But Proposition~\ref{prop:tcdedual} actually does add to Theorem~\ref{thm:shiftyounglatcde} because the class of initial intervals of the shifted Young's lattice corresponding to shifted-balanced strict partitions is not closed under duality.
\end{remark}

\section{Minuscule posets and minuscule lattices} \label{sec:minuscule}

Minuscule posets are certain posets that arise in the representation theory of semisimple Lie algebras and which enjoy many remarkable combinatorial properties (see for instance~\cite{proctor1984bruhat},~\cite{stembridge1994minuscule}, and~\cite{green2013combinatorics}). Minuscule posets are the only known examples of \emph{Gaussian posets}~\cite[\S11.3]{green2013combinatorics}; $P$ is Gaussian if the rank generating function for $m$-multichains of~$J(P)$ satisfies a certain algebraic relation. Reiner-Tenner-Yong proved~\cite[Theorem 2.10]{reiner2016poset} that if $P$ is a connected minuscule poset then $P$ is mCDE (and CDE, since all such posets are graded). They conjectured~\cite[Conjecture 2.12]{reiner2016poset} that if $P$ is an arbitrary minuscule poset then the associated distributive lattice $J(P)$ is mCDE. In this section we complete the case-by-case proof of this conjecture; in fact, we show that such~$J(P)$ are tCDE. Actually we have already done most of the work toward resolving this conjecture in our investigation of intervals of Young's lattice and the shifted Young's lattice in Sections~\ref{sec:younglat} and~\ref{sec:shiftyounglat}. We need only address one more, very simple infinite family of posets and two exceptional posets. Before saying exactly what a minuscule poset is, let us deal with this other infinite family.

\begin{figure}
\begin{center}
 \begin{tikzpicture}
	\SetFancyGraph
	\Vertex[LabelOut,Ldist=0.5,Lpos=0,x=0,y=0]{w_1}
	\Vertex[LabelOut,Ldist=0.5,Lpos=0,x=0,y=0.5]{w_2}
	\Vertex[LabelOut,Ldist=0.5,Lpos=-10,x=0,y=1.5]{w_a}
	\Vertex[LabelOut,Ldist=0.5,Lpos=180,x=-0.75,y=2]{x_1}
	\Vertex[LabelOut,Ldist=0.5,Lpos=180,x=-0.75,y=2.5]{x_2}
	\Vertex[LabelOut,Ldist=0.5,Lpos=180,x=-0.75,y=3.5]{x_b}
	\Vertex[LabelOut,Ldist=0.5,Lpos=0,x=0.75,y=2]{y_1}
	\Vertex[LabelOut,Ldist=0.5,Lpos=0,x=0.75,y=2.5]{y_2}
	\Vertex[LabelOut,Ldist=0.5,Lpos=0,x=0.75,y=3.5]{y_c}
	\Vertex[LabelOut,Ldist=0.5,Lpos=0,x=0,y=4]{z_1}
	\Vertex[LabelOut,Ldist=0.5,Lpos=0,x=0,y=4.5]{z_2}
	\Vertex[LabelOut,Ldist=0.5,Lpos=0,x=0,y=5.5]{z_d}
	\Edges[style={thick}](w_1,w_2)
	\Edges[style={thick,dashed}](w_2,w_a)
	\Edges[style={thick}](w_a,x_1)
	\Edges[style={thick}](x_1,x_2)
	\Edges[style={thick,dashed}](x_2,x_b)
	\Edges[style={thick}](w_a,y_1)
	\Edges[style={thick}](y_1,y_2)
	\Edges[style={thick,dashed}](y_2,y_c)
	\Edges[style={thick}](x_b,z_1)
	\Edges[style={thick}](y_c,z_1)
	\Edges[style={thick}](z_1,z_2)
	\Edges[style={thick,dashed}](z_2,z_d)
\end{tikzpicture}
\end{center}
\caption{The poset $P_{a,b,c,d}$.} \label{fig:posabcd}
\end{figure}
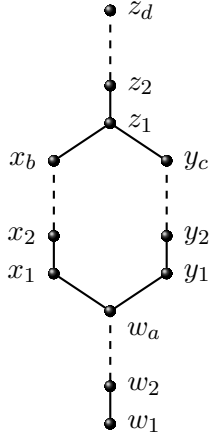

Define $P_{a,b,c,d}$ to be the poset parameterized by positive integers~$a,b,c,d$ depicted in Figure~\ref{fig:posabcd}. Reiner-Tenner-Yong~\cite[Example 2.9]{reiner2016poset} showed that $P_{a,b,c,d}$ is always CDE and mCDE. Here we are interested in the distributive lattice~$J(P_{a,b,c,d})$, which is not always CDE: for instance $J(P_{1,1,2,1})$ is already not CDE. However, we have the following proposition which is enough for our purposes.

\begin{prop} \label{prop:posabcdcde}
The distributive lattice $J(P_{a,1,1,d})$ is tCDE with edge-density one.
\end{prop}
\begin{proof}
Although it is probably overkill, we can use a ``rook''-style argument once again. So define the following four rook statistics $R_w, R_x, R_y, R_z\colon J(P_{a,1,1,d}) \to \mathbb{R}$:
\begin{align*}
R_w &:= \sum_{i=1}^{a} \mathcal{T}^{+}_{w_i} + \mathcal{T}^{-}_{w_a} + \mathcal{T}^{-}_{x_1} + \mathcal{T}^{-}_{y_1} - \mathcal{T}^{+}_{z_1} +  \sum_{i=1}^{d} \mathcal{T}^{-}_{z_i}, \\
R_x &:= \sum_{i=1}^{a} \mathcal{T}^{+}_{w_i} + \mathcal{T}^{+}_{x_1}  + \mathcal{T}^{-}_{x_1} + \sum_{i=1}^{d} \mathcal{T}^{-}_{z_i}, \\
R_y &:= \sum_{i=1}^{a} \mathcal{T}^{+}_{w_i} + \mathcal{T}^{+}_{y_1}  + \mathcal{T}^{-}_{y_1} + \sum_{i=1}^{d} \mathcal{T}^{-}_{z_i}, \\
R_z &:= \sum_{i=1}^{a} \mathcal{T}^{+}_{w_i} - \mathcal{T}^{-}_{w_a} + \mathcal{T}^{+}_{x_1} + \mathcal{T}^{+}_{y_1} + \mathcal{T}^{+}_{z_1} +  \sum_{i=1}^{d} \mathcal{T}^{-}_{z_i}. \\
\end{align*}
It is easy to check that for any order ideal~$I \in J(P_{a,1,1,d})$ we have
\[ R_w(I) = R_x(I) = R_y(I) = R_z(I) = 1\]
so that for any probability distribution $\mu$ on $J(P_{a,1,1,d})$ we have
\begin{equation} \label{eqn:posabcdtogsym}
 \mathbb{E}(\mu; R_w + R_x + R_y + R_z) = 4.
\end{equation}
On the other hand, if $\mu$ is toggle-symmetric then
\begin{align*}
\mathbb{E}(\mu;R_w) &= \mathbb{E}(\mu;\mathrm{ddeg}) + \mathbb{E}(\mu;\mathcal{T}^{-}_{w_a}) -  \mathbb{E}(\mu;\mathcal{T}^{-}_{z_1}), \\
\mathbb{E}(\mu;R_x) &= \mathbb{E}(\mu;\mathrm{ddeg}) + \mathbb{E}(\mu;\mathcal{T}^{-}_{x_1}) -  \mathbb{E}(\mu;\mathcal{T}^{-}_{y_1}), \\
\mathbb{E}(\mu;R_y) &= \mathbb{E}(\mu;\mathrm{ddeg}) + \mathbb{E}(\mu;\mathcal{T}^{-}_{y_1}) -  \mathbb{E}(\mu;\mathcal{T}^{-}_{x_1}), \\
\mathbb{E}(\mu;R_z) &= \mathbb{E}(\mu;\mathrm{ddeg}) + \mathbb{E}(\mu;\mathcal{T}^{-}_{z_1}) -  \mathbb{E}(\mu;\mathcal{T}^{-}_{w_a}),
\end{align*}
and so
\begin{equation} \label{eqn:posabcdexpectation}
 \mathbb{E}(\mu; R_w + R_x + R_y + R_z) = 4\cdot  \mathbb{E}(\mu;\mathrm{ddeg}).
\end{equation}
Comparing~(\ref{eqn:posabcdtogsym}) and~(\ref{eqn:posabcdexpectation}) shows that if $\mu$ is toggle-symmetric then $ \mathbb{E}(\mu;\mathrm{ddeg}) = 1$.
\end{proof}

Now we define what it means for a poset to be minuscule. For more details on the underlying Lie theory, consult the recent book~\cite[\S5]{green2013combinatorics}. Let $\mathfrak{g}$ be a finite-dimensional semisimple Lie algebra over the complex numbers with Weyl group $W$ and weight lattice $\Lambda$. We say that a finite-dimensional irreducible representation $V^{\lambda}$ of $\mathfrak{g}$ with highest weight $\lambda$ is \emph{minuscule} if $W$ acts transitively on the weights of $V^{\lambda}$. We call the weight $\lambda$ \emph{minuscule} in this case as well. Recall that (having chosen a set of simple roots) the weight lattice $\Lambda$ comes with a partial order whereby $\omega' \geq \omega$ iff $\omega - \omega'$ is a sum of simple roots. A \emph{minuscule lattice} is a poset of the form $W\lambda$ for $\lambda$ a minuscule weight with the partial order induced from $\Lambda$. Every such minuscule lattice is in fact a distributive lattice; that is, $W\lambda = J(P)$ for some poset $P$. The posets $P$ that arise this way are called \emph{minuscule posets}. The minuscule posets have been classified (see e.g.~\cite[Proposition 4.2]{proctor1984bruhat}). A poset is minuscule iff each of its connected components is minuscule; the connected minuscule posets are precisely the following:
\begin{enumerate}[label=(\alph*)]
\item the product $\mathbf{a} \times \mathbf{b}$ of two chain posets for $a,b \geq 1$;
\item the interval $[\varnothing,b^2]$ in Young's lattice for $b\geq 1$;
\item the special case~$P_{a,1,1,a}$ of the poset depicted in Figure~\ref{fig:posabcd} for $a \geq 1$;
\item the exceptional posets $P(E_6)$ and $P(E_7)$ depicted in Figure~\ref{fig:expectionalposets}.
\end{enumerate}

\begin{figure}
\begin{tikzpicture}[scale=0.8]
	\node[circle, draw=black, inner sep=0pt, thick, minimum size=18pt] (1) at (0.0,0.0) {\footnotesize $-4$};
	\node[circle, draw=black, inner sep=0pt, thick, minimum size=18pt] (2) at (-0.75,0.75) {\footnotesize $-5$};
	\node[circle, draw=black, inner sep=0pt, thick, minimum size=18pt] (3) at (-1.5,1.5) {\footnotesize $-6$};
	\node[circle, draw=black, inner sep=0pt, thick, minimum size=18pt] (4) at (-2.25,2.25) {\footnotesize $-4$};
	\node[circle, draw=black, inner sep=0pt, thick, minimum size=18pt] (5) at (-3,3) {\footnotesize $-2$};
	\node[circle, draw=black, inner sep=0pt, thick, minimum size=18pt] (6) at (-0.75,2.25) {\footnotesize $-3$};
	\node[circle, draw=black, inner sep=0pt, thick, minimum size=18pt] (7) at (-1.5,3) {\footnotesize $-3$};
	\node[circle, draw=black, inner sep=0pt, thick, minimum size=18pt] (8) at (-2.25,3.75) {\footnotesize $-1$};
	\node[circle, draw=black, inner sep=0pt, thick, minimum size=18pt] (9) at (-0.75,3.75) {\footnotesize $-2$};
	\node[circle, draw=black, inner sep=0pt, thick, minimum size=18pt] (10) at (-1.5,4.5) {\footnotesize $0$};
	\node[circle, draw=black, inner sep=0pt, thick, minimum size=18pt] (11) at (-2.25,5.25) {\footnotesize $0$};
	\node[circle, draw=black, inner sep=0pt, thick, minimum size=18pt] (12) at (0,4.5) {\footnotesize $-1$};
	\node[circle, draw=black, inner sep=0pt, thick, minimum size=18pt] (13) at (-0.75,5.25) {\footnotesize $1$};
	\node[circle, draw=black, inner sep=0pt, thick, minimum size=18pt] (14) at (-1.5,6) {\footnotesize $3$};
	\node[circle, draw=black, inner sep=0pt, thick, minimum size=18pt] (15) at (-2.25,6.75) {\footnotesize $2$};
	\node[circle, draw=black, inner sep=0pt, thick, minimum size=18pt] (16) at (-3,7.5) {\footnotesize $1$};
	
	\draw[thick] (1) -- (2) -- (3) -- (4) -- (5);
	\draw[thick] (3) -- (6);
	\draw[thick] (4) -- (7);
	\draw[thick] (5) -- (8);
	\draw[thick] (6) -- (7) -- (8);
	\draw[thick] (7) -- (9);
	\draw[thick] (8) -- (10);
	\draw[thick] (9) -- (10) -- (11);
	\draw[thick] (9) -- (12);
	\draw[thick] (10) -- (13);
	\draw[thick] (11) -- (14);
	\draw[thick] (12) -- (13) -- (14) -- (15) -- (16);
	
	\node at (-1.5,-0.75) {$P(E_6)$};
\end{tikzpicture} \qquad \qquad \begin{tikzpicture}[scale=0.8]
	\node[circle, draw=black, inner sep=0pt, thick, minimum size=18pt] (1) at (0.0,0.0) {\footnotesize $-3$};
	\node[circle, draw=black, inner sep=0pt, thick, minimum size=18pt] (2) at (-0.75,0.75) {\footnotesize $-4$};
	\node[circle, draw=black, inner sep=0pt, thick, minimum size=18pt] (3) at (-1.5,1.5) {\footnotesize $-5$};
	\node[circle, draw=black, inner sep=0pt, thick, minimum size=18pt] (4) at (-2.25,2.25) {\footnotesize $-6$};
	\node[circle, draw=black, inner sep=0pt, thick, minimum size=18pt] (5) at (-3,3) {\footnotesize $-4$};
	\node[circle, draw=black, inner sep=0pt, thick, minimum size=18pt] (6) at (-3.75,3.75) {\footnotesize $-2$};
	\node[circle, draw=black, inner sep=0pt, thick, minimum size=18pt] (7) at (-1.5,3) {\footnotesize $-3$};
	\node[circle, draw=black, inner sep=0pt, thick, minimum size=18pt] (8) at (-2.25,3.75) {\footnotesize $-4$};
	\node[circle, draw=black, inner sep=0pt, thick, minimum size=18pt] (9) at (-3,4.5) {\footnotesize $-2$};
	\node[circle, draw=black, inner sep=0pt, thick, minimum size=18pt] (10) at (-1.5,4.5) {\footnotesize $-3$};
	\node[circle, draw=black, inner sep=0pt, thick, minimum size=18pt] (11) at (-2.25,5.25) {\footnotesize $-2$};
	\node[circle, draw=black, inner sep=0pt, thick, minimum size=18pt] (12) at (-3,6) {\footnotesize $-1$};
	\node[circle, draw=black, inner sep=0pt, thick, minimum size=18pt] (13) at (-0.75,5.25) {\footnotesize $-2$};
	\node[circle, draw=black, inner sep=0pt, thick, minimum size=18pt] (14) at (-1.5,6) {\footnotesize $-1$};
	\node[circle, draw=black, inner sep=0pt, thick, minimum size=18pt] (15) at (-2.25,6.75) {\footnotesize $0$};
	\node[circle, draw=black, inner sep=0pt, thick, minimum size=18pt] (16) at (-3,7.5) {\footnotesize $0$};
	\node[circle, draw=black, inner sep=0pt, thick, minimum size=18pt] (17) at (-3.75,8.25) {\footnotesize $0$};
	\node[circle, draw=black, inner sep=0pt, thick, minimum size=18pt] (18) at (0,6) {\footnotesize $-1$};
	\node[circle, draw=black, inner sep=0pt, thick, minimum size=18pt] (19) at (-0.75,6.75) {\footnotesize $0$};
	\node[circle, draw=black, inner sep=0pt, thick, minimum size=18pt] (20) at (-1.5,7.5) {\footnotesize $1$};
	\node[circle, draw=black, inner sep=0pt, thick, minimum size=18pt] (21) at (-2.25,8.25) {\footnotesize $2$};
	\node[circle, draw=black, inner sep=0pt, thick, minimum size=18pt] (22) at (-3.0,9) {\footnotesize $2$};
	\node[circle, draw=black, inner sep=0pt, thick, minimum size=18pt] (23) at (-1.5,9) {\footnotesize $1$};
	\node[circle, draw=black, inner sep=0pt, thick, minimum size=18pt] (24) at (-2.25,9.75) {\footnotesize $4$};
	\node[circle, draw=black, inner sep=0pt, thick, minimum size=18pt] (25) at (-1.5,10.5) {\footnotesize $3$};
	\node[circle, draw=black, inner sep=0pt, thick, minimum size=18pt] (26) at (-0.75,11.25) {\footnotesize $2$};
	\node[circle, draw=black, inner sep=0pt, thick, minimum size=18pt] (27) at (0,12) {\footnotesize $1$};
	
	\draw[thick] (1) -- (2) -- (3) -- (4) -- (5) -- (6);
	\draw[thick] (4) -- (7);
	\draw[thick] (5) -- (8);
	\draw[thick] (6) -- (9);
	\draw[thick] (7) -- (8) -- (9);
	\draw[thick] (8) -- (10);
	\draw[thick] (9) -- (11);
	\draw[thick] (10) -- (11) -- (12);
	\draw[thick] (10) -- (13);
	\draw[thick] (11) -- (14);
	\draw[thick] (12) -- (15);
	\draw[thick] (13) -- (14) -- (15) -- (16) -- (17);
	\draw[thick] (13) -- (18);
	\draw[thick] (14) -- (19);
	\draw[thick] (15) -- (20);
	\draw[thick] (16) -- (21);
	\draw[thick] (17) -- (22);
	\draw[thick] (18) -- (19) -- (20) -- (21) -- (22);
	\draw[thick] (21) -- (23);
	\draw[thick] (22) -- (24);
	\draw[thick] (23) -- (24);
	\draw[thick] (24) -- (25) -- (26) -- (27);
	
	\node at (-1.85,-0.75) {$P(E_7)$};
\end{tikzpicture}
\caption{The exceptional minuscule posets $P(E_6)$ and $P(E_7)$ with elements labeled by coefficients appearing in the proof of Theorem~\ref{thm:minusculecde}.} \label{fig:expectionalposets}
\end{figure}
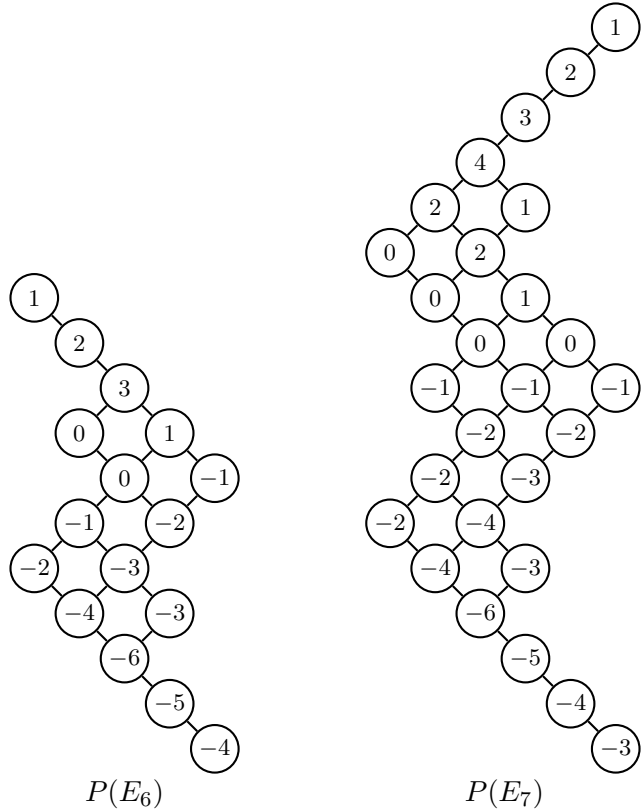

\begin{thm} \label{thm:minusculecde}
Let $J(P)$ be a minuscule lattice. Then $J(P)$ is tCDE.
\end{thm}
\begin{proof}
First note that since $J(P_1 + P_2 + \cdots + P_k) = J(P_1) \times J(P_2) \times \cdots \times J(P_2)$ we can apply Proposition~\ref{prop:tcdeproduct} to reduce to the case where $P$ is connected. 

The proof for~$P$ connected is case-by-case.
\begin{enumerate}[label=(\alph*)]
\item For $P = \mathbf{a} \times \mathbf{b}$: note that $P = P_{b^a}$ and thus Theorem~\ref{thm:younglatcde} says $J(P)$ is tCDE.
\item For $P = [\varnothing,b^2]$: note that $P = P^{\mathrm{shift}}_{\delta_{b+1}}$ and thus Theorem~\ref{thm:shiftyounglatcde} says $J(P)$ is tCDE.
\item For $P = P_{a,1,1,a}$: we showed that $J(P)$ is tCDE in Proposition~\ref{prop:posabcdcde} above.
\item For $P = P(E_6)$: let $\kappa_p$ be the coefficient labeling $p \in P(E_6)$ in Figure~\ref{fig:expectionalposets}. We have the following equality of functions $J(P(E_6)) \to \mathbb{R}$:
\begin{equation} \label{eqn:pe6}
3\cdot \mathrm{ddeg} + \sum_{p \in P(E_6)} \kappa_p \cdot \left(\mathcal{T}^{-}_p - \mathcal{T}^{+}_p \right) = 4\cdot \mathbb{1}.
\end{equation}
By taking the expectation of both sides of~(\ref{eqn:pe6}) we see that $\mathbb{E}(\mu;\mathrm{ddeg}) = \frac{4}{3}$ for any toggle-symmetric distribution~$\mu$ on $J(P(E_6))$. Similarly, for $P=P(E_7)$: let $\kappa_p$ be the coefficient labeling $p \in P(E_7)$ in Figure~\ref{fig:expectionalposets}. We have the following equality of functions $J(P(E_7)) \to \mathbb{R}$:
\begin{equation} \label{eqn:pe7}
2\cdot \mathrm{ddeg} + \sum_{p \in P(E_7)} \kappa_p \cdot \left(\mathcal{T}^{-}_p - \mathcal{T}^{+}_p \right) = 3\cdot \mathbb{1}.
\end{equation}
By taking the expectation of both sides of~(\ref{eqn:pe7}) we see that $\mathbb{E}(\mu;\mathrm{ddeg}) = \frac{3}{2}$ for any toggle-symmetric distribution~$\mu$ on $J(P(E_7))$. 
\end{enumerate}
The coefficients in Figure~\ref{fig:expectionalposets} were found using SAGE mathematical software.
\end{proof}

\begin{remark} \label{rem:uniproof}
It would be desirable to give a uniform proof of Theorem~\ref{thm:minusculecde}. Such a proof seems within reach in light of the work of Rush and Shi~\cite{rush2013orbits}, who offered a Lie theoretic interpretation of the toggles $\tau_p$ acting on a minuscule lattice~$J(P)$: roughly speaking, acting by a toggle corresponds to multiplication by a simple reflection (see~\cite[Theorem 6.3]{rush2013orbits}). Rush and Shi found this interpretation in the course of studying rowmotion orbits (see Section~\ref{sec:homomesy}) and the cyclic sieving phenomenon~\cite{reiner2004cyclic} in minuscule posets via the theory of heaps~\cite[\S2]{green2013combinatorics}. As Reiner-Tenner-Yong mentioned immediately after stating their conjecture~\cite[Conjecture 2.12]{reiner2016poset}, it would also be interesting to give a geometric interpretation or proof of the fact that any minuscule lattice~$J(P)$ is mCDE; such an interpretation seems plausible because the multichains of $J(P)$ are significant in standard monomial theory (see~\cite{seshadri1978geometry} and~\cite[\S12]{proctor1984bruhat}).
\end{remark}

One interesting consequence of the classification of minuscule posets is that if $P$ is a connected minuscule poset then $P$ itself is a distributive lattice; we have
\begin{enumerate}[label=(\alph*)]
\item $\mathbf{a} \times \mathbf{b} = J(\mathbf{c} + \mathbf{d})$ where $c = a-1$, $d=b-1$;
\item $[\varnothing,b^2] = J(P_{b^2})$;
\item $P_{a,1,1,a} = J(P_{a-1,1,1,a-1})$ for $a \geq 2$, and $P_{1,1,1,1} = \mathbf{2} \times \mathbf{2} = J(\mathbf{1} + \mathbf{1})$;
\item $P(E_6) = J([\varnothing,4^2]) = J(P_{\delta_4}^{\mathrm{shift}})$ and $P(E_7) = J(P(E_6))$.
\end{enumerate}
The following result, which generalizes the aforementioned~\cite[Theorem 2.10]{reiner2016poset}, drops out of what we have already done.

\begin{thm} \label{thm:minusculecde2}
Let $P$ be a connected minuscule poset. Then $P$ is tCDE.
\end{thm}
\begin{proof}
The proof is case-by-case.
\begin{enumerate}[label=(\alph*)]
\item For $P = \mathbf{a} \times \mathbf{b}$: since $\mathbf{a} \times \mathbf{b} = J(\mathbf{c} + \mathbf{d})$ where $c = a-1$, $d=b-1$, Example~\ref{ex:chaintcde} together with Proposition~\ref{prop:tcdeproduct} say that $P$ is tCDE.
\item For $P = [\varnothing,b^2]$: since $P = J(P_{b^2})$, Theorem~\ref{thm:younglatcde} says $P$ is tCDE.
\item For $P = P_{a,1,1,a}$: if $a=1$ we addressed this case already in the first bullet point; otherwise $P_{a,1,1,a} = J(P_{a-1,1,1,a-1})$ and so we showed that $P$ is tCDE in Proposition~\ref{prop:posabcdcde} above.
\item For $P = P(E_6)$: since $P(E_6) = J(P_{\delta_4}^{\mathrm{shift}})$, Theorem~\ref{thm:shiftyounglatcde} says $P$ is tCDE. For $P=P(E_7)$: since $P(E_7) = J(P(E_6))$, Theorem~\ref{thm:minusculecde} says $P$ is tCDE.
\end{enumerate}
\end{proof}

In contrast with the discussion in Remark~\ref{rem:uniproof} above, it is much less clear to me what a uniform proof of Theorem~\ref{thm:minusculecde2} would look like.

\section{Formulas for standard barely set-valued tableaux} \label{sec:tableaux}

In the last two sections we discuss some applications of our main results. In this section, following Reiner-Tenner-Yong~\cite[\S5]{reiner2016poset}, we explain how the study of the CDE property for intervals of Young's lattice (and the shifted Young's lattice) leads to formulas for some particular coefficients of stable Grothendieck polynomials (and their Type~B/C analogs). Recall that the Grothendieck polynomials, introduced by Lascoux and Sch\"{u}ztenberger~\cite{lascoux1982structure}~\cite{lascoux1990anneau}, represent Schubert varieties in the $K$-theory of flag manifolds in the same way that Schubert polynomials represent Schubert varieties in the cohomology of flag manifolds. Thus in particular Grothendieck polynomials are indexed by permutations $w \in \mathfrak{S}_n$ in the symmetric group. Under a natural inclusion of symmetric groups the coefficients of Grothendieck polynomials stabilize~\cite{fomin1994grothendieck}~\cite{fomin1996yangbaxter}, allowing us to define the stable Grothendieck polynomials (which are actually formal power series and not polynomials) as their limit. We will be interested in stable Grothendieck polynomials corresponding to skew shapes rather than arbitrary permutations. These arise in the $K$-theory of the Grassmannian as opposed to more general flag manifolds. Stable Grothendieck polynomials corresponding to skew shapes have a combinatorial definition, due to Buch~\cite{buch2002littlewood}, in terms of set-valued tableaux.

\begin{definition} (See Reiner-Tenner-Yong~\cite[Definition 3.2]{reiner2016poset}).
Let $\lambda/\nu$ be a skew shape. A \emph{set-valued filling $T$ of shape $\lambda/\nu$} is an assignment to each box $u$ of~$\lambda/\nu$ of a finite, nonempty subset  $T(u)\subseteq\{1,2,\ldots\}$. To any set-valued filling $T$ we associate the monomial
\[ \mathbf{x}^T := \prod_{u \in \lambda/\nu} \; \prod_{i \in T(u)} x_i. \]
We define $|T|$ to be the degree of $\mathbf{x}^{T}$. A \emph{set-valued tableau\footnote{Sometimes these tableaux are called \emph{column-strict} or \emph{semi-standard}. We will not use these adjectives as all tableaux will be assumed to be column-strict.} $T$ of shape $\lambda/\nu$} is a set-valued filling $T$ satisfying
\begin{itemize}
\item $\mathrm{max}\,T(u) \leq \mathrm{min} \, T(v)$ if $v$ is east of and in the same row as $u$;
\item $\mathrm{max}\,T(u) < \mathrm{min} \, T(v)$ if $v$ is south of and in the same column as $u$.
\end{itemize}
This $T$ is \emph{standard} if $\mathbf{x}^{T} = x_1x_2\cdots x_N$ for some $N$. We say that $T$ is simply a \emph{tableau} if~$\#T(u)=1$ for all boxes $u$ of $\lambda/\nu$, while we say that $T$ is \emph{barely set-valued} if $\#T(u)=1$ for all boxes $u$ of $\lambda/\nu$ except for a unique $u_0$ which has $\#T(u_0) = 2$.
\end{definition}

Now for the combinatorial definition of the stable Grothendieck polynomials according to Buch~\cite[Theorem 3.1]{buch2002littlewood}. Let $\lambda/\nu$ be a skew shape. The associated Schur polynomial is
\[ s_{\lambda/\nu}(x_1,x_2,\ldots) := \sum_{\substack{\textrm{$T$ a tableau} \\ \textrm{of shape $\lambda/\nu$}}} \mathbf{x}^T.\]
while the associated stable Grothendieck polynomial is
\[ G_{\lambda/\nu}(x_1,x_2,\ldots) := \sum_{\substack{\textrm{$T$ a set-valued}\\ \textrm{tableau of shape $\lambda/\nu$}}} (-1)^{|T| -|\lambda| + |\nu|}\mathbf{x}^T.\]
Observe that $s_{\lambda/\nu}(x_1,x_2,\ldots) $ is the homogenous part of $G_{\lambda/\nu}(x_1,x_2,\ldots)$ of lowest degree. Also observe that the number of standard tableaux and standard barely-set valued tableaux are encoded as coefficients of $G_{\lambda/\nu}(x_1,x_2,\ldots)$:
\begin{align*}
[x_1x_2\ldots x_{|\lambda|-|\nu|}] G_{\lambda/\nu}(x_1,x_2,\ldots) &= \textrm{$\#$ of standard tableaux of shape $\lambda/\nu$}; \\
[x_1x_2\ldots x_{|\lambda|-|\nu|}x_{|\lambda|-|\nu|+1}] G_{\lambda/\nu}(x_1,x_2,\ldots) &= (-1)\cdot\parbox{2.2in}{\begin{center}$\#$ of standard barely set-valued tableaux of shape $\lambda/\nu$\end{center}}.
\end{align*}
As is customary, we use $f^{\lambda/\nu}$ to denote the number of standard tableaux of shape $\lambda/\nu$. There is a determinantal formula for $f^{\lambda/\nu}$ due to Aitken~\cite{aitken1943monomial}.

\begin{lemma}[{Aitken~\cite{aitken1943monomial}; see also~\cite[Corollary 7.16.3]{stanley1999ec2}}]
For any skew shape $\lambda/\nu$,
\[f^{\lambda/\nu} = (|\lambda|-|\nu|)! \; \mathrm{det}_{i,j=1}^{k} \left[ \frac{1}{(\lambda_i-i-\nu_j+j)!}\right].\]
\end{lemma}

In the special case of a straight shape~$\lambda$, there is an even better formula for $f^{\lambda}$: the famous \emph{hook-length formula} of Frame-Robinson-Thrall~\cite{frame1954hook}. Here the \emph{hook-length} of a box $u \in P_{\lambda}$ is~$h_{\lambda}(u) := \#\{v \in P_{\lambda}\colon \textrm{$u\leq v$ and $v$ is in the same row or column as $u$}\}$.

\begin{lemma}[{Frame-Robinson-Thrall~\cite{frame1954hook}; see also~\cite[Corollary 7.21.6]{stanley1999ec2}}] \label{lem:hooklength}
For any straight shape $\lambda$,
\[f^{\lambda} = |\lambda|! \; \prod_{u \in P_{\lambda}} \frac{1}{h_{\lambda}(u)!}.\]
\end{lemma}

As Reiner-Tenner-Yong observed, the number of standard barely set-valued tableaux of shape $\lambda/\nu$ can be obtained from $f^{\lambda/\nu}$ via knowledge of $\mathbb{E}(\mathrm{maxchain}_{[\nu,\lambda]};\mathrm{ddeg})$. (They stated this lemma only for straight shapes but it applies equally to skew shapes and with the same proof.)

\begin{lemma}[{Reiner-Tenner-Yong~\cite[Corollary 3.7]{reiner2016poset}}] \label{lem:tableaux}
The number of standard barely set-valued tableaux of shape $\lambda/\nu$ is $(|\lambda|-|\nu|+1)f^{\lambda/\nu} \cdot \mathbb{E}(\mathrm{maxchain}_{[\nu,\lambda]};\mathrm{ddeg})$.
\end{lemma}

\begin{remark}
 Chan, L\'{o}pez Mart\'{i}n, Pflueger and Teixidor i Bigas essentially proved Lemma~\ref{lem:tableaux} as well. Specifically, a key lemma~\cite[Lemma 2.6]{chan2015genera} in their paper says that the number of edges of the \emph{Brill-Noether graph} associated to the skew shape~$\lambda/\nu$ is given by the formula appearing in Lemma~\ref{lem:tableaux}. The Brill-Noether graph~\cite[Definition~2.2]{chan2015genera} associated to $\lambda/\nu$ is the simplicial complex of all set-valued tableaux $T$ of shape~$\lambda/\nu$ which satisfy~$\mathbf{x}^{T} \mid x_{1}x_{2}\cdots x_{|\lambda|-|\nu|}x_{|\lambda|-|\nu|+1}$ and with $T \subseteq T'$ iff $T(u) \subseteq T'(u)$ for every box~$u \in \lambda/\nu$. The edges of this graph clearly correspond to standard barely set-valued tableaux of shape $\lambda/\nu$.
\end{remark}

Reiner-Tenner-Yong used Lemma~\ref{lem:tableaux} as a tool to compute $\mathbb{E}(\mathrm{maxchain}_{[\nu,\lambda]};\mathrm{ddeg})$. Our perspective will be somewhat the opposite, combining Lemma~\ref{lem:tableaux} with our prior computation of~$\mathbb{E}(\mathrm{maxchain}_{[\nu,\lambda]};\mathrm{ddeg})$ in order to deduce formulas for the number of these tableaux. That is to say, thanks to Lemma~\ref{lem:tableaux} we have the following immediate corollary of Theorem~\ref{thm:younglatcde}, i.e., corollary of the work of CHHM~\cite{chan2015expected}.

\begin{cor} \label{cor:tabeleaux}
For a balanced skew shape $\lambda/\nu$ of height $a$ and width $b$, the number of standard barely set-valued tableaux of shape $\lambda/\nu$ is $\frac{ab}{a+b}(|\lambda|-|\nu|+1) f^{\lambda/\nu}$.
\end{cor}
\begin{proof}
Theorem~\ref{thm:younglatcde} tells us that in this case $[\nu,\lambda]$ is tCDE, in particular, CDE, with edge density $\frac{ab}{a+b}$. Thus $\mathbb{E}(\mathrm{maxchain}_{[\nu,\lambda]};\mathrm{ddeg}) = \frac{ab}{a+b}$. The formula for tableaux then follows from  Lemma~\ref{lem:tableaux}.
\end{proof}

We now explain the shifted version of the preceding story. To do that we need to review the Type B/C analogs of stable Grothendieck polynomials defined by Ikeda and Naruse~\cite{ikeda2013ktheoretic}.\footnote{In fact, Ikeda and Naruse defined equivariant versions of these Type B/C stable Grothendieck polynomials, but before them no one had defined even the non-equivariant versions. We will be concerned only with the non-equivariant versions.} For this we need to talk about shifted set-valued tableaux.

\begin{definition}(Ikeda and Naruse~\cite[\S9.1]{ikeda2013ktheoretic}).
Let $\lambda$ be a strict partition. A \emph{shifted set-valued filling $T$ of shape $\lambda$} is an assignment to each box $u$ of~$P^{\mathrm{shift}}_{\lambda}$ of a finite, nonempty subset  $T(u)\subseteq\{1,1',2,2',\ldots\}$. Here the set of unprimed and primed positive integers is given the total order $1 < 1' < 2 < 2' < \cdots$. To any such set-valued filling $T$ we associate the monomial
\[ \mathbf{x}^T := \prod_{u \in P_{\lambda}^{\mathrm{shift}}} \left( \prod_{i \in T(u)} x_i \right) \left( \prod_{i' \in T(u)} x_i \right). \]
We define $|T|$ to be the degree of $\mathbf{x}^{T}$. A \emph{shifted set-valued tableau $T$ of shape $\lambda$} is a set-valued filling $T$ satisfying
\begin{itemize}
\item $\mathrm{max}\,T(u) \leq \mathrm{min} \, T(v)$ for all $u \leq v \in P_{\lambda}^{\mathrm{shift}}$;
\item each unprimed integer appears at most once in each column;
\item each primed integer appears at most once in each row.
\end{itemize}
This $T$ is \emph{standard} if $\mathbf{x}^{T} = x_1x_2\cdots x_N$ for some $N$. We say that $T$ is simply a \emph{shifted tableau} if~$\#T(u)=1$ for all boxes $u$ of $\lambda/\nu$, while we say that $T$ is \emph{barely set-valued} if $\#T(u)=1$ for all boxes $u$ of $\lambda/\nu$ except for a unique $u_0$ which has $\#T(u_0) = 2$. We say that $T$ is \emph{unprimed} if no primed integers appear in $T$, while we say that $T$ is \emph{diagonally unprimed} if no primed integers appear in main diagonal boxes in $T$.
\end{definition}

Now for the combinatorial definition of the Type B/C stable Grothendieck polynomials according to Ikeda and Naruse~\cite[Theorem 9.1]{ikeda2013ktheoretic}. Let $\lambda$ be a strict partition. The corresponding Schur $P$- and $Q$-functions are
\begin{align*}
P_{\lambda}(x_1,x_2,\ldots) := \sum_{\substack{\textrm{$T$ a diagonally unprimed}\\ \textrm{shifted tableau} \\ \textrm{of shape $\lambda$}}} \mathbf{x}^{T} \\
Q_{\lambda}(x_1,x_2,\ldots) := \sum_{\substack{\textrm{$T$ a shifted tableau}\\ \textrm{of shape $\lambda$}}} \mathbf{x}^{T} \\
\end{align*}
while the corresponding Type B/C stable Grothendieck polynomials are
\begin{align*}
GP_{\lambda}(x_1,x_2,\ldots) := \sum_{\substack{\textrm{$T$ a diagonally unprimed}\\ \textrm{shifted set-valued tableau} \\ \textrm{of shape $\lambda$}}} (-1)^{|T| -|\lambda|}  \mathbf{x}^{T} \\
GQ_{\lambda}(x_1,x_2,\ldots) := \sum_{\substack{\textrm{$T$ a shifted set-valued}\\ \textrm{tableau of shape $\lambda$}}} (-1)^{|T| -|\lambda|} \mathbf{x}^{T} \\
\end{align*}
Observe that $P_{\lambda}(x_1,x_2,\ldots)$ (respectively, $Q_{\lambda}(x_1,x_2,\ldots)$) is the homogenous component of $GP_{\lambda}(x_1,x_2,\ldots)$ (resp., $GQ_{\lambda}(x_1,x_2,\ldots)$) of lowest degree. Also observe that we have~$P_{\lambda}(x_1,x_2,\ldots) = 2^{-\ell(\lambda)} Q_{\lambda}(x_1,x_2,\ldots)$, but that there is not as simple a relationship between $GP_{\lambda}(x_1,x_2,\ldots)$ and $GQ_{\lambda}(x_1,x_2,\ldots)$.

Let us give a quick summary of the difference between the ordinary objects and their shifted analogs:
\begin{itemize}
\item the Schur functions $s_{\lambda}(x_1,x_2,\ldots)$ arise in the ordinary representation theory of the symmetric group, while the Schur $P$- and $Q$-functions arise in the projective representation theory of the symmetric group (and indeed, this is why Schur~\cite{schur1911uber} originally introduced his $P$- and $Q$-functions; see also~\cite{stembridge1989shifted});
\item the Schur functions $s_{\lambda}(x_1,x_2,\ldots)$ are related to the cohomology of the ordinary Grassmannian, while the Schur $P$- and $Q$-functions are related to the cohomology of the orthogonal and symplectic Grassmannians (as was first observed in the seminal work of Pragacz~\cite{pragacz1991algebro});
\item the stable Grothendieck polynomials $G_{\lambda}(x_1,x_2,\ldots)$ come from the $K$-theory of the ordinary Grassmannian, while $GP_{\lambda}(x_1,x_2,\ldots)$ and $GQ_{\lambda}(x_1,x_2,\ldots)$ come from the $K$-theory of the orthogonal and symplectic Grassmannians (and this is why Ikeda and Naruse introduced them).
\end{itemize}
This explains why we refer to $GP_{\lambda}(x_1,x_2,\ldots)$ and $GQ_{\lambda}(x_1,x_2,\ldots)$ as Type B/C stable Grothendieck polynomials.

Again, the number of standard shifted tableaux and standard shifted barely set-valued tableaux can be extracted from $GP_{\lambda}(x_1,x_2,\ldots)$ and $GQ_{\lambda}(x_1,x_2,\ldots)$:
\begin{align*}
[x_1x_2\ldots x_{|\lambda|}] GP_{\lambda}(x_1,x_2,\ldots) &= \parbox{2.5in}{\begin{center}$\#$ of diagonally unprimed standard \\ shifted tableaux of shape $\lambda$\end{center}}; \\
[x_1x_2\ldots x_{|\lambda|}] GQ_{\lambda}(x_1,x_2,\ldots) &= \textrm{$\#$ of standard shifted tableaux of shape $\lambda$}; \\
[x_1x_2\ldots x_{|\lambda|}x_{|\lambda|+1}] GP_{\lambda}(x_1,x_2,\ldots) &= (-1)\cdot\parbox{2.5in}{\begin{center}$\#$ of diagonally unprimed standard \\\ shifted barely set-valued tableaux \\ of shape $\lambda$\end{center}}; \\
[x_1x_2\ldots x_{|\lambda|}x_{|\lambda|+1}] GQ_{\lambda}(x_1,x_2,\ldots) &= (-1)\cdot\parbox{2.2in}{\begin{center}$\#$ of standard shifted barely set-valued tableaux of shape $\lambda$\end{center}}. \\
\end{align*}
We use $g^{\lambda}$ to denote the number of unprimed standard shifted tableaux of shape $\lambda$. Note that the number of diagonally unprimed standard shifted tableaux of shape $\lambda$ is~$2^{|\lambda|-\ell(\lambda)}g^{\lambda}$ and the number of standard shifted tableaux of shape $\lambda$ is~$2^{|\lambda|}g^{\lambda}$. Remarkably, there is also a hook-length formula for $g^{\lambda}$ due to Thrall~\cite{thrall1952combinatorial}. For a strict partition~$\lambda$ and a box $u = [i,j] \in P_{\lambda}^{\mathrm{shift}}$, we define the \emph{shifted hook-length} of $u$ to be
\[ \scalebox{0.88}{$h_{\lambda}^{*}(u) := \#\left(\{u\} \cup \{[i,j']\in P_{\lambda}^{\mathrm{shift}}\colon j' > j\}  \cup \{[i',j]\in P_{\lambda}^{\mathrm{shift}}\colon i' > i\}  \cup \{[j+1,j']\in P_{\lambda}^{\mathrm{shift}}\colon j' > j\} \right).$} \]
For an explanation as to why this quantity is the size of a ``shifted hook'', see~\cite[\S2.2]{sagan1990ubiquitous}. At any rate, we have the following.

\begin{lemma}[{Thrall~\cite{thrall1952combinatorial}; see also~\cite[Theorem 2.2.1]{sagan1990ubiquitous}}] \label{lem:shifthooklength}
For any strict partition $\lambda$,
\[g^{\lambda} = |\lambda|! \; \prod_{u \in P^{\mathrm{shift}}_{\lambda}} \frac{1}{h^{*}_{\lambda}(u)!}.\]
\end{lemma}

Now for the connection to CDE intervals of the shifted Young's lattice: we can compute the number of standard shifted barely set-valued tableaux of shape $\lambda$ from $g^{\lambda}$ via knowledge of $\mathbb{E}(\mathrm{maxchain}_{[\varnothing,\lambda]_{\mathrm{shift}}};\mathrm{ddeg})$.

\begin{lemma} \label{lem:shifttableaux}
Let $\lambda$ be a strict partition.
\begin{itemize}
\item The number of standard shifted barely set-valued tableaux of shape $\lambda$ is 
\[(|\lambda|+1) \, 2^{|\lambda|+1} g^{\lambda} \cdot \mathbb{E}(\mathrm{maxchain}_{[\varnothing,\lambda]_{\mathrm{shift}}};\mathrm{ddeg}). \]
\item The number of diagonally unprimed standard shifted barely set-valued tableaux of shape $\lambda$ is 
\[(|\lambda|+1) \, 2^{|\lambda|-\ell(\lambda)} \, g^{\lambda} \cdot \mathbb{E}\left(\mathrm{maxchain}_{[\varnothing,\lambda]_{\mathrm{shift}}};2\cdot\mathrm{ddeg}- \textstyle \sum_{i=1}^{\ell(\lambda)}\mathcal{T}^{-}_{[i,i]}\right). \]
\end{itemize}
\end{lemma}
\begin{proof}
The proof is a minor variation on the proof of Lemma~\ref{lem:tableaux} given by Reiner-Tenner-Yong in~\cite[Corollary 3.7]{reiner2016poset}. The basic observation is that an unprimed standard shifted tableau of shape $\lambda$ is the same as a maximal chain of $[\varnothing,\lambda]_{\mathrm{shiftt}}$: treating an unprimed standard shifted tableau of shape $\lambda$ as a map $T\colon P_{\lambda}^{\mathrm{shift}} \to [|\lambda|]$, we map a tableau~$T$ to the chain $\varnothing = T^{-1}(\varnothing) \subseteq T^{-1}([1]) \subseteq T^{-1}([2]) \subseteq \cdots \subseteq T^{-1}([|\lambda|]) = \lambda$, or conversely we map a chain $c = c_0 < c_1 < \cdots < c_{|\lambda|}$ to the tableau $T$ that has $T(u) = i$ if $u = c_i \setminus c_{i-1}$. (These objects are also the same as linear extensions of $P_{\lambda}^{\mathrm{shift}}$.)

Now to establish the first bullet point: let us show that the set
\[\left\{ (c,S,\nu,\rho,\iota)\colon \parbox{3in}{\begin{center}$c$ is a maximal chain of $[\varnothing,\lambda]_{\mathrm{shift}}$, $S \subseteq P_{\lambda}^{\mathrm{shift}}$, \\ $\nu = c_{i-1}$ for some $i$, $\rho \lessdot \nu$, $\iota \in \{i,i'\}$ \end{center}} \right\}. \]
in in bijection with the set of standard shifted barely set-valued tableaux of shape~$\lambda$. Let $(c,S,\nu,\rho,\iota)$ be a quintuple as above. First map $c$ to an unprimed standard shifted tableau of shape $\lambda$ as described in first paragraph above. Then raise by one the value of every entry $j$ where $j \geq i$. Then turn every unprimed entry $j$ in any box~$u$ with~$u \in S$ into a primed entry~$j'$. Finally, add an entry of $\iota$ to the box $u_0 = \nu \setminus \rho$. For example, with~$\lambda = (3,2)$, the quintuple
\[ \ytableausetup{boxsize=0.75em} (\varnothing \subseteq \ydiagram{1} \subseteq \ydiagram{2} \subseteq \ydiagram{2,1+1} \subseteq \ydiagram{3,1+1} \subseteq \ydiagram{3,1+2}, \, \begin{ytableau} *(yellow) \bullet & *(yellow) \bullet & \, \\ \none & \, & *(yellow) \bullet \end{ytableau}, \, \ydiagram{3,1+1}, \,  \ydiagram{3}, \, 5'  ) \ytableausetup{boxsize=1.5em}\]
(where the boxes of $S$ are shaded in yellow) is mapped to the tableau
\[\ytableausetup{boxsize=2em} \begin{ytableau} 
1' & 2' & 4 \\ \none & 3,5' & 6'
\end{ytableau} \ytableausetup{boxsize=1.5em}\]
It is easy to see how to invert this procedure. The claimed formula for the number of standard shifted barely set-valued tableaux of shape $\lambda$ then follows from the fact that
\[ \mathbb{E}(\mathrm{maxchain}_{[\varnothing,\lambda]_{\mathrm{shift}}}; \mathrm{ddeg}) = \frac{\#\{(c,\nu,\rho)\colon \textrm{$c$ is a max.~chain of $[\varnothing,\lambda]_{\mathrm{shift}}$}, \nu \in c, \rho \lessdot \nu \}}{\#\{(c,\nu)\colon \textrm{$c$ is a max.~chain of $[\varnothing,\lambda]_{\mathrm{shift}}$}, \nu \in c\}}.\]

The second bullet point is only slightly more involved than the first. We claim that the set
\[\left\{ (c,S,\nu,\rho,\iota)\colon \parbox{3in}{\begin{center}$c$ is a maximal chain of $[\varnothing,\lambda]_{\mathrm{shift}}$, $S \subseteq P_{\lambda}^{\mathrm{shift}}$ \\ such that $S$ contains no main diagonal boxes, \\ $\nu = c_{i-1}$ for some $i$, $\rho \lessdot \nu$, $\iota \in \{i,i'\}$ \\ with $\iota = i$ if $\nu \setminus \rho$ is a main diagonal box \end{center}} \right\}. \]
is in bijection with the set of diagonally unprimed standard shifted barely set-valued tableaux of shape~$\lambda$. Indeed, the bijection is the same as the bijection described in the second paragraph above; we just need to restrict the allowable quintuples to account for the fact that main diagonal boxes cannot have primed entries. The claimed formula for the number of diagonally unprimed standard shifted barely set-valued tableaux of shape $\lambda$ then follows from the fact that $\mathbb{E}\left(\mathrm{maxchain}_{[\varnothing,\lambda]_{\mathrm{shift}}}; 2\cdot\mathrm{ddeg}- \textstyle \sum_{i=1}^{\ell(\lambda)}\mathcal{T}^{-}_{[i,i]}\right)$ is equal to
\[ \frac{\#\left\{(c,\nu,\rho)\colon \parbox{1.7in}{\begin{center}$c$ max.~chain of $[\varnothing,\lambda]_{\mathrm{shift}}$, \\ $\nu \in c$, $\rho \lessdot \nu$ \end{center}}\right\} \hspace{-0.1cm} +  \hspace{-0.1cm} \#\left\{(c,\nu,\rho)\colon \parbox{1.7in}{\begin{center}$c$ max.~chain of $[\varnothing,\lambda]_{\mathrm{shift}}$, $\nu \in c$, $\rho \lessdot \nu$, $\nu\setminus\rho$ is not a main diagonal box \end{center}}\right\}}{\#\{(c,\nu)\colon \textrm{$c$ is a max.~chain of $[\varnothing,\lambda]_{\mathrm{shift}}$}, \nu \in c\}}.\]
\end{proof}

Recalling from Section~\ref{subsec:shiftyounglatconj1} the definition of shifted-balanced strict partitions of Type~(1) and Type~(2), Lemma~\ref{lem:shifttableaux} leads to the following corollaries of Theorem~\ref{thm:shiftyounglatcde}.

\begin{cor}  \label{cor:shifttabeleaux1}
For a strict partition $\lambda$ that is shifted-balanced of Type~(1) or Type~(2), the number of standard shifted barely set-valued tableaux of shape $\lambda$ is 
\[ (\lambda_1+1) \, (|\lambda|+1) \, 2^{|\lambda|-1} \, g^{\lambda}. \]
\end{cor}
\begin{proof}
Theorem~\ref{thm:shiftyounglatcde} tells us that in this case $[\nu,\lambda]$ is tCDE, in particular, CDE, with edge density $\frac{\lambda_1+1}{4}$. Thus $\mathbb{E}(\mathrm{maxchain}_{[\varnothing,\lambda]_{\mathrm{shift}}};\mathrm{ddeg}) = \frac{\lambda_1+1}{4}$. The formula for tableaux then follows from Lemma~\ref{lem:shifttableaux}.
\end{proof}

\begin{cor}  \label{cor:shifttabeleaux2}
For a strict partition $\lambda$ that is shifted-balanced of Type~(1), the number of diagonally unprimed standard shifted barely set-valued tableaux of shape~$\lambda$ is 
\[ \lambda_1 \, (|\lambda|+1) \, 2^{|\lambda|-\ell(\lambda)-1} \, g^{\lambda}. \]
\end{cor}
\begin{proof}
Again, Theorem~\ref{thm:shiftyounglatcde} tells us that in this case $[\nu,\lambda]$ is tCDE, in particular, CDE, with edge density $\frac{\lambda_1+1}{4}$ so that~$\mathbb{E}(\mathrm{maxchain}_{[\varnothing,\lambda]_{\mathrm{shift}}};\mathrm{ddeg}) = \frac{\lambda_1+1}{4}$. But now we have to deal with the expectation of $\sum_{i=1}^{\ell(\lambda)}\mathcal{T}^{-}_{[i,i]}$ which also appears in Lemma~\ref{lem:shifttableaux}. However, note that because $\lambda$ is shifted-balanced of Type~(1) we must have~$\lambda_{\ell(\lambda)}=1$. Thus we have the following equality of functions $J(P_{\lambda}^{\mathrm{shift}}) \to \mathbb{R}$:
\begin{equation} \label{eq:diagtoginout}
\sum_{i=1}^{\ell(\lambda)}\mathcal{T}^{+}_{[i,i]} + \sum_{i=1}^{\ell(\lambda)}\mathcal{T}^{-}_{[i,i]}= \mathbb{1}.
\end{equation}
Since $\mathrm{maxchain}_{[\varnothing,\lambda]_{\mathrm{shift}}} = \mathrm{chain}(|\lambda|)_{[\varnothing,\lambda]_{\mathrm{shift}}}$ is toggle-symmetric as a consequence of Lemma~\ref{lem:chaintogsym}, we conclude from~(\ref{eq:diagtoginout}) that $\mathbb{E}(\mathrm{maxchain}_{[\varnothing,\lambda]_{\mathrm{shift}}};\sum_{i=1}^{\ell(\lambda)}\mathcal{T}^{-}_{[i,i]}) = \frac{1}{2}$, and
\[\mathbb{E}\left(\mathrm{maxchain}_{[\varnothing,\lambda]_{\mathrm{shift}}};2\cdot\mathrm{ddeg}- \textstyle \sum_{i=1}^{\ell(\lambda)}\mathcal{T}^{-}_{[i,i]}\right) = 2 \, \frac{\lambda_1+1}{4} - \frac{1}{2} = \frac{\lambda_1}{2}.\]
The formula for tableaux then follows from Lemma~\ref{lem:shifttableaux}.
\end{proof}

\begin{remark}
For any (strict) partition $\lambda$, there are simple product formulas for the coefficient of~$x_1x_2\cdots x_{|\lambda|}$ in $G_{\lambda}(x_1,x_2,\ldots)$ (or in  $GP_{\lambda}(x_1,x_2,\ldots)$ or $GQ_{\lambda}(x_1,x_2,\ldots)$): the hook-length formulas, Lemmas~\ref{lem:hooklength} and~\ref{lem:shifthooklength}. Meanwhile, the corollaries we deduce from the study of the CDE property in intervals of Young's lattice and the shifted Young's lattice, Corollaries~\ref{cor:tabeleaux},~\ref{cor:shifttabeleaux1}, and~\ref{cor:shifttabeleaux2}, tell us that there are again product formulas for the coefficient of $x_1x_2\cdots x_{|\lambda|}x_{|\lambda|+1}$ in the stable Grothendieck polynomial (or its Type~B/C analogs) for some special choices of~$\lambda$. It therefore seems reasonable to search for similar product formulas for the coefficient of $x_1x_2\ldots x_{N}$ in~$G_{\lambda}(x_1,x_2,\ldots)$ or in~$GP_{\lambda}(x_1,x_2,\ldots)$ or~$GQ_{\lambda}(x_1,x_2,\ldots)$ for larger values of $N$, while still restricting to special~$\lambda$ such as rectangles and staircases. It might also be worth looking for formulas for particular (square-free) coefficients of other related polynomials such as the \emph{dual stable Grothendieck polynomials} of Lam and Pylyavskyy~\cite{lam2007combinatorial}.
\end{remark}

\section{Antichain cardinality homomesy for rowmotion and gyration} \label{sec:homomesy}

In this section, reiterating some comments from~\cite[\S3.2]{chan2015expected}, we explain how the tCDE property is related to the homomesy paradigm of Propp and Roby~\cite{propp2015homomesy} for certain actions on sets of order ideals. So in this section $P$ will be a fixed connected poset and we will consider a couple of invertible maps $\Phi\colon J(P) \to J(P)$ on the poset of order ideals of~$P$. 

The first of these maps is rowmotion, which was introduced by Brouwer and Schrijver~\cite{brouwer1974period} and subsequently studied by several other authors~\cite{fonderflaass1993orbits}~\cite{cameron1995orbits}~\cite{striker2012promotion}. 

\begin{definition}
\emph{Rowmotion} $\Phi_{\mathrm{row}}\colon J(P) \to J(P)$ is the map defined by
\[\Phi_{\mathrm{row}}(I) := \{x \in P\colon \textrm{$x \leq y$ for some $y \in \mathrm{min}(P\setminus I)$}\}\]
for $I \in J(P)$.
\end{definition}

It may not be obvious that rowmotion so defined is invertible. Although it is certainly possible to explicitly describe the inverse, another way to see that rowmotion is invertible is is to observe that it is the composition of three simpler bijections between the sets~$J(P)$, $J(P^*)$ and $A(P)$:
\begin{align*}
J(P) &\xrightarrow{\sim}  J(P^*) \\
I &\mapsto P\setminus I; \\
J(P^*) &\xrightarrow{\sim} A(P) \\
I &\mapsto \mathrm{min}(I); \\
A(P) &\xrightarrow{\sim}  J(P) \\
A &\mapsto \{x \in P \colon  \textrm{$x \leq y$ for some $y \in A$}\}.
\end{align*}
One interesting feature of rowmotion is its tendency to have a small order. For instance, Brouwer and Schrijver~\cite{brouwer1974period} showed that the order of $\Phi_{\mathrm{row}}$ acting on $J(\mathbf{a} \times \mathbf{b})$ is $a+b$ (which is indeed small compared to $\#J(\mathbf{a} \times \mathbf{b})=\binom{a+b}{a}$). This tendency to have small order persists even for birational analogs of rowmotion~\cite{grinberg2016iterative}~\cite{grinberg2015iterative}.

\begin{example}
Let $P = P_{(3,2,1)}^{\mathrm{shift}}$. We represent an order ideal $\nu \in J(P_{(3,2,1)}^{\mathrm{shift}})$ by shading the boxes of $\nu$ in yellow. The orbit of $\nu = \varnothing  \in J(P_{(3,2,1)}^{\mathrm{shift}})$ under $\Phi_{\mathrm{row}}$ is
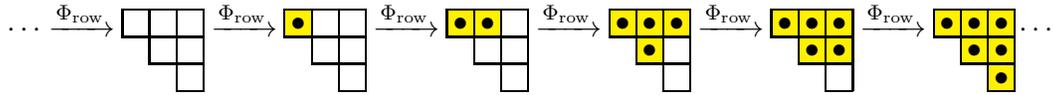
\[ \ytableausetup{boxsize=0.9em} \cdots \xrightarrow{\Phi_{\mathrm{row}}} \begin{ytableau} \, & \, & \, \\ \none & \, & \, \\ \none & \none & \, \end{ytableau} \xrightarrow{\Phi_{\mathrm{row}}} \begin{ytableau} *(yellow) \bullet & \, & \, \\ \none & \, & \, \\ \none & \none & \, \end{ytableau} \xrightarrow{\Phi_{\mathrm{row}}} \begin{ytableau} *(yellow) \bullet & *(yellow) \bullet & \, \\ \none & \, & \, \\ \none & \none & \, \end{ytableau} \xrightarrow{\Phi_{\mathrm{row}}} \begin{ytableau} *(yellow) \bullet & *(yellow) \bullet & *(yellow) \bullet \\ \none & *(yellow) \bullet & \, \\ \none & \none & \, \end{ytableau} \xrightarrow{\Phi_{\mathrm{row}}} \begin{ytableau} *(yellow) \bullet & *(yellow) \bullet & *(yellow) \bullet \\ \none & *(yellow) \bullet & *(yellow) \bullet \\ \none & \none & \, \end{ytableau} \xrightarrow{\Phi_{\mathrm{row}}} \begin{ytableau} *(yellow) \bullet & *(yellow) \bullet & *(yellow) \bullet \\ \none & *(yellow) \bullet & *(yellow) \bullet \\ \none & \none & *(yellow) \bullet \end{ytableau} \cdots \ytableausetup{boxsize=1.5em} \]
while the orbit of $\nu = (2,1) \in J(P_{(3,2,1)}^{\mathrm{shift}})$ under $\Phi_{\mathrm{row}}$ is
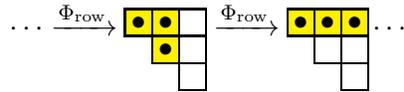
\[ \ytableausetup{boxsize=0.9em} \cdots \xrightarrow{\Phi_{\mathrm{row}}} \begin{ytableau} *(yellow) \bullet & *(yellow) \bullet & \, \\ \none & *(yellow) \bullet & \, \\ \none & \none & \, \end{ytableau} \xrightarrow{\Phi_{\mathrm{row}}} \begin{ytableau} *(yellow) \bullet & *(yellow) \bullet & *(yellow) \bullet \\ \none & \, & \, \\ \none & \none & \, \end{ytableau} \cdots \ytableausetup{boxsize=1.5em} \]
\end{example}

There is another description of rowmotion in terms of the toggle group due to Cameron and Fon-der-Flaass~\cite{fonderflaass1993orbits}. Recalling the definition of toggles and the toggle group from Section~\ref{sec:distlat},  Cameron and Fon-der-Flaass~\cite{fonderflaass1993orbits} showed that
\[\Phi_{\mathrm{row}} = \tau_{p_1} \circ \tau_{p_2} \circ \cdots \circ \tau_{p_{\#P}} \]
where $p_1,p_2,\ldots,p_{\#P}$ is any \emph{linear extension} of $P$, i.e., total ordering of the elements of $P$ compatible with its partial order in the sense that $p_i \leq p_j$ implies $i \leq j$. Note that $\tau_p$ and $\tau_q$ commute unless $p$ and $q$ are adjacent in the Hasse diagram of $P$ and therefore the above composition of toggles is in fact well-defined.

The second invertible map we will consider is gyration, which takes its name from an invertible map introduced by Wieland~\cite{wieland2000large} in his study of alternating sign matrices. Gyration of alternating sign matrices was crucially applied by Cantini and Sportiello~\cite{cantini2011proof} to resolve the Razumov-Stroganov conjecture~\cite{razumov2004combinatorial} from statistical mechanics. Striker and Williams~\cite[\S8.2]{striker2012promotion} explained how gyration acting on the set of alternating sign matrices is in equivariant bijection with some toggle group element acting on the set of order ideals of a certain poset. Striker~\cite[Definition~6.3]{striker2015toggle} subsequently extended the definition of this toggle group element to an arbitrary ranked poset.

\begin{definition}
Suppose that $P$ is ranked. Then \emph{gyration} $\Phi_{\mathrm{gyr}}\colon J(P) \to J(P)$ is the map $\Phi_{\mathrm{gyr}} := \tau_{o_1} \circ \tau_{o_2} \circ \cdots \circ \tau_{o_k} \circ \tau_{e_1} \circ \tau_{e_2} \circ \cdots \circ \tau_{e_j}$ where $\{e_1,\ldots,e_j\}$ is the set of elements of $P$ of even rank and $\{o_1,\ldots,o_k\}$ is the set of elements of odd rank. (Gyration is well-defined because, again, most toggles commute.)
\end{definition}

Rowmotion and gyration are conjugate elements of the toggle group and therefore have the same orbit structure (see~\cite[Proposition 6.4]{striker2015toggle} and~\cite[Lemma 4.1]{striker2012promotion}). On the other hand, as the following example demonstrates, it is not the case that every rowmotion orbit must also be equal as a set to some gyration orbit.

\begin{example}
Let $P = P_{(4,2)}$. The orbit of $\varnothing \in J(P_{(4,2)})$ under $\Phi_{\mathrm{row}}$ is
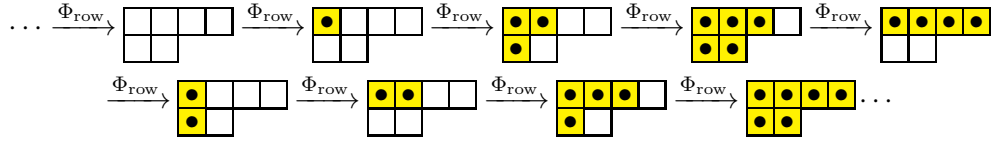
\begin{gather*} 
\ytableausetup{boxsize=0.9em} \cdots \xrightarrow{\Phi_{\mathrm{row}}} \begin{ytableau} \, & \, & \, & \, \\  \, & \, \end{ytableau} \xrightarrow{\Phi_{\mathrm{row}}} \begin{ytableau} *(yellow) \bullet & \, & \, & \, \\  \, & \, \end{ytableau} \xrightarrow{\Phi_{\mathrm{row}}} \begin{ytableau} *(yellow) \bullet & *(yellow) \bullet & \, & \, \\  *(yellow) \bullet & \, \end{ytableau} \xrightarrow{\Phi_{\mathrm{row}}} \begin{ytableau} *(yellow) \bullet & *(yellow) \bullet & *(yellow) \bullet & \, \\  *(yellow) \bullet & *(yellow) \bullet \end{ytableau}  \xrightarrow{\Phi_{\mathrm{row}}} \begin{ytableau} *(yellow) \bullet & *(yellow) \bullet & *(yellow) \bullet & *(yellow) \bullet \\ \, & \, \end{ytableau}   \\
\xrightarrow{\Phi_{\mathrm{row}}} \begin{ytableau} *(yellow) \bullet & \, & \, & \, \\  *(yellow) \bullet & \, \end{ytableau} \xrightarrow{\Phi_{\mathrm{row}}} \begin{ytableau} *(yellow) \bullet & *(yellow) \bullet & \, & \, \\  \, & \, \end{ytableau} \xrightarrow{\Phi_{\mathrm{row}}} \begin{ytableau} *(yellow) \bullet & *(yellow) \bullet & *(yellow) \bullet & \, \\  *(yellow) \bullet & \, \end{ytableau} \xrightarrow{\Phi_{\mathrm{row}}} \begin{ytableau} *(yellow) \bullet & *(yellow) \bullet & *(yellow) \bullet & *(yellow) \bullet \\  *(yellow) \bullet & *(yellow) \bullet \end{ytableau}  \cdots \ytableausetup{boxsize=1.5em}
\end{gather*}
while the orbit of $\varnothing \in J(P_{(4,2)})$ under $\Phi_{\mathrm{gyr}}$ is
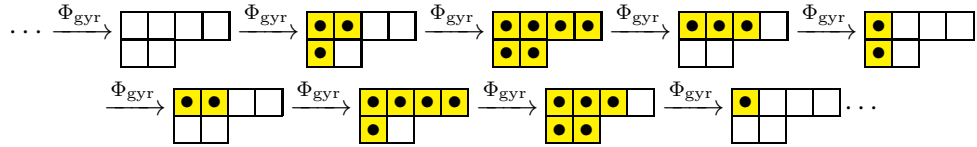
\begin{gather*} 
\ytableausetup{boxsize=0.9em} \cdots \xrightarrow{\Phi_{\mathrm{gyr}}} \begin{ytableau} \, & \, & \, & \, \\  \, & \, \end{ytableau} \xrightarrow{\Phi_{\mathrm{gyr}}} \begin{ytableau} *(yellow) \bullet & *(yellow) \bullet & \, & \, \\  *(yellow) \bullet & \, \end{ytableau} \xrightarrow{\Phi_{\mathrm{gyr}}} \begin{ytableau} *(yellow) \bullet & *(yellow) \bullet & *(yellow) \bullet & *(yellow) \bullet \\  *(yellow) \bullet & *(yellow) \bullet \end{ytableau} \xrightarrow{\Phi_{\mathrm{gyr}}} \begin{ytableau} *(yellow) \bullet & *(yellow) \bullet & *(yellow) \bullet & \, \\  \, & \, \end{ytableau}  \xrightarrow{\Phi_{\mathrm{gyr}}} \begin{ytableau} *(yellow) \bullet & \, & \, & \, \\ *(yellow) \bullet & \, \end{ytableau}   \\
\xrightarrow{\Phi_{\mathrm{gyr}}} \begin{ytableau} *(yellow) \bullet & *(yellow) \bullet & \, & \, \\ \, & \, \end{ytableau} \xrightarrow{\Phi_{\mathrm{gyr}}} \begin{ytableau} *(yellow) \bullet & *(yellow) \bullet & *(yellow) \bullet & *(yellow) \bullet \\  *(yellow) \bullet & \, \end{ytableau} \xrightarrow{\Phi_{\mathrm{gyr}}} \begin{ytableau} *(yellow) \bullet & *(yellow) \bullet & *(yellow) \bullet & \, \\  *(yellow) \bullet & *(yellow) \bullet \end{ytableau} \xrightarrow{\Phi_{\mathrm{gyr}}} \begin{ytableau} *(yellow) \bullet & \, & \, & \, \\  \, & \, \end{ytableau}  \cdots \ytableausetup{boxsize=1.5em}
\end{gather*}
\end{example}

The connection between rowmotion/gyration and the tCDE property is via the following recent result of Striker~\cite{striker2015toggle}.

\begin{lemma}[{Striker~\cite[Lemma 6.2 and Theorem 6.7]{striker2015toggle}}] \label{lem:rowmotiontogsym}
Let $\mathcal{O}$ be either a $\Phi_{\mathrm{row}}$-orbit or a $\Phi_{\mathrm{gyr}}$-orbit (in the case that $P$ is ranked). Then the distribution on $J(P)$ that is uniform on $\mathcal{O}$ and $0$ outside of $\mathcal{O}$ is toggle-symmetric.
\end{lemma}

Striker stated her results in terms of homomesy rather than toggle-symmetry. Accordingly, let us review the definition of homomesy according to Propp and Roby~\cite{propp2015homomesy}. As they mention, homomesy is a kind of cousin of the cyclic sieving phenomenon~\cite{reiner2004cyclic}.

\begin{definition}(Propp and Roby~\cite[Definition 1]{propp2015homomesy}). 
Let $\mathcal{S}$ be a finite set of combinatorial objects, let $\Phi\colon \mathcal{S} \to \mathcal{S}$ be some invertible map, and let $f\colon \mathcal{S} \to \mathbb{R}$ be any statistic. Then we say that $(\mathcal{S},\Phi,f)$ \emph{exhibits homomesy} if there is some constant $c$ such that the average of $f$ along every $\Phi$-orbit is equal to $c$. In this case we also say that $f$ is homomesic (or more specifically, $c$-mesic) with respect to the action of $\Phi$ on $\mathcal{S}$.
\end{definition}

What Striker~\cite[Lemma 6.2 and Theorem 6.7]{striker2015toggle} showed was that for every $p \in P$, the \emph{signed toggleability statistic} $\mathcal{T}_p := \mathcal{T}^{+}_p - \mathcal{T}^{-}_p$ is $0$-mesic with respect to the action of~$\Phi_{\mathrm{row}}$ and~$\Phi_{\mathrm{gyr}}$. This is clearly equivalent to Lemma~\ref{lem:rowmotiontogsym} as we stated it above.

In fact, Striker's arguments can be adapted to show something slightly stronger than Lemma~\ref{lem:rowmotiontogsym}. Suppose that $P$ is ranked with $r := \mathrm{max}_{p \in P} \{\mathrm{rk}(p)\}$. For $i = 0,1,\ldots,r$ define $\tau_i \colon J(P) \to J(P)$ to be ``toggling rank $i$,'' i.e.~to be the composition of all toggles~$\tau_p$ of elements $p \in P$ with $\mathrm{rk}(p) = i$ (which is well-defined because these all commute). Then for any permutation $\sigma = (\sigma(0),\sigma(1),\ldots,\sigma(r))$ of $0,1,\ldots,r$ define the toggle group element $\Phi_{\mathrm{row}(\sigma)}\colon J(P) \to J(P)$ by~\mbox{$\Phi_{\mathrm{row}(\sigma)} := \tau_{\sigma(0)} \circ \tau_{\sigma(1)} \circ \cdots \circ \tau_{\sigma(r)}$}. In other words, $\Phi_{\mathrm{row}(\sigma)}$ toggles each of the ranks of $P$ according to some fixed order. We thus refer to $\Phi_{\mathrm{row}(\sigma)}$ as a \emph{rank-permuted} version of rowmotion. Observe that~$\Phi_{\mathrm{row}(0,1,2,\ldots)} = \Phi_{\mathrm{row}}$ and that~$\Phi_{\mathrm{row}(1,3,5,\ldots,0,2,4,\ldots)} = \Phi_{\mathrm{gyr}}$. Also note that all the~$\Phi_{\mathrm{row}(\sigma)}$ are conjugate elements of the toggle group, as was proved by Cameron and Fon-der-Flasss~\cite[Lemma 2]{cameron1995orbits}. We have the following extension of Lemma~\ref{lem:rowmotiontogsym}.

\begin{lemma} \label{lem:gengyrmotiontogsym}
Suppose that $P$ is ranked. Let $\mathcal{O}$ be a $\Phi_{\mathrm{row}(\sigma)}$-orbit for any permutation~$\sigma$ of the ranks of $P$. Then the distribution on $J(P)$ that is uniform on $\mathcal{O}$ and $0$ outside of $\mathcal{O}$ is toggle-symmetric.
\end{lemma}
\begin{proof}
The proof adapts arguments given by Striker~\cite[Lemma 6.2 and Theorem~6.7]{striker2015toggle}. Fix some $p \in P$ and a $\Phi_{\mathrm{row}(\sigma)}$-orbit $\mathcal{O}$. It suffices to show that the signed toggleability statistic $\mathcal{T}_p$ alternates in sign as we traverse the orbit $\mathcal{O}$. With respect to~$p$, an order ideal $I \in J(P)$ is always in exactly one of four ``cases'':
\begin{itemize}
\item Case 1: $\mathcal{T}_p(I) = -1$;
\item Case 2: $\mathcal{T}_p(I) = 1$;
\item Case 3: $\mathcal{T}_p(I) = 0$ and $p \in I$;
\item Case 4: $\mathcal{T}_p(I) = 0$ and $p \notin I$.
\end{itemize}
Striker's key observation is that given that $I$ in some case, the cases that $\Phi_{\mathrm{row}(\sigma)}(I)$ could potentially be in are rather restricted. Let $i :=\mathrm{rk}(p)$ and suppose for the moment that~$i \neq 0$ and~$i \neq \mathrm{max}_{p \in P}\{\mathrm{rk}(p)\}$. How $\Phi_{\mathrm{row}(\sigma)}$ affects the various cases depends on the relative order of $\sigma^{-1}(i)$ with respect to $\sigma^{-1}(i-1)$ and $\sigma^{-1}(i+1)$. Consider first the situation where~$\sigma^{-1}(i-1) < \sigma^{-1}(i) < \sigma^{-1}(i+1)$. Then the way the cases of order ideals evolve under $\Phi_{\mathrm{row}(\sigma)}$ looks like:
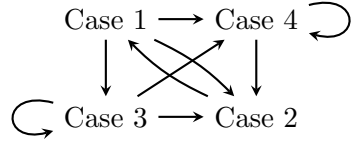
\begin{center}
 \begin{tikzpicture}[thick, ->, >=stealth]
 	\node (1) at (0,0) {Case 1};
	\node (2) at (2,-1.3) {Case 2};
	\node (3) at (0,-1.3) {Case 3};
	\node (4) at (2,0) {Case 4};
	\draw (1) [bend left=10] to  (2);
	\draw (2) [bend left=10] to  (1);
	\draw (1) to (3);
	\draw (1) to (4);
	\draw (3) to (2);
	\draw (4) to (2);
	\draw (3) to (4);
	\draw (3) [out=165,in=195,looseness=5] to (3);
	\draw (4) [out=15,in=-15,looseness=5] to (4);
\end{tikzpicture}
\end{center}
In other words, we could have that $I$ is in Case 1 and $\Phi_{\mathrm{row}(\sigma)}(I)$ in Case 4, but cannot have that $I$ is in Case 4 and $\Phi_{\mathrm{row}(\sigma)}(I)$ in Case 1, and so on. For more details consult the proofs of~\cite[Lemma 6.2 and Theorem~6.7]{striker2015toggle}. The crucial property of this diagram is that to travel from the node labeled Case 1 back to itself we have to pass through Case 2 at some point, and vice-versa, which shows that the sign of $\mathcal{T}_p$ indeed alternates along $\mathcal{O}$. In the situation $\sigma^{-1}(i-1) > \sigma^{-1}(i) > \sigma^{-1}(i+1)$ the way the cases evolve under $\Phi_{\mathrm{row}(\sigma)}$ looks like the previous diagram but with the arrows reversed. So now consider the situation $\sigma^{-1}(i-1) < \sigma^{-1}(i)$ and $\sigma^{-1}(i+1) < \sigma^{-1}(i)$. In this situation the way the cases evolve under $\Phi_{\mathrm{row}(\sigma)}$ looks like:
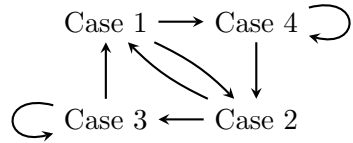
\begin{center}
 \begin{tikzpicture}[thick, ->, >=stealth]
 	\node (1) at (0,0) {Case 1};
	\node (2) at (2,-1.3) {Case 2};
	\node (3) at (0,-1.3) {Case 3};
	\node (4) at (2,0) {Case 4};
	\draw (1) [bend left=10] to  (2);
	\draw (2) [bend left=10] to  (1);
	\draw (3) to (1);
	\draw (1) to (4);
	\draw (2) to (3);
	\draw (4) to (2);
	\draw (3) [out=165,in=195,looseness=5] to (3);
	\draw (4) [out=15,in=-15,looseness=5] to (4);
\end{tikzpicture}
\end{center}
In the situation $\sigma^{-1}(i-1) > \sigma^{-1}(i)$ and $\sigma^{-1}(i+1) > \sigma^{-1}(i)$ the way the cases evolve under $\Phi_{\mathrm{row}(\sigma)}$ looks like the previous diagram but with the arrows reversed. If $p$ is of minimal or maximal rank the argument is basically the same.
\end{proof}

\begin{remark}
Rowmotion~$\Phi_{\mathrm{row}}$ and the rank-permuted versions of rowmotion~$\Phi_{\mathrm{row}(\sigma)}$ are examples of \emph{Coxeter elements} in the toggle group of~$P$, that is, they are compositions of all the toggles~$\tau_p$ for $p \in P$ in some order. It is quite common when looking for homomesies to consider Coxeter elements;  see for instance the recent paper~\cite{einstein2015noncrossing}. In light of Lemmas~\ref{lem:rowmotiontogsym} and~\ref{lem:gengyrmotiontogsym}, one might wonder whether the uniform distribution on the orbit of any Coxeter element is toggle-symmetric. However, this is not the case. Indeed, Striker~\cite[\S6]{striker2015toggle} offered the following very simple example of $\mathcal{T}_p$ failing to be $0$-mesic with respect to the action of some Coxeter element. Let~$P$ be the following poset:
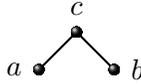
\begin{center}
 \begin{tikzpicture}
	\SetFancyGraph
	\Vertex[LabelOut,Ldist=0.5,Lpos=90,x=0,y=0]{c}
	\Vertex[LabelOut,Ldist=0.5,Lpos=180,x=-0.5,y=-0.5]{a}
	\Vertex[LabelOut,Ldist=0.5,Lpos=0,x=0.5,y=-0.5]{b}
	\Edges[style={thick}](a,c)
	\Edges[style={thick}](b,c)
\end{tikzpicture}
\end{center}
Set $\Phi := \tau_b \circ \tau_c \circ \tau_a$ (which is ``promotion'' in the sense of Stiker and Williams~\cite{striker2012promotion}). Then~$\mathcal{T}_c$ is not $0$-mesic with respect to $\Phi$, because, for instance, one $\Phi$-orbit is $\{\varnothing, \{a,b\}\}$.
\end{remark}

One of the major examples~\cite[Theorem 27]{propp2015homomesy} of homomesy exhibited by Propp and Roby is the following: we take $\mathcal{S} := J(\mathbf{a} \times \mathbf{b})$ to be the set of order ideals of the product of two chain posets, we take $\Phi := \Phi_{\mathrm{row}}$ to be rowmotion, and we take $f(I) := \#\mathrm{max}(I)$ to be the \emph{antichain cardinality statistic}. (We give $f$ this name because it is equal to the cardinality of the antichain $\mathrm{max}(I)$ associated to the order ideal~$I$.) Note that for any~$I \in J(P)$ we have~$\#\mathrm{max}(I) = \mathrm{ddeg}(I)$. Thus the tCDE property, together with Lemma~\ref{lem:rowmotiontogsym}, leads to many variations on this homomesy example of Propp and Roby.

\begin{cor} \label{cor:homomesy}
Let $P$ be a poset such that $J(P)$ is tCDE with edge density $c$. Then the antichain cardinality statistic is $c$-mesic with respect to the action of $\Phi_{\mathrm{row}}$ on $J(P)$. If~$P$ is ranked, the antichain cardinality statistic is also $c$-mesic with respect to the action of $\Phi_{\mathrm{row}(\sigma)}$ on $J(P)$ for any permutation~$\sigma$ of the ranks of $P$.
\end{cor}
\begin{proof}
Let $\mathcal{O}$ be a $\Phi_{\mathrm{row}}$-orbit or a $\Phi_{\mathrm{row}(\sigma)}$-orbit (in the case that $P$ is ranked). Let~$\mu_{\mathcal{O}}$ denote the the distribution on $J(P)$ that is uniform on $\mathcal{O}$ and $0$ outside of $\mathcal{O}$. By Lemmas~\ref{lem:rowmotiontogsym} and~\ref{lem:gengyrmotiontogsym}, $\mu_{\mathcal{O}}$ is toggle-symmetric. Thus by the assumption that $J(P)$ is tCDE with edge density $c$, we have $\mathbb{E}(\mu_{\mathcal{O}};\mathrm{ddeg}) = c$. But $\mathbb{E}(\mu_{\mathcal{O}};\mathrm{ddeg})$ is precisely the average of the antichain cardinality statistic along~$\mathcal{O}$.
\end{proof}

The results of Sections~\ref{sec:younglat},~\ref{sec:shiftyounglat}, and~\ref{sec:minuscule} together with Corollary~\ref{cor:homomesy} yield the following.

\begin{cor}[{CHMM~\cite[Corollary 3.11]{chan2015expected}}]
Let $\lambda/\nu$ be a balanced skew shape of height $a$ and width $b$. Then the  antichain cardinality statistic is $\frac{ab}{a+b}$-mesic with respect to the action of $\Phi_{\mathrm{row}(\sigma)}$ on $J(P_{\lambda/\nu})$ for any~$\sigma$.
\end{cor}

\begin{cor}
Let $\lambda$ be a strict partition that is shifted-balanced of Type~(1) or Type~(2). Then the antichain cardinality statistic is $\frac{\lambda_1+1}{4}$-mesic with respect to the action of $\Phi_{\mathrm{row}(\sigma)}$ on $J(P_{\lambda}^{\mathrm{shift}})$ for any~$\sigma$.
\end{cor}

\begin{cor} \label{cor:minusculehomo}
Let $P$ be a connected minuscule post. Then the antichain cardinality statistic is homomesic with respect to the action of $\Phi_{\mathrm{row}(\sigma)}$ on $J(P)$ for any~$\sigma$.
\end{cor}

We remark that the part of this last corollary concerning ordinary rowmotion $\Phi_{\mathrm{row}}$ recovers a recent result of Rush and Wang~\cite[Theorem 1.4]{rush2015orbits}. Extending the earlier work of Rush and Shi~\cite{rush2013orbits}, Rush and Wang proved uniformly that the antichain cardinality statistic is homomesic with respect to the action of rowmotion on a minuscule lattice. As discussed in Remark~\ref{rem:uniproof}, it might be possible to use technology developed by Rush and his coauthors~\cite{rush2013orbits}~\cite{rush2015orbits} to prove uniformly that a minuscule lattice $J(P)$ is tCDE. A uniform proof that minuscule lattice are tCDE would immediately yield a uniform proof of Corollary~\ref{cor:minusculehomo}.

\bibliography{cde}{}
\bibliographystyle{plain}

\end{document}